\documentclass{amsart}

\usepackage{amssymb}
\usepackage{amsmath}
\usepackage{amsfonts}
\usepackage{tikz-cd}
\usepackage{amsthm}
\usepackage{enumerate}
\usepackage{mathrsfs}
\usepackage{IEEEtrantools}
\usepackage{todonotes}
\usetikzlibrary{decorations.pathmorphing}
\usepackage[utf8]{inputenc}

\DeclareMathOperator{\HOM}{hom}
\DeclareMathOperator{\Hom}{Hom}
\newcommand{\homu}{\HOM_{\textbf{u}}}
\newcommand{\D}{\mathscr{D}}

\DeclareMathOperator{\Der}{Der}
\DeclareMathOperator{\Log}{log}
\DeclareMathOperator{\ann}{ann}
\DeclareMathOperator{\SP}{Sp}
\DeclareMathOperator{\Sp}{\SP^{\bullet}}
\DeclareMathOperator{\gr}{gr}
\DeclareMathOperator{\In}{in}
\DeclareMathOperator{\Spec}{Spec}
\DeclareMathOperator{\Jac}{Jac}
\DeclareMathOperator{\V}{V}
\DeclareMathOperator{\x}{\mathfrak{x}}
\DeclareMathOperator{\y}{\mathfrak{y}}

\newtheorem{theorem}{Theorem}[section]
\newtheorem{lemma}[theorem]{Lemma}

\newtheorem{cor}[theorem]{Corollary}
\newtheorem{prop}[theorem]{Proposition}
\newtheorem{conjecture}[theorem]{Conjecture}

\theoremstyle{definition}
\newtheorem{example}[theorem]{Example}
\newtheorem{definition}[theorem]{Definition}
\newtheorem{Convention}[theorem]{Convention}

\theoremstyle{remark}
\newtheorem{remark}[theorem]{Remark}

\numberwithin{equation}{section}

\begin{document}
\-

% \title[short text for running head]{full title}
\title{Bernstein--Sato varieties and annihilation of powers}

%    Only \author and \address are required; other information is
%    optional.  Remove any unused author tags.

%    author one information
% \author[short version for running head]{name for top of paper}
\author{Daniel Bath}
\address{Department of Mathematics, Purdue University, West Lafayette, IN, USA.}
%\curraddr{}
\email{dbath@purdue.edu}
\thanks{This work was in part supported by the NSF through grant DMS-1401392 and by the Simons Foundation Collaboration Grant for Mathematicians \#580839.}

%    \subjclass is required.

%\date{}

%\dedicatory{}

%    Abstract is required.
\begin{abstract}
    Given a complex germ $f$ near the point $\x$ of the complex manifold $X$, equipped with a factorization $f = f_{1} \cdots f_{r}$, we consider the $\D_{X,\x}[s_{1}, \dots, s_{r}]$-module generated by $ F^{S} := f_{1}^{s_{1}} \cdots f_{r}^{s_{r}}$. We show for a large class of germs that the annihilator of $F^{S}$ is generated by derivations and this property does not depend on the chosen factorization of $f$. 
    
    We further study the relationship between the Bernstein--Sato variety attached to $F$ and the cohomology support loci of $f$, via the $\D_{X,\x}$-map $\nabla_{A}$. This is related to multiplication by $f$ on certain quotient modules. We show that for our class of divisors the injectivity of $\nabla_{A}$ implies its surjectivity. Restricting to reduced, free divisors, we also show the reverse, using the theory of Lie--Rinehart algebras. In particular, we analyze the dual of $\nabla_{A}$ using techniques pioneered by Narv\'aez--Macarro.
    
    As an application of our results we establish a conjecture of Budur in the tame case: if $\V(f)$ is a central, essential, indecomposable, and tame hyperplane arrangement, then the Bernstein--Sato variety associated to $F$ contains a certain hyperplane. By the work of Budur, this verifies the Topological Mulivariable Strong Monodromy Conjecture for tame arrangements. Finally, in the reduced and free case, we characterize local systems outside the cohomology support loci of $f$ near $\x$ in terms of the simplicity of modules derived from $F^{S}.$
\end{abstract}

\maketitle

\setcounter{tocdepth}{1}
\tableofcontents

\section{Introduction}

Let $X$ be a smooth analytic space or $\mathbb{C}$-scheme of dimension $n$ with structure sheaf $\mathscr{O}_{X}$ and with the sheaf of $\mathbb{C}$-linear differential operators $\D_{X}$. Take a global function $f \in \mathscr{O}_{X}$. The classical construction of the Bernstein--Sato polynomial of $f$ is as follows:

\begin{enumerate}
    \item Consider the $\mathscr{O}_{X}[f^{-1}, s]$-module generated by the symbol $f^{s}.$ This has a $\D_{X}[s]$-module structure induced by the formal rules of calculus.
    \item  The Bernstein--Sato ideal $B_{f}$ of $f$ is
     \[ 
     B_{f} := \mathbb{C}[s] \cap \left(\D_{X}[s] \cdot f + \ann_{\D_{X}[s]} f^{s} \right).
     \]
    \item For $X = \mathbb{C}^{n}$ and $f$ a polynomial, Bernstein showed in \cite{Bernstein} that $B_{f}$ is not zero. For $f$ local and analytic, Kashiwara \cite{Kashiwara-local-existence} proved the same. Since $B_{f}$, or the local version $B_{f, \x}$, is an ideal in $\mathbb{C}[s]$ it has a monic generator, the Bernstein--Sato polynomial of $f$. 
    
\end{enumerate}

The variety $\V(B_{f})$ contains a lot of information about the divisor of $f$ and its singularities. For example, if $\text{Exp}(a) = e^{2\pi i a}$ and if $M_{f,\y}$ is the Milnor Fiber of $f$ at $\y \in \V (f)$, cf. \cite{Milnor}, then Malgrange and Kashiwara showed in \cite{Malgrange}, \cite{Kashiwara} that
\[
\text{Exp}(\V (B_{f,\x})) = \bigcup\limits_{\y \in \V (f) \text{ near } \x} \{ \text{ eigenvalues of the algebraic monodromy on } M_{f,\y} \}
\] 

Suppose $f$ factors as $ f_{1} \cdots f_{r}$. Let $F = (f_{1}, \dots, f_{r}).$ Then there is a generalization of the Bernstein--Sato ideal $B_{f}$ of $f$ called the multivariate Bernstein--Sato ideal $B_{F}$ of $F$ obtained in a similar way.

\begin{enumerate}
    \item Introduce new variables $S := s_{1}, \dots, s_{r}$. Consider the $\mathscr{O}_{X}[F^{-1}, S]$-module generated by the symbol $F^{S} = \prod f_{k}^{s_{k}}.$ Again, this is a $\D_{X}[S]$-module via formal differentiation.
    \item The multivariate Bernstein--Sato ideal $B_{F}$ is
    \[ B_{F} := \mathbb{C}[S] \cap \left( \D_{X}[S] \cdot f + \ann_{\D_{X}[S]}F^{S} \right). \]
    \item For $X = C^{n}$ and $f_{1}, \dots, f_{r}$ polynomials, $B_{F}$ is nonzero, see \cite{Lichtin}. Sabbah proved in \cite{Sabbah} the corresponding statement for $f_{1}, \dots, f_{r}$ local and analytic. However neither $B_{F}$ nor $B_{F,\x}$ need be principal: cf. Bahloul and Oaku \cite{BahloulOaku}. 
\end{enumerate}

The significance of $\V(B_{F})$ or the local version $\V(B_{F,\x})$ is less developed than the univariate counterparts. 
\iffalse In showing $\text{Exp}(\V(B_{f,\x}))$ are the eigenvalues of the algebraic monodromy of nearby Milnor Fibers, one first realizes $\text{Exp}(\V(B_{f,\x}))$ as the support of a complex coming from the nearby cycle functor. In \cite{BudurBernstein--Sato} Budur generalizes this: for $F = (f_{1}, \dots, f_{r}=)$ mutually distinct and irreducible, he identifies $\text{Exp}(\V(B_{F,\x}))$ as the support of what is called the Sabbah Specialization complex.
\fi
Let $f = f_{1} \cdots f_{r}$ be a product of distinct and irreducible germs at $\x$ and let $F = (f_{1}, \dots, f_{r})$. Let $U_{F,\y}$ be the intersection of a small ball about $\y \in \V(f)$ with $X \setminus \V(f)$. Denote by $V(U_{F,\y})$ the rank one local systems on $U_{F,\y}$ with nontrivial cohomology, i.e. the set of rank one local systems $L$ such that $H^{k}(U_{F, \y}, L)$ is nonzero for some $k$. This is the cohomology support locus of $f$ at $\y$ in the language of Budur and others. Since local systems can be identified with representations $\pi_{1}(U_{F,\y}) \to \mathbb{C}^{\star}$, regard $V(U_{F,\y}) \subseteq (\mathbb{C}^{*})^{r}.$ In \cite{BudurBernstein--Sato}, Budur proposes that the relationship between the roots of the Bernstein--Sato polynomial and the eigenvalues of the algebraic monodromy is generalized by the conjecture
\begin{equation} \label{eqn budur conjecture intro}
\text{Exp}(\V(B_{F,\x})) = \bigcup\limits_{\y \in \V(f) \text{ near } \x} \text{res}_{\y}^{-1} (V(U_{F,\y})). 
\end{equation}
where $\text{res}_{\y}$ restricts a local system on $U_{\x}$ to a local system on $U_{\y}$.
(This generalization passes through the support of the Sabbah specialization complex in the same way that the proof of the univariate version uses the support of the nearby cycle functor.)

This paper follows two threads. First we study the logarithmic derivations $\Der_{X}(- \log f)$ of $f$ inside $\ann_{\D_{X}[S]} F^{S}.$ We are motivated by \cite{uli} where Walther shows that, in the univariate case and with some mild hypotheses on the divisor of $f$, these members generate $\ann_{\D_{X}[s]}f^{s}.$ 

We restrict ourselves to ``nice" divisors: strongly Euler-homogeneous (possessing a particular logarithmic derivation locally everywhere); Saito-holonomic (the logarithmic stratification is locally finite); tame (a restriction on homological dimension). The main result of Section 2 is the following:

\begin{theorem} \label{thm intro first theorem}
Let $F = (f_{1}, \dots, f_{r})$ be a decomposition of $f = f_{1} \cdots f_{r}.$ If $f$ is strongly Euler-homogeneous, Saito-holonomic, and tame then 
\[
\ann_{\D_{X}[S]}F^{S} = \D_{X}[S] \cdot \{ \delta - \sum\limits_{k=1}^{r} s_{k} \frac{\delta \bullet f_{k}}{f_{k}} \mid \delta \in \Der_{X}(-\log f) \}.
\]
\end{theorem}

The strategy is to take a filtration of $\D_{X}[S]$ and consider the associated graded object of $\ann_{\D_{X}[S]}F^{S}$. This object can be given a second filtration so its initial ideal is similar to the Liouville ideal of \cite{uli}. Appendix A provides the mild generalizations of Gr\"obner type arguments necessary to transfer properties from this initial ideal to the ideal itself and Section 2 proves nice things about our associated graded objects, culminating in Theorem \ref{thm intro first theorem}. In \cite{MaisonobeArrangement}, Maisonobe proves a similar statement in the more restrictive setting of free divisors where many of these methods are not needed. We crucially use one of his techniques.

\iffalse

The strategy is to take a filtration of $\D_{X}[S]$ and consider the associated graded object of $\ann_{\D_{X}[S]}F^{S}$ and RHS agree--this is similar to the philosophy in \cite{uli}. The associated graded of $\ann_{\D_{X}[S]} F^{S}$ can be given a new grading so that is new initial object is very similar to the Liouville ideal of \cite{uli}. Appendix A builds up the theory necessary to transfer properties from initial ideals to ideals themselves and Section 2 proves nice things about our associated graded objects, as well as proving the above theorem. Maisonobe in \cite{MaisonobeArrangement} proves similar things in the more restrictive setting free divisors that lets him avoid Appendix A and a lot of Section 2. We crucially use one of his techniques.
\fi

Not much is known about particular elements of $\V(B_{F})$ even when $F$ corresponds to a factorization (not necessarily into linear forms) of a hyperplane arrangement. In \cite{BudurBernstein--Sato} Budur generalized the $- \frac{n}{d}$ conjecture (see Conjecture 1.3 of \cite{uli}) as follows:

\begin{conjecture} \label{conjecture intro}
\text{\normalfont (Conjecture 3 in \cite{BudurBernstein--Sato})} Let $f = f_{1} \cdots f_{r}$ be a central, essential, indecomposable hyperplane arrangement in $\mathbb{C}^{n}$. Let $F = (f_{1}, \dots, f_{r})$ where the $f_{k}$ are central hyperplane arrangements, not necessarily reduced, of degree $d_{k}.$ Then 
\[
\{ d_{1} s_{1} + \dots + d_{r}s_{r} + n = 0 \} \subseteq V(B_{F}).
\]
\end{conjecture} 
Using Theorem \ref{thm intro first theorem}, we can prove Conjecture \ref{conjecture intro} in the tame case:

\begin{theorem} \label{thm intro conjecture}

Let $f = f_{1} \cdots f_{r}$ be a central, essential, indecomposable, and tame hyperplane arrangement in $\mathbb{C}^{n}$. Let $F = (f_{1}, \dots, f_{r})$ where the $f_{k}$ are central hyperplane arrangements, not necessarily reduced, of degree $d_{k}.$ Then 
\[
\{ d_{1} s_{1} + \dots + d_{r} s_{r} + n = 0 \} \subseteq V(B_{F}).
\]
\end{theorem}

Conjecture \ref{conjecture intro} was motivated by the formulation of the Topological Multivariable Strong Monodromy Conjecture due to Budur, see Conjecture 5 of \cite{BudurBernstein--Sato}. We now state this. First let $f = f_{1} \cdots f_{r}$ with each $f_{k} \in \mathbb{C}[x_{1}, \dots, x_{n}]$ and let $F = (f_{1}, \dots, f_{r})$. Given a log resolution $\mu: Y \to X$ of $f$, let $\{E_{i} \}_{i \in S}$ be the irreducible components of $f \circ \mu $, let $a_{i, j}$ be the order of vanishing of $f_{j}$ along $E_{i}$, let $k_{i}$ be the order of vanishing of the determinant of the Jacobian of $\mu$ along $E_{i}$, and, for $I \subseteq S$, let $E_{I}^{\circ} = \cap_{i \in I} \setminus \cup_{i \in S \setminus I} E_{i}$. The \emph{topological zeta function} of $F$ is 
\[
Z_{F}^{\text{top}}(S) := \sum\limits_{I \subseteq S} \chi(E_{I}^{\circ}) \cdot \prod\limits_{i \in I} \frac{1}{a_{i,1}s_{1} + \cdots + a_{i, r}s_{r} + k_{i} + 1}
\]
and this is independent of the resolution. Conjecture 5 of \cite{BudurBernstein--Sato} states:

\begin{conjecture} \label{conjecture topological multi strong}
\text{\normalfont (Topological Multivariable Strong Monodromy Conjecture)}
The polar locus of $Z_{F}^{\text{top}}(S)$ is contained in $\V(B_{F}).$

\end{conjecture}

By work of Budur in loc. cit., Conjecture \ref{conjecture intro} implies Conjecture \ref{conjecture topological multi strong} for hyperplane arrangements. Consequently, we conclude Section 2 with the following:

\begin{cor} \label{cor topological multivariable mondodromy}
The Topological Multivariable Strong Monodromy Conjecture is true for (not necessarily reduced) tame hyperplane arrangements.
\end{cor}

The paper's second thread follows the link between $\text{Exp}(\V(B_{F,\x}))$ and the cohomology support loci of $f$ near $\x$. The bridge between the two is, with $A = (a_{1}, \dots, a_{r}) \in \mathbb{C}^{r}$, resp. $A-1 = (a_{1}-1, \dots, a_{r}-1) \in \mathbb{C}^{r}$, the $\D_{X,\x}$-linear map
\[ \nabla_{A}:  \frac{\D_{X,\x}[S]F^{S}}{(S-A)\D_{X,\x}[S]F^{S}} \to \frac{\D_{X,\x}[S]F^{S}}{(S-(A-1))\D_{X,\x}[S]F^{S}}.
\]
Here $(S-A)\D_{X,\x}[S]F^{S}$, resp. $(S-(A-1))\D_{X,\x}[S]F^{S}$, is the submodule of $\D_{X,\x}[S]F^{S}$ generated by $s_{1}-a_{1}, \dots, s_{r} - a_{r}$, resp. $s_{1} - (a_{1} - 1), \dots, s_{r} - (a_{r} - 1)$, and $\nabla_{A}$ is induced by $F^{S} \mapsto F^{S+1}.$ In the classical, univariate case, the following are equivalent (cf. Bj\"ork, 6.3.15 of \cite{Bjork}): (a) $A-1 \notin \V(B_{f,\x})$; (b) $\nabla_{A}$ is injective; (c) $\nabla_{A}$ is surjective.
\iffalse
\begin{enumerate}[(a)]
    \item $s-(a-1) \notin \V(B_{f,\x})$
    \item $\nabla_{A}$ is injective
    \item $\nabla_{A}$ is surjective.
\end{enumerate}
\fi
Showing that (a), (b), and (c) are equivalent in the multivariate case would verify that $\text{Exp}(\V(B_{F,\x}))$ equals the cohomology suport loci of $f$ near $\x$. Moreover, under the hypotheses of Theorem \ref{thm intro first theorem}, it would show that intersecting $V(B_{F,\x})$ with appropriate hyperplanes gives $V(B_{f,\x})$.

In any case, (a) implies (b) and (c). Under the hypotheses of Theorem \ref{thm intro first theorem}, we prove that $s_{1}-a_{1}, \dots, s_{r} - a_{r}$ behaves like a regular sequence on $\D_{X}[S]F^{S}$. This allows us to recreate a picture similar to Bj\"ork's and prove, using different methods, the main result of Section 3:

\begin{theorem}
Let $f = f_{1} \cdots f_{r}$ be strongly Euler-homogeneous, Saito-holonomic, and tame and let $F = (f_{1}, \dots, f_{r}).$ If $\nabla_{A}$ is injective then $\nabla_{A}$ is surjective.
\end{theorem}

In Section 4 we strengthen the hypotheses of Theorem \ref{thm intro first theorem} and assume $f$ is reduced and free, that is, we assume $\Der_{X,\x}(-\log f)$ is a free $\mathscr{O}_{X,\x}$-module. In \cite{MacarroDuality} Narv\'aez--Macarro computed the $\D_{X,\x}[s]$-dual of $\D_{X,\x}[s]f^{s}$ for certain free divisors; in \cite{MaisonobeArrangement}, Maisonobe shows that this computation easily applies to $\D_{X,\x}[S]$-dual of $\D_{X,\x}[S]F^{S}$. For our free divisors we compute the $\D_{X,\x}$-dual of $\frac{\D_{X,\x}[S]F^{S}}{(S-A)\D_{X,\x}[S]F^{S}}$ and lift $\nabla_{A}$ to this dual. Consequently, we prove:

\begin{theorem}

Let $f = f_{1} \cdots f_{r}$ be reduced, strongly Euler-homogeneous, Saito-holonomic, and free and let $F = (f_{1}, \dots, f_{r}).$ Then $\nabla_{A}$ is injective if and only if $\nabla_{A}$ surjective.

\end{theorem}

In Section 5 we summarize the relationship between the cohomology support loci of $f$ near $\x$, $\text{Exp}(V(B_{F,\x})$, and $\nabla_{A}$. In \cite{BudurWangEtAl}, the authors characterize membership in the cohomology support loci of $f$ near $\x$ in terms of the simplicity of certain perverse sheaves. When $f$ is reduced, strongly Euler-homogeneous, Saito-holonomic, and free, we show this characterization can be stated in terms of the simplicity of the $\D_{X,\x}$-module $\frac{\D_{X,\x}[S]F^{S}}{(S-A)\D_{X,\x}[S]F^{S}}.$

After this paper's completion, a preprint \cite{BudurConjectureOver} by Budur, Veer, Wu, and Zhou was announced proving \eqref{eqn budur conjecture intro}, that is, proving Budur's Conjecture 2 from \cite{BudurBernstein--Sato}. While this makes Section 5 less interesting, it does not effect Sections 2-4.

Finally, the author would like to thank Michael Kaminski, Harrison Wong, Luis Narv\'aez-Macarro, Uli Walther, and the referee for all their helpful conversations and comments.

\section{The $\D_{X}[S]$-Annihilator of $F^{S}$}

\iffalse 
\begin{definition}

Define an equivalence on $X$ by identifying two points $\x$ and $\y$ if there exists $U \ni \x$, $\y$, a derivation $\delta \in \Der_{U}(\log (Y \cap U))$ such that (1) $\delta$ is nowhere vanishing on $U$ and (2) the integral flow of $\delta$ passes through $\x$ and $\y$. This equivalence relationship breaks up $X$ into equivalence classes. The irreducible components of the equivalence classes are called the logarithmic strata; the collection of all strata the logarithmic stratification. 

We say $Y$ is Saito-holonomic if the logarithmic stratification is locally finite, i.e. at every $\x \in X$ there is a $U \ni \x$ such that $U$ intersects finitely many logarithmic strata. 

\end{definition}
\fi

As in the introduction, let $X$ be a smooth analytic space or $\mathbb{C}$-scheme of dimension $n$ and with structure sheaf $\mathscr{O}_{X}$. Let $f \in \mathscr{O}_{X}$ be regular with divisor $Y = \text{Div}(f)$ and corresponding ideal sheaf $\mathscr{I}_{Y}$. Throughout, $Y = \text{Div}(f)$ will not necessarily be reduced. Let $\D_{X}$ be the sheaf of $\mathbb{C}$-linear differential operators with $\mathscr{O}_{X}$-coefficients and let $\D_{X}[s]$ and $\D_{X}[S] = \D_{X}[s_{1}, \dots, s_{r}]$ be polynomials rings over $\D_{X}.$

Recall the \emph{order filtration} $F_{(0,1)}$ on $\D_{X}$ induced, in local coordinates, by making every $\partial_{x_{k}}$ weight one and every element of $\mathscr{O}_{X}$ weight zero. Denote the differential operators of order at most $k$ as $F_{(0,1)}^{k}$ and the associated graded object as $\gr_{(0,1)}(\D_{X}).$ \iffalse We will be primarily interested in a subset of $\Der_{X}$ called the logarithmic derivations: \fi

\begin{definition}
Let $\Der_{X}(-\Log f) = \Der_{X}(-\Log(Y))$, be the sheaf of \emph {logarithmic derivations}, i.e. the $\mathscr{O}_{X}$-module with local generators on $U$ the set \[
\Der_{X}(-\log f) := \{\delta \text{ a vector field in } \D_{X}(U) \mid \delta \bullet \mathscr{I} \subseteq \mathscr{I} \}.
\] We also put
\[
\Der_{X}(-\Log_{0} f) := \{\delta \in \Der_{X}(-\Log f) \mid \delta \bullet f = 0 \}.
\]
Note that $\Der_{X}(- \Log_{0} f)$ may depend on the choice of defining equation for f, which is why we have fixed a global $f$.
\end{definition}

\begin{definition}
For $\x \in X$, we say that $f \in \mathscr{O}_{X,\x}$ is \emph{Euler-homogeneous} at $\x$ if there exists $E_{\x} \in \Der_{X,\x}(-\log f)$ such that $E_{\x} \bullet f = f.$ If $E_{\x}$ vanishes at $\x$ then $f$ is \emph{strongly Euler-homogeneous} at $\x.$

Finally, a divisor $Y$ is (strongly) Euler-homogeneous if there is a defining equation $f$ at each $\x$ such that $f$ is (strongly) Euler-homogeneous at $\x$.
\end{definition}

\begin{example} \label{example running non-free example strong Euler}
Let $f = x(2x^{2} + yz)$. Note that $\text{Sing}(f) = \{z-\text{axis}\} \cup \{y-\text{axis}\}$. Along the $z$-axis there is the strong Euler-homogeneity induced by $\frac{1}{3} x (\partial_{x} \bullet f) + \frac{2}{3} y (\partial_{y} \bullet f)$; along the $y$-axis there is the strong Euler-homogeneity induced by $\frac{1}{3} x (\partial_{x} \bullet f) + \frac{2}{3} z (\partial_{z} \bullet f)$. Since $f$ is automatically strongly Euler-homogeneous on the smooth locus, $f$ is strongly Euler-homogeneous everywhere.
\end{example}

\begin{example} \label{example hyperplane strongly Euler-homogeneous}
Let $f$ be a central hyperplane arrangement. Then the Euler vector field $\sum x_{i} \partial_{x_{i}}$ shows that $f$ is strongly Euler-homogeneous at the origin. A coordinate change argument implies $f$ is strongly Euler-homogeneous. 
\end{example}

\begin{definition} \label{def total order filtration}

Define the \emph{total order filtration} $F_{(0,1,1)}$ as the filtration on $\D_{X}[S]$ induced by the $(0,1,1)$-weight assignment that, in local coordinates, gives elements of the form $\mathscr{O}_{U}\partial^{\textbf{u}}S^{\textbf{v}}$, $\textbf{u}$, $\textbf{v}$ non-negative integral vectors, weight $\sum u_{i} + \sum v_{i}$. Let $F_{(0,1,1)}^{k}$ be the homogeneous operators of weight at most $k$ with respect to the total order filtration. When the context is clear, we will use $F_{(0,1,1)}^{k}$ to refer to the similarly defined filtration on $\D_{X}[s]$ (the classical case). Denote the associated graded object associated to $F_{(0,1,1)}$ as $\gr_{(0,1,1)}(\D_{X}[S]).$
\end{definition}

Our principal objective is to study the annihilator of $F^{S}$---the left $\D_{X}[S]$-ideal $\ann_{\D_{X}[S]} F^{S}$. Take the $\mathscr{O}_{X}[f_{1}^{-1},\dots, f_{r}^{-1}, S]$-module generated freely by the symbol $F^{S} = \prod f_{k}^{s_{k}}.$ To make this a $\D_{X}[S]$-module define, for a derivation $\delta$ and $h \in \mathscr{O}_{X},$
\[\delta \bullet \frac{h S^{\textbf{v}}}{f^{j}}F^{S} = \delta \bullet (\frac{h}{f^{j}}) S^{\textbf{v}} F^{S} + \sum\limits_{k} s_{k} \frac{(\delta \bullet f_{k}) h S^{\textbf{v}}}{f_{k}f^{j}} F^{S}.\]

In most cases $\ann_{\D_{X}[S]}F^{S}$ is very hard to compute. In the classical setting, there is a natural identification between the $(0,1,1)$-homogeneous elements of $\ann_{\D_{X}[s]} f^{s}$ and $\Der_{X}(-\log f).$ We will establish a similar correspondence.

\begin{definition}

The \emph{annihilating derivations} of $F^{S}$ are the elements of the $\mathscr{O}_{X}$-module
\[\theta_{F} := \ann_{\D_{X}[S]}F^{S} \cap F_{(0,1,1)}^{1}.\]
We say $\ann_{\D_{X}[S]} F^{S}$ is \emph{generated by derivations} when $\ann_{\D_{X}[S]} F^{S} = \D_{X}[S] \cdot \theta_{F}.$
\end{definition}

\begin{prop} \label{prop psi} For $f = f_{1} \cdots f_{r}$, let $F = (f_{1}, \dots , f_{r}).$ Then as $\mathscr{O}_{X}$-modules, 
\[
\psi_{F}: \Der_{X}(- \log f) \xrightarrow{\simeq} \theta_{F}
\] where $\psi_{F}$ is given by 
\[
\delta \mapsto \delta - \sum_{k=1}^{r} s_{k} \frac{\delta \bullet f_{k}}{f_{k}}.
\]
\end{prop}

\begin{proof}

We first prove the claim locally. By Lemma 3.4 of \cite{GrangerSchulzeFormal}, $\Der_{X}(-\Log f) = \bigcap \Der_{X}(-\Log f_{k})$; in particular, $\delta - \sum_{k} s_{k} \frac{\delta \bullet f_{k}}{f_{k}}$ lies in $\D_{X,\x}[S].$
\iffalse
First, in $\mathscr{O}_{x,\x}$, the condition of each pair of the $f_{i}$'s having greatest common divisor 1 includes the case of some of the $f_{i}$'s being units.
\fi

Fix a coordinate system. Take $P \in \theta_{F,\x}$, $P = \delta + p(S)$, where $\delta \in \D_{X,\x}$ is a derivation and $p(S) = \sum_{k} b_{k}s_{k} \in \mathscr{O}_{X,\x}[S]$ is necessarily $S$-homogeneous of $S$-degree 1. \iffalse Write $P = \sum_{i} a_{i} \partial_{x_{i}} - \sum_{k} b_{k} s_{k} \in \theta_{F, \x}$, where $a_{i}$, $b_{j} \in \mathscr{O}_{X, \x}$. \fi Keep the notation $F^{s}$ and the $f_{k}$ for the local versions at $\x$.  By definition,
\begin{align*}
    0 = \left( \delta - \sum_{k} b_{k}s_{k} \right) \bullet F^{S} = \sum_{k} \left( s_{k} \frac{\delta \bullet f_{k}}{f_{k}} - b_{k}s_{k} \right) F^{S}.
\end{align*}
\iffalse 
\begin{align*}
   0    &= \left( \sum\limits_{i} a_{i} \partial_{x_{i}} - \sum\limits_{k} b_{k} s_{k} \right) \bullet F^S 
            = \left( \sum\limits_{i} \sum\limits_{k} s_{i} \frac{a_{i} \partial_{x_{i}} \bullet f_{k}}{f_{k}} - \sum\limits_{k} b_{k} s_{k} \right) F^{S} \\
        &= \left( \sum\limits_{k}\frac{s_{k}}{f_{k}} \sum\limits_{i} a_{i}\partial_{x_{i}} \bullet f_{k}  - \sum\limits_{k} b_{k} s_{k} \right) F^{S} 
            = \left( \sum\limits_{k} \frac{s_{k}}{f_{k}} (\delta \bullet f_{k}) - \sum\limits_{k} b_{k} s_{k} \right) F^{S}.
\end{align*}
\fi
Because $\D_{X,\x} F^{S}$ is a free $\mathscr{O}_{X,\x}[f^{-1}, S]$-module $\sum_{k} (s_{k} \frac{\delta \bullet f_{k}}{f_{k}} - b_{k}s_{k})= 0.$ Thus for each $k$, $\frac{\delta \bullet f_{k}}{f_{k}}s_{k} -  b_{k} s_{k} = 0$. So $\delta \bullet f_{k} \in \mathscr{O}_{X,\x} \cdot f_{k}$; moreover, $\delta \bullet f_{k} = b_{k} f_{k}.$
\iffalse
Multiply by $f$ and write $\hat{f_{j}} = \frac{f}{f_{j}}$ to obtain:
\[
    0   = \left(\sum\limits_{j} s_{j}\hat{f_{j}} (\delta_{\x} \bullet      f_{j}) - \sum\limits_{j} f b_{j} s_{j}\right) F^{S}.
\]
Because  $\D_{X,\x}F^{S}$ is $ \mathscr{O}_{X,\x}[F^{-1}, S]$ free
\[ 0 = \sum\limits_{j} s_{j}\hat{f_{j}} (\delta_{\x} \bullet f_{j}) - \sum\limits_{j} f b_{j} s_{j}. \]
Going modulo a particular $f_{j}$, the definition of $\hat{f_{j}}$ implies 
 0   &= \sum\limits_{j} \overline{s_{j}\hat{f_{j}} (\delta_{\x} \bullet f_{j})} \text{ in } \mathscr{O}_{X,\x} / (f_{j}), \text{ which by definition of } \hat{f_{j}} \text{ gives } \\ 
 
\[ 0 = \overline{s_{j}\hat{f_{j}} (\delta_{\x} \bullet f_{j})} \text{ in } \mathscr{O}_{X,\x} / (f_{j}).
\]
This holds for all $j$. Therefore $\delta_{\x} \bullet f_{j} = \lambda_{j} f_{j}$ for suitable $\lambda_{j} \in \mathscr{O}_{X,\x}.$  $\in (f_{j})\mathscr{O}_{X,\x}$ \ while if $f_{j}$ is a unit this is immediate.\fi 

We have shown $\delta \in \bigcap \Der_{X,\x}(-\Log f_{k})$ and, in fact,

\iffalse So if $\delta_{\x} + p(S)_{\x} \in \theta_{F,\x}$, $\delta_{\x} \in \bigcap \Der(- \Log f_{i})$. Write $\delta_{\x} \bullet f_{j} = \lambda_{j}f_{j}$, $\lambda_{j}  \in \mathscr{O}_{X,\x}$. 

Then 
\begin{align*}
     0 = (\delta - p(S)) \bullet F^{S} 
        &=  \sum\limits_{j} s_{j} \lambda_{j} F^{S} - \sum\limits_{j} b_{j} s_{j} F^{S}.
\end{align*}

So $p(S) = \sum\limits_{j} s_{j} \lambda_{j}$ and 
\fi
\[\theta_{F,\x} = \{ \delta - \sum\limits_{k} b_{k}s_{k} \mid \delta \in \Der_{X,\x}(- \Log f) \text{, } \delta \bullet f_{k} = b_{k} f_{k} \}.\]

Thus the map $\psi_{F}: \Der_{X,\x}(-\log f) \to \theta_{F,\x}$ given by $\delta \mapsto \delta - \sum_{k} \frac{\delta \bullet f_{k}}{f_{k}} s_{k}$ is a well-defined $\mathscr{O}_{X,\x}$-linear isomorphism for a fixed coordinate system. Showing that $\theta_{F,\x}$ commutes with coordinate change is routine and is effectively shown in Remark \ref{remark well-behaved} (b).

Since $\delta \in \Der(- \Log f)$ precisely when $\delta \bullet f = 0$ in $\mathscr{O}_{X} / (f)$, membership in  $\Der(- \Log f)$ is a local condition. The above shows that $\psi_{F,\x}^{-1}$ is an $\mathscr{O}_{X,\x}$-isomorphism at all $\x$; hence $\psi_{F}^{-1}$ is an isomorphism. \iffalse of $\theta_{F, \x}$ and $\Der_{X,\x}(- \Log f)$; hence $\psi_{F}^{-1}(\theta_{F}) \subseteq \Der(- \Log f).$ Refine the target: $\psi_{F}^{-1}: \theta_{F} \to \Der_{X}(- \Log f)$. Again, the above shows that this is an isomorphism at each point. \fi
\end{proof}

\subsection{Hypotheses on $Y$ and $F$.} \text{ }

In this subsection we introduce many of the standard hypothesis on $Y$ and $F$ we use throughout the paper.

\begin{definition}

Let $U \subseteq X$ be open and $f \in \mathscr{O}_{X}(U).$ We will say $F = (f_{1}, \dots, f_{r})$ is a \emph{decomposition} of $f$ when $f = f_{1} \cdots f_{r}.$

\end{definition}

We will also restrict to divisors $Y$ such that $\Der_{X}(-\log Y)$ has a light constraint on its dimension.

\begin{definition} \label{def tame}

Consider the sheaf of differential forms of degree $k$: $\Omega_{X}^{k} = \bigwedge^{k} \Omega_{X}^{1}$ and the differential $\text{d}: \Omega_{X}^{k} \to \Omega_{X}^{k+1}.$ We define the subsheaf of \textit{logarithmic differential forms} $\Omega_{X}^{k}(\log f)$ by
\[ \Omega_{X}^{k}(\log f) := \{ w \in \frac{1}{f} \Omega_{X}^{k} \mid d(w) \in \frac{1}{f} \Omega_{X}^{k+1} \}. \] 
See 1.1 and 1.2 in \cite{SaitoTheoryLogarithmic} for more details.

We say $f \in \mathscr{O}_{X}(U)$, $U \subseteq X$ open, is \textit{tame} if the projective dimension of $\Omega_{U}^{k}( \log f)$ is at most k at each $\x \in U$. A divisor $Y$ is tame if it admits tame defining equations locally everywhere. See Defintion 3.8 and the surrounding text in \cite{uli} for more details on tame divisors. In particular, if $n \leq 3$ then $Y$ is automatically tame.

\iffalse The construction of $\Omega_{X}^{k}(\log f)$ does not depend on the defining equation of $f$ nor on coordinate system this definition.
\fi

\end{definition}

\iffalse

We will not have to get into details into tame divisors; for the most part we will cite results in \cite{uli} where tameness is assumed.
\fi

We will also use a stratification of $X$ that respects the logarithmic data of $Y$.

\begin{definition} (Compare with 3.8 in \cite{SaitoTheoryLogarithmic}) Define a relation on $X$ by identifying two points $\x$ and $\y$ if there exists an open $U \subseteq X$, $\x, \y \in U$ and a derivation $\delta \in \Der_{U}(-\log (Y \cap U))$ such that (i) $\delta$ is nowhere vanishing on $U$ and (ii) the integral flow of $\delta$ passes through $\x$ and $\y$. The transitive closure of this relation stratifies $X$ into equivalence classes. The irreducible components of the equivalence classes are called the \textit{logarithmic strata}; the collection of all strata the \emph{logarithmic stratification}. 

We say $Y$ is \textit{Saito-holonomic} if the logarithmic stratification is locally finite, i.e. at every $\x \in X$ there is an open $U \subseteq X$, $\x \in U$, such that $U$ intersects finitely many logarithmic strata. Equivalently, $Y$ is Saito-holonomic if the dimension of $ \{ \x \in X \mid \text{rk}_{\mathbb{C}}(\Der_{X}(-\log Y) \otimes \mathscr{O}_{X,\x} / \mathfrak{m}_{X,\x} = i\}$ is at most $i$.

\end{definition}

\begin{remark} \label{remark saito-holonomic}

\begin{enumerate}[(a)]
    \item Pick $\x \in X$ and consider its log stratum $D$ with respect to $f = f_{1} \cdots f_{r}.$ We can find logarithmic derivations $\delta_{1}, \dots, \delta_{m}$ at $\x$, $m = \text{dim } D$, that are $\mathbb{C}$-independent at $\x$. Each $\delta_{i}$ also lies in $\Der_{X,\x}(-\log f_{i})$. By Proposition 3.6 of \cite{SaitoTheoryLogarithmic} there exists a coordinate system $(x_{1}, \dots, x_{n})$ so that these generators can be written as $\delta_{k} = \frac{\partial}{\partial x_{n-m +k}} + \sum_{1 \leq j \leq n-m} g_{j k}(x) \frac{\partial}{\partial x_{j}}$, with the $g_{jk}$ analytic functions defined near $\x$. \iffalse Hereafter write the first $n-m$ $x$'s as $x$, the last $m$ $x$'s as $y.$  \fi
    \item By Lemma 3.5 and Proposition 3.6 of \cite{SaitoTheoryLogarithmic}, the same change of coordinates $\phi_{F}$ from \ref{remark saito-holonomic}.(a) fixes the last $m$ coordinates and satisfies $\phi_{F}(x_{1}, \dots, x_{n-m},0) = (x_{1}, \dots, x_{m}, 0)$. Moreover, it simultaneously satisfies $f_{i}(\phi_{F}(x_{1}, \dots, x_{m})) = u_{i}(x_{1}, \dots, x_{m})f_{i}(x_{1}, \dots, x_{n-m},0)$ where $u_{i}(x_{1}, \dots, x_{m})$ is a unit for $1 \leq i \leq m$.
    
    \iffalse
    \item Let's sketch the construction of the above coordinate transformation. So assume $m=1$, giving us only one y and one $\delta$ with local flow $\theta(x,y,t)$. Working on some small neighborhood (so we can use the flow), project onto the first $n$ coordinates, $(x,y) \mapsto (x, 0)$ and then use $\theta$ to flow $y$ $t$-units, that is, travel the integral curve of $\theta$ starting at $(x, 0)$ for a duration of $t = y$. Call the end point $(\tilde{x}, \tilde{y}).$ Since the coefficient function of $\frac{\partial}{\partial y}$ in $\delta$ is 1, $\theta_{n+1}(x,y,t) = y + t.$ So in particular, starting at $(x,0)$ and flowing $y=t$ units the $n+1$ coordinate goes from $0$ to $y$; i.e. $\tilde{y} = y.$ The coordinate transformation is then given by $(x,y) \mapsto (\tilde{x}, \tilde{y}) = (\tilde{x}, y)$. Since $\delta \bullet f_{i} = h_{i}f_{i}$, taking the differential with respect to t shows $f_{i}(\theta(x,y,t)) = e^{t h_{i}(x,y)} f_{i}(x,y)$ and in particular, $f_{i}(\tilde{x}, \tilde{y}) = f_{i}(\theta(x, 0, y)) = e^{y h_{i}(x,0)}f_{i}(x,0)$. Since $e^{y h_{i}(x,0)}$ is a unit, this coordinate change has the desired properties. Morally, we are selecting the integral curve of $\delta$ starting at $(x,y)$ to be our $y$-coordinate in the new coordinate system.
    \fi
    \item Now assume the logarithmic stratification is locally finite and the log stratum $D$ of $\x$ has dimension $0$. So $D = \{\x \}.$ \iffalse Let $U \supseteq D$ be a open set intersecting finitely many strata different from $D$. \fi Since every other zero dimensional strata is disjoint from $D$, there exists an open $U \ni \x$ such that $U \setminus \x$ consists only of points whose logarithmic stratum are of positive dimension. \iffalse Suppose $D^{\prime}$ is another strata intersecting every such $U$. Then $D = \{\x \} \subseteq \overline{D^{\prime}}$.  By the ``frontier condition" of stratification (cf. Lemma 3.2 \cite{SaitoTheoryLogarithmic}) this implies $D \subseteq \partial{D^{\prime}}.$ In particular, $D^{\prime}$ cannot be a point. The utility of this argument for us is the following:  So $\text{dim } D^{\prime} \geq 1.$ Thus, if $\x$ is a point whose logarithmic strata is of dimension $0$, then there exists an open $U \ni \x$ such that $U \setminus \x$ consists only of points whose logarithmic stratum are of positive dimension. \fi 
    \item By Lemma 3.4 of \cite{SaitoTheoryLogarithmic}, for a divisor $Y$ connected components of $X \setminus Y$ and $Y \setminus \text{Sing}(Y)$ are logarithmic strata. \iffalse In particular, if $\text{dim } X = n$, any log stratum of dimension $n$ is comprised solely of points in $ X \setminus \text{Sing}(Y).$ \fi
\end{enumerate}

\end{remark}

\begin{example} \label{example running non-free example saito-holonomic}
Let $f= x(2x^2 + yz)$ and note that $\text{Sing}(f) = \{z-\text{axis}\} \cup \{ y-\text{axis} \}.$ Since the Euler derivation $x \partial_{x} + y \partial_{y} + z \partial_{z}$ is a logarithmic derivation, the $z$-axis $\setminus \{0\}$ and the $y$-axis $\setminus \{ 0 \}$ are logarithmic strata. Therefore $f$ is Saito-holonomic.
\end{example}

\begin{example} \label{example hyperplane arrangement saitoholonomic}
By Example 3.14 of \cite{SaitoTheoryLogarithmic}, hyperplane arrangements are Saito-holonomic.
\end{example}

\iffalse We will almost always assume our divisors $Y$ are strongly Euler-homogeneous for reasons that will become clear in Remark \ref{remark well-behaved}.
\fi
\iffalse
\begin{Convention}

For the rest of this section we will always assume our divisors $Y$ are strongly Euler-homogeneous, Saito-holonomic, and tame. 
\end{Convention}
\fi
\subsection{Generalized Liouville Ideals.} \text{ }

In Section 3 of \cite{uli}, Walther defines the \emph{Liouville ideal} $L_{f}$ as the ideal in $\gr_{(0,1)}(\D_{X})$ generated by the symbols $\gr_{(0,1)}(\Der_{X}(- \log_{0} f)$. As $\Der_{X}(-\log_{0} f) \subseteq \ann_{\D_{X}} f^{s}$, $L_{f}$ represents the contribution of $\Der_{X}(-\log_{0} f)$ to $\gr_{(0,1)}(\ann_{\D_{X}} f^{s})$. When $f$ is strongly Euler-homogeneous with strong Euler-homogeneity $E_{\x}$, $L_{f}$ is coordinate independent (see Remark 3.2 \cite{uli}) and $\gr_{(0,1,1)}(\D_{X,\x}[s]) \cdot L_{f,\x}$ and $\gr_{(0,1)}(E_{\x}) - s$ generate the simplest degree one elements of $\gr_{(0,1,1)}(\ann_{\D_{X,\x}[s]} f^{s})$.

If we want to generalize this to $F^{S}$, there is no obvious inclusion between $\Der_{X}(-\log_{0} f)$ and $\ann_{\D_{X}} F^{S}$. In fact, $\delta \in \Der_{X}(-\log_{0} f)$ is in $\ann_{\D_{X}} F^{S}$ preciscely when $\delta \in \bigcap \Der_{X}(- \log_{0} f_{i}).$ Trying to define a generalized Liouville ideal using $\bigcap \Der_{X}(- \log_{0} f_{i})$ would lose too many elements of $\Der_{X}(-\log_{0} f)$. 

\iffalse
So instead of trying to relate $\Der_{X}(-\log_{0} f)$ to $\gr_{(0,1)}(\D_{X})$ we will relate it to $\gr_{(0,1,1)}(\D_{X}[S])$ using the splitting of Remark \ref{remark split} (c).
\fi

\begin{definition} \label{definition generalized Liouville}
Recall the isomorphism of $\mathscr{O}_{X}$-modules from Proposition \ref{prop psi}
\[
\psi_{F} : \Der_{X}(-\Log f) \xrightarrow{\simeq} \theta_{F},
\]
which is given by
\[
\psi_{F}(\delta) = \delta - \sum s_{k} \frac{\delta \bullet f_{k}}{f_{k}}.
\]
This restricts to a map of sheaves of $\mathscr{O}_{X}$-modules:
\[
\psi_{F} : \Der_{X}(- \Log_{0}f) \hookrightarrow \theta_{F}.
\] 
Let the \textit{generalized Liouville ideal} $L_{F}$ by the ideal in $\gr_{(0,1,1)}(\D_{X}[S])$ generated by the symbols of $\psi_{F}(\Der_{X}(- \Log_{0}f))$ in the associated graded ring:
\[ L_{F} := \gr_{(0,1,1)}(\D_{X}[S]) \cdot \gr_{(0,1,1)}(\psi_{F}(\Der_{X}(- \Log_{0}f))).\]
We also define
\begin{align*}
    \widetilde{L_{F}} 
        :&= \gr_{(0,1,1)}(D_{X}[S]) \cdot \gr_{(0,1,1)}(\theta_{F}) \\
        &= \gr_{(0,1,1)}(D_{X}[S]) \cdot \gr_{(0,1,1)}(\psi_{F}(\Der_{X}(-\Log (f))).
\end{align*}
\iffalse
By Remark \ref{remark split}, at $\x \in X$ such that $f$ is Euler-homogeneous, $ \widetilde{L_{F}}$ is the ideal sum $L_{F} + \gr_{(0,1,1)}(E_{\x} -
\sum\limits_{k} s_{k} \frac{E_{\x} \bullet f_{k}}{f_{k}}),$ though this may depend on $E_{\x}$.
\fi

\end{definition}

\begin{remark} \label{remark well-behaved} 
\iffalse The splittings of $\Der_{X,\x}(\log f)$ into $\Der_{X,\x}(-\log_{0} f)$ and $\mathscr{O}_{X, \x} \cdot E_{\x}$, $E_{\x}$ a Euler-homogeneity at $\x \in X$, of Remark \ref{remark split} may depend on the choice of equation of $f$ at $\x$. \fi With $E_{x}$ a Euler-homogeneity for $f$ at $\x$, the $\mathscr{O}_{X,\x}$-module direct sum $\Der_{X,\x}(- \log f) \simeq \Der_{X,\x}(-\log_{0} f) \oplus \mathscr{O}_{X,\x} \cdot E_{x}$ $L_{F,\x}$, depends on the choice of defining equation for $f$. Following Remark 3.2 of \cite{uli}, if the the divisor of $f$ is strongly Euler-homogeneous, then the algebraic properties of $L_{F,\x}$ and $\widetilde{L_{F,\x}}$ are independent of the choice of local equation of $\text{Div}(f)$.

\begin{enumerate}[(a)]

\item To this end, let $x$ and $x^{\prime}$ denote two coordinates systems, $J = (\frac{\partial x_{j}^{\prime}}{\partial x_{i}})$ the Jacobian matrix with rows $i$, columns $j$, $\partial$ and $\partial^{\prime}$ column vectors of partial differentials in the $x$ and $x^{\prime}$ coordinates, respectively. Let $\nabla(g)$, $\nabla^{\prime}(g)$ be the gradient, as a column vector, of g in the two coordinate systems. Finally, express a derivation $\delta$ in terms of the two coordinate systems as $\delta = c_{\delta}^{T}\partial = c_{\delta}^{\prime \,T} \partial^{\prime}$, where $c_{\delta}$, $c_{\delta}^{\prime}$, are column vectors of $\mathscr{O}_{X}$ functions representing the coefficients of the partials in the $x$ and $x^{\prime}$ coordinates. Note that in $x^{\prime}$-coordinates $c_{\delta}^{\prime} = J^{T} c_{\delta}$. 

\item In $x$-coordinates $\psi_{F}(\delta) = c_{\delta}^{T}\partial - \sum_{k} s_{k} \frac{c_{\delta}^{T} \nabla(f_{k})}{f_{k}}$. In $x^{\prime}$-coordinates $\delta = c_{\delta}^{T} J \partial^{\prime}$ and $\psi_{F}(\delta) = c_{\delta}^{T}J\partial^{\prime} -\sum_{k} s_{k} \frac{c_{\delta}^{T} J \nabla^{\prime}(f_{k})}{f_{k}}$. Thus $\psi_{F}$ commutes with coordinate change. (Note that strongly Euler-homogeneous is not needed here.)

\item Suppose $E_{\x}$ is a strong Euler-homogeneity at $\x \in X$ for $f$. Recall from Remark 3.2 of \cite{uli} that for a unit $u \in \mathscr{O}_{X,\x}$, the map $\alpha_{u} : \Der_{X,\x}(- \Log_{0} f) \to \Der_{X,\x}(- \Log_{0} uf)$ given by $\alpha_{u}(\delta) = \delta - \frac{\delta \bullet u}{u + E_{\x} \bullet u} E_{\x}$ is an $\mathscr{O}_{X,\x}$-isomorphism that commutes with coordinate change. In particular, let $u = \prod_{1 \leq i \leq r} u_{i}$ be a product of units and let $uF = (u_{1}f_{1}, \dots, u_{r}f_{r})$. Then we have an $\mathscr{O}_{X,\x}$-isomorphism 
\[
\psi_{F} \circ \alpha_{u} \circ \psi_{F}^{-1} : \psi_{F}(\Der_{X,\x}(- \Log_{0}f)) \to \psi_{F}(\Der_{X,\x}(- \Log_{0} uf))
\]
that commutes with coordinate change. %(To make sure notation is clear, $\psi_{F}(\Der_{X,\x}(-\log_{0} uf) \subseteq \theta_{uF}.)$

\item To be precise, 
\begin{align*}
    \psi_{F} \circ \alpha_{u} \circ \psi_{F}^{-1}(\delta - \sum\limits_{k} s_{k} \frac{\delta \bullet f_{k}}{f_{k}}) &= \delta - \frac{\delta \bullet u}{u + E_{\x} \bullet u} E_{\x} - \sum\limits_{k} s_{k} \frac{\delta \bullet u_{k}}{u_{k}}  \\
        &- \sum\limits_{k} s_{k} \frac{\delta \bullet f_{k}}{f_{k}}+ \sum\limits_{k} s_{k} \frac{(\delta \bullet u)(E_{\x} \bullet (u_{k}f_{k}))}{u_{k}{f_{k}}(u + E_{\x} \bullet u)}.
\end{align*} 
Note that $E_{\x} \bullet (u_{k}f_{k})$ is a multiple of $f_{k}$ and $\delta \in \Der_{X,\x}(- \Log f_{k})$ so all these fractions make sense.

\item Inspection reveals that the morphism of graded objectes induced by $\psi_{F} \circ \alpha_{u} \circ \psi_{F}^{-1}$ is an $\mathscr{O}_{X,\x}[S]$-linear endomorphism $\beta_{u}$ on  $\gr_{(0,1,1)}\D_{X,\x}[S]$, where
\begin{align*}
   \beta_{u}(\gr_{(0,1,1)}(\partial)) 
        &= \gr_{(0,1,1)}(\partial) - \frac{\partial \bullet u}{u + E_{\x} \bullet u} \gr_{(0,1,1)}(E_{\x}) - \sum\limits_{k} s_{k} \frac{\partial \bullet u_{k}}{u_{k}} \\
        &+ \sum\limits_{k} s_{k} \frac{(\partial \bullet u)(E_{\x} \bullet (u_{k}f_{k}))}{u_{k}{f_{k}}(u + E_{\x} \bullet u)}.  
\end{align*}
Since the $\mathscr{O}_{X,\x}$-linear endomorphism of $\gr_{(0,1)}(\D_{X,\x})$ given by $\gr_{(0,1)}(\partial) \to \gr_{(0,1)}(\partial) - \frac{\partial \bullet u}{u + E_{\x} \bullet u} \gr_{(0,1)}(E_{\x})$ is surjective and injective, $\beta_{u}$ is as well. So $\beta_{u}(L_{F, \x}) = L_{uF, \x}$. It is clear by (d) that $\beta_{u}$ commutes with coordinate change.

\item Therefore for strongly-Euler-homogeneous $f$, the local algebraic properties of $\gr_{(0,1,1)}(\D_{X}[S]) / L_{F}$ are independent of the choice of local equations for the $f_{1}, \dots, f_{r}.$

\item It is also clear that $\alpha_{u}$ sends $E_{\x}$, a strong Euler-homogeneity for f, to a strong Euler-homogeneity for $uf$ and so $\beta_{u}(\widetilde{L_{F,\x}}) = \widetilde{L_{uF, \x}}.$ Hence, if $f$ is strongly Euler-homogeneous then the local properties of $\widetilde{L_{F}}$ do not depend on the defining equations of the $f_{k}$.
\end{enumerate}
\end{remark}

At the smooth points of $f$, $L_{F}$ and $\widetilde{L_F}$ are well understood. First, a lemma:

\begin{lemma} \label{lemma Euler calculation}

Suppose $f = f_{1} \cdots f_{r}$ has the Euler-homogeneity $E_{\x}$ at $\x \in X.$ Let $F = (f_{1}, \dots, f_{r}).$ Then \[\gr_{(0,1,1)}(\psi_{F}(E_{\x})) \notin \mathfrak{m}_{\x} \mathscr{O}_{X,\x}[Y][S] \subseteq \mathscr{O}_{X,\x}[Y][S] \simeq \gr_{(0,1,1)}(\D_{X,\x}[S]).\]

\end{lemma}
\begin{proof}

Working at $\x \in X$ and letting $\widehat{f_{k}} = \underset{j \neq k}{\prod} f_{j}$:
\begin{align*}
    f = E_{\x} \bullet f = \sum\limits_{k} (E_{\x} \bullet f_{k}) \widehat{f_{k}} = \left(\sum\limits_{k} \frac{E_{\x} \bullet f_{k}}{f_{k}} \right) f.
\end{align*}
So $ 1 = \sum \frac{E_{\x} \bullet f_{k}}{f_{k}}$ in $\mathscr{O}_{X,\x}$; thus there exists a $j$ such that $\frac{E_{\x} \bullet f_{j}}{f_{j}} \notin \mathfrak{m}_{\x}$. As $\psi_{F}(E_{\x}) = E_{\x} + \sum s_{k} \frac{E_{\x} \bullet f_{k}}{f_{k}}$ the claim follows after looking at the symbol $\gr_{(0,1,1)}(\psi_{F}(E_{\x})).$
\end{proof}

\begin{prop}\label{prop dim reference}

Let $f = f_{1} \cdots f_{r}$ be strongly Euler-homogeneous and let $F=(f_{1}, \dots, f_{r})$. Then locally at smooth points, $L_{F}$ and $\widetilde{L_{F}}$ are prime ideals of dimension $n+r+1$ and $n+r$ respectively. Moreover, for any $\x \in X$:
$$\dim \gr_{(0,1,1)} (\D_{X,\x}[S]) / L_{F,\x} \geq n + r + 1 ;$$
$$\dim \gr_{(0,1,1)} (\D_{X,\x}[S]) / \widetilde{L_{F,\x}} \geq n + r.$$

\end{prop}

\begin{proof}

Let $\x \in X$ be a part of the smooth locus of $f$; fix coordinates and choose $\partial_{x_{i}}$ such that $\partial_{x_{i}} \bullet f$ is a unit in $\mathscr{O}_{X,\x}.$ Then $\Gamma = \{\partial_{x_{k}} - \frac{\partial_{x_{k}} \bullet f}{\partial_{x_{i}} \bullet f} \partial_{x_{i}}\}_{k=1, k \neq i}^{n} \subseteq \Der_{X,\x}(- \Log_{0} f)$ is a set of $n-1$ linearly independent elements. Saito's Criterion (cf. page 270 of \cite{SaitoTheoryLogarithmic}) implies that $\Gamma$ together with $E_{\x}$, the strong Euler derivation, gives a free basis for $\Der_{X,\x}(- \log f)$. Hence, $\Gamma$ generates $\Der_{X,\x}(- \Log_{0} f)$ freely. As $\mathscr{O}_{X,\x}[Y][S] / L_{F,\x} \simeq \mathscr{O}_{X,\x}[y_{i}][S]$, $L_{F,\x}$ is a prime ideal of dimension $n+r+1$. 

By Lemma \ref{lemma Euler calculation}, and the choice of $j$ outlined in its proof, there is a ring map
\[
\mathscr{O}_{X,\x}[Y][S] / \gr_{(0,1,1)}(\psi_{F}(E_{\x})) \simeq \mathscr{O}_{X,\x}[Y][s_{1}, \dots, s_{j-1}, s_{j+1}, \dots, s_{r}].
\]
\iffalse Consider the image of $y_{k} - \frac{\partial_{x_{k}} \bullet f}{\partial_{x_{i}} \bullet f} y_{i}$ (denoted hereafter by $\overline(-)$) in $\mathscr{O}_{X,\x}[Y][S] / \gr_{(0,1,1)}(\psi_{F}(E_{\x}))$. \fi Consider the image of $\gr_{(0,1,1)}(\psi_{F}(y_{k} - \frac{\partial_{x_{k}} \bullet f}{\partial_{x_{i}} \bullet f} y_{i}))$ (hereafter denoted with $\overline{(-)}$) in $\mathscr{O}_{X,\x}[Y][S] / \gr_{(0,1,1)}(\psi_{F}(E_{\x}))$ \iffalse \simeq \mathscr{O}_{X,\x}[Y][s_{1}, \dots, s_{j-1}, s_{j+1}, \dots, s_{r}] \fi. Since $E_{\x}$ is a strong Euler-homogeneity, the coefficient of each $y_{k}$ in $\gr_{(0,1,1)}(\psi_{F}(E_{\x}))$ lies in $\mathfrak{m}_{\x}$. Thus the coefficient of $y_{k}$ in 
\begin{align*}
\overline{\gr_{(0,1,1)}(\psi_{F}(y_{k} - \frac{\partial_{x_{k}} \bullet f}{\partial_{x_{i}} \bullet f} y_{i}))} &\in \mathscr{O}_{X,\x}[Y][S] / \gr_{(0,1,1)}(\psi_{F}(E_{\x})) \\ 
    &\simeq \mathscr{O}_{X,\x}[Y][s_{1}, \dots, s_{j-1}, s_{j+1}, \dots, s_{r}].
\end{align*}
belongs to $\mathscr{O}_{X,\x} \setminus \mathfrak{m}_{\x}.$ So as rings, $\mathscr{O}_{X,\x}[Y] / \widetilde{L_{F,\x}} \simeq \mathscr{O}_{X,\x}[y_{i}][s_{1}, \dots, s_{j-1}, s_{j+1}, \dots, s_{r}]$ and $\widetilde{L_{F,\x}}$ is a prime ideal of dimension $n+r$.

Since the smooth points are dense, we get the desired inequalities.
\end{proof}

Take a generator $\gr_{(0,1,1)} \left( \delta - \sum s_{k} \frac{\delta \bullet f_{k}}{f_{k}} \right)$, $\delta \in \Der_{X}(-\log_{0}f)$, of $L_{F,\x}$. Erasing the $s_{k}$-terms results in $\gr_{(0,1,1)}(\delta) = \gr_{(0,1)}(\delta) \in L_{f,\x}$. This process is formalized by filtering $\gr_{(0,1,1)}(\D_{X,\x}[S])$ in such a way that the $s_{k}$-terms have degree 0 and then taking the initial ideal of $L_{F,\x}$.

\begin{definition}

It is well known that for an open $U \subseteq X$ with a fixed coordinate system $\gr_{(0,1,1)}(\D_{X}(U)[S]) \simeq \mathscr{O}_{X}(U)[Y][S]$, where $y_{i} = \gr_{(0,1,1)}(\partial_{x_{i}})$. Grade this by the integral vector $(0,1,0) \in \mathbb{N}^{n} \times \mathbb{N}^{n} \times \mathbb{N}^{r}$. For example the element $gY^{\textbf{u}}S^{\textbf{v}}$, where $\textbf{u}, \textbf{v}$ are nonnegative integral vectors and $g \in \mathscr{O}_{U}$, will have $(0,1,0)$-degree $\sum_{j} u_{j}.$ Changing coordinate systems does not effect the number of $y$-terms so this extends to a grading on $\gr_{(0,1,1)}(\D_{X}(U)[S])$. 

Define $\In_{(0,1,0)}L_{F}$ to be the initial ideal of the generalized Liouville ideal with respect to the $(0,1,0)$-grading. See Appendix A for details about initial ideals.

\end{definition}

We now have three ideals: $L_{F}$, $\In_{(0,1,0)} L_{F}$, and $ \gr_{(0,1,1)}(\D_{X}[S]) \cdot L_{f}$, the ideal extension of $L_{f}$ to $\gr_{(0,1,1)}(\D_{X}[S]).$ Proposition \ref{main grobner theorem} shows how some nice properties of $\In_{(0,1,0)} L_{F}$ transfer to nice properties of $L_{F}.$ The following construction will let us transfer nice properties of $L_{f}$, and consequently of $\gr_{(0,1,1)}(\D_{X}[S]) \cdot L_{f}$, to nice properties of $\In_{(0,1,1)} L_{F}.$

\begin{prop}\label{prop main map}

Assume $f = f_{1} \cdots f_{r}$ is strongly Euler-homogeneous and let $F= (f_{1}, \dots, f_{r}).$ Consider $\gr_{(0,1,1)}(\D_{X}[S]) \cdot L_{f}$, the extension of the Liouville ideal to $\gr_{(0,1,1)}(\D_{X}[S])$. Then there is a surjection of rings:
\begin{equation} \label{eqn Lf-initialLF}
\frac{\gr_{(0,1,1)}(\D_{X}[S])}{\gr_{(0,1,1)}(\D_{X}[S]) \cdot L_{f}} \twoheadrightarrow \frac{\gr_{(0,1,1)}(\D_{X}[S])}{\In_{(0,1,0)}L_{F}}.
\end{equation}

\end{prop}

\begin{proof}

$L_{f}$ is generated by the symbols of $\delta \in \Der_{X}(- \Log_{0} f)$ in $\gr_{(0,1)}(\D_{X})$. Thinking of $\gr_{(0,1)}(\D_{X}) \subseteq \gr_{(0,1,1)}(\D_{X}[S])$, $\gr_{(0,1,1)}(\D_{X}[S]) \cdot L_{f}$ will have the generators $\gr_{(0,1,1)}(\delta)$. On the other hand $L_{F}$ is locally generated by $\gr_{(0,1,1)} \left( \delta - \sum s_{k} \frac{\delta \bullet f_{k}}{f_{k}} \right)$ for $\delta \in \Der_{X,\x}(- \Log_{0} f)$. Each such generator has $(0,1,0)$-initial term $\gr_{(0,1,1)}(\delta)$. So $\gr_{(0,1,1)}(\D_{X,\x}[S]) \cdot L_{f,\x} \subseteq \In_{(0,1,0)}L_{F,\x}.$
\end{proof}

\begin{prop}\label{prop main isomorphism}

Suppose $f = f_{1} \cdots f_{r}$ is a strongly Euler-homogeneous divisor and let $F=(f_{1}, \dots, f_{r}).$ Then the following data transfer from the Liouville ideal to the initial ideal of the generalized Liouville ideal:

\begin{enumerate}[(a)]
    \item If $\dim \gr_{(0,1)}(\D_{X,\x}) / L_{f,\x} = n+1$, then \[\dim \gr_{(0,1,1)}(\D_{X,\x}[S]) / L_{F,\x} = n+r+1;\]
    \item If $L_{f}$ is locally a prime ideal, then there is an isomorphism of rings
    \[\frac{\gr_{(0,1,1)}(\D_{X}[S])}{\gr_{(0,1,1)}(\D_{X}[S]) \cdot L_{f}} \simeq \frac{\gr_{(0,1,1)}(\D_{X}[S])}{ \In_{(0,1,0)}L_{F}};\]
    \item If $L_{f}$ is locally Cohen--Macaulay and prime, then $L_{F}$ is locally Cohen--Macaulay.
\end{enumerate}

\end{prop}

\begin{proof}

Because $\gr_{(0,1,1)}(\D_{X}[S]) \cdot L_{f}$ is the extension of $L_{f}$ into a ring with new variables $S$, there are ring isomorphisms
\begin{align*}
    \frac{\gr_{(0,1,1)}(D_{X}[S])}{\gr_{(0,1,1)}(D_{X}[S]) \cdot L_{f}} \simeq \frac{\mathscr{O}_{X}[Y][S]}{\mathscr{O}_{X}[Y][S] \cdot L_{f}} \simeq \frac{\mathscr{O}_{X}[Y]}{L_{f}} [S] \simeq \frac{\gr_{(0,1)}(\D_{X})}{L_{f}}[S].
\end{align*} 
So if $\dim \gr_{(0,1)}(\D_{X,\x}) / L_{f,\x} = n+1$, $\dim \gr_{(0,1,1)}(\D_{X,\x}[S]) / \gr_{(0,1,1)} (\D_{X,\x}[S]) \cdot L_{f} = n+r+1$. Similarly if $L_{f,\x}$ is prime, then $\gr_{(0,1,1)}(\D_{X,\x}[S]) \cdot L_{f,\x}$ is prime. 

The map \eqref{eqn Lf-initialLF} gives $n+r+1 \geq \dim \In_{(0,1,0)}L_{F,\x}.$ By Proposition \ref{main grobner theorem} and Remark \ref{remark grobner justification}, $\dim \In_{(0,1,0)}L_{F,\x} \geq \dim L_{F,\x}$. Proposition \ref{prop dim reference} gives $\dim L_{F,\x} \geq n+r+1$, proving (a). As for (b), the hypotheses guarantee that the map \eqref{eqn Lf-initialLF} is locally a surjection from a domain to a ring of the same dimension and hence an isomorphism. \iffalse A nontrivial kernel makes the local dimension of $\gr_{(0,1,1)}(\D_{X}[S]) / \In_{(0,1,0)}L_{F}$ too small. Hence the map is locally an isomorphism. \fi To prove (c), recall Proposition \ref{main grobner theorem} and Remark \ref{remark grobner justification} show that if $\In_{(0,1,0)} L_{F}$ is locally Cohen--Macaulay, then $L_{F}$ is locally Cohen--Macaulay. So (b) implies (c).
\end{proof}

\iffalse
\todo{make numbering of special equations and list numbering in proofs look different}
\fi
\subsection{Primality of $L_{F,\x}$ and $\widetilde{L_{F,\x}}$.} \text{ }

Now we show that when $f$ is strongly Euler-homogeneous and Saito-holonomic and $F$ a decomposition of $f$, that the conclusions of Proposition \ref{prop main isomorphism} imply $L_{F}$ and $\widetilde{L_{F}}$ are locally prime. The method of argument relies on the Saito-holonomic condition: we use the coordinate transformation of Remark \ref{remark saito-holonomic} to reduce the dimension of logarithmic stratum.

Our first proof mirrors the proof of Theorem 3.17 in \cite{uli}. Because our situation is a little more technical and because we end up using this argument again in Theorem \ref{thm main prime theorem}, we give full details.

\begin{theorem} \label{theorem generalized liouville prime}

Suppose $f = f_{1} \cdots f_{r}$ is strongly Euler-homogeneous and Saito-holonomic and let $F=(f_{1}, \dots, f_{r}).$ If $L_{f}$ is locally Cohen--Macaulay and prime of dimension $n+1$, then $L_{F}$ is locally Cohen--Macaulay and prime of dimension $n+r+1.$ In particular, this happens when $f$ is strongly Euler-homogeneous, Saito-holonomic, and tame. 

\end{theorem}

\begin{proof}

If we prove the second sentence, the third will follow by Theorem 3.17 and Remark 3.18 of \cite{uli}. By Proposition \ref{prop main isomorphism}, the only thing to prove in the second sentence is that $L_{F}$ is locally prime. To do this we induce on the dimension of $X$. If $\dim X$ is 1, then $L_{F,\x} = 0$ and the claim is trivially true.

So we may assume the claim holds for all $X$ with dimension less than $n$. Suppose $\x$ belong to a logarithmic stratum $\sigma$ of dimension $k$. If $k = n$, then by Proposition \ref{prop dim reference} and Remark \ref{remark saito-holonomic}, $L_{F,\x}$ is prime. Now assume $0 < k < n.$ By Remark \ref{remark saito-holonomic}, we can find a coordinate transformation near $\x$ such that each $f_{i} = u_{i}g_{i}$, where $u_{i}$ is a unit near $\x$ and $g_{i}(x_{1}, \dots, x_{n}) = f_{i}(x_{1}, \dots, x_{n-k}, 0, \dots, 0)$, cf. 3.6 of \cite{SaitoTheoryLogarithmic}. By Remark \ref{remark well-behaved}, $L_{F,\x}$ is well-behaved under coordinate transformations and multiplication by units, so we may instead prove the claim for $L_{G,\x}$, where $g = \prod g_{i}$ and $G = (g_{1}, \dots, g_{r}).$ Let $X^{\prime}$ be the space of the first $n-k$ coordinates and $\x{\prime}$ the first $n-k$ coordinates of $\x$. When viewing $g_{i}^{\prime}$ as an element of $\mathscr{O}_{X^{\prime}, \x{\prime}}$, call it $g_{i}^{\prime}.$ Let $g^{\prime} = \prod g_{i}^{\prime}$ and $G^{\prime} = (g_{1}^{\prime}, \dots, g_{r}^{\prime}).$ Because strongly Euler-homogeneous descends from $X$ to $X^{\prime}$, see Remark 2.8 in \cite{uli}, local properties $L_{G^{\prime}}$ do not depend on the choice of the defining equations for the $g_{i}$. Now 
\[\Der_{X,\x}(-\log g) = \mathscr{O}_{X,\x} \cdot \Der_{X^{\prime}, \x{\prime}}(- \log g^{\prime}) + \sum\limits_{1 \leq j \leq k} \mathscr{O}_{X,\x} \cdot \partial_{x_{n-k+j}},\]
where $\partial_{x_{n-k+j}} \in \Der_{X, \x}(-\log_{0} g_{i})$ for each $1 \leq j \leq k$ and $1 \leq i \leq r$. Therefore $\mathscr{O}_{X,\x}[y_{1}, \dots , y_{n}][S] / L_{G, \x} \simeq \mathscr{O}_{X, \x}[y_{1}, \dots , y_{n-k}][S] / L_{G^{\prime}, \x{\prime}}.$ Since Saito-holonomicity descends to $g^{\prime}$, see 3.5 and 3.6 of \cite{SaitoTheoryLogarithmic} and Remark 2.6 of \cite{uli}, we may instead prove the claim for $X^{\prime}$ and $L_{G^{\prime}, \x{\prime}}$. Since $\dim X^{\prime} < \dim X$, the induction hypothesis proves the claim. 

So we may assume $\sigma$ has dimension $0$. By Remark \ref{remark saito-holonomic}, there is some open $U \ni \x$, such that $\x = U \cap \sigma$ and $U \setminus \x$ consists of points whose logarithmic strata are of strictly positive dimension. The discussion above implies $L_{F}$ is prime at all points of  $U \setminus \x$. 

Let $\pi: \Spec \mathscr{O}_{X}[Y][S] \twoheadrightarrow \Spec \mathscr{O}_{X}$ be the map induced by $\mathscr{O}_{X} \hookrightarrow \mathscr{O}_{X}[Y][S].$ If $L_{F}$ is not prime at $\x$, it must have more than one irreducible component that intersects $\pi^{-1}(\x)$. As $L_{F}$ is prime at points of $U \setminus \x$, if $L_{F,x}$ is not prime it must have an ``extra" irreducible component $\V(\mathfrak{q})$ lying inside $\pi^{-1}(\x)$. By assumption, $L_{F,\x}$ is Cohen--Macaulay of dimension $n+r+1$ and so $\V(\mathfrak{q})$ has dimension $n+r+1$. But $\pi^{-1}(\x)$ has dimension $n+r$. Thus $L_{F,\x}$ is prime completing the induction.
\end{proof}

\begin{prop} \label{prop CM dim 2}
Suppose $f = f_{1} \cdots f_{r}$ is a strongly Euler-homogeneous divisor and let $F=(f_{1}, \dots, f_{r}).$ If $L_{F}$ is locally prime, Cohen--Macaulay, and of dimension $n+r+1$, then $\widetilde{L_{F}}$ is locally Cohen--Macaulay of dimension $n+r$. 
In particular, this happens when $f$ is strongly Euler-homogeneous, Saito-holonomic, and tame.
\end{prop}

\begin{proof}
Let $E_{\x}$ be a strong Euler-homogeneity and consider $\gr_{(0,1,1)}(\psi_{F}(E_{\x}))$, which is $(0,1,1)$-homogeneous of degree $1$. The generalized Liouville ideal is generated by the elements $\psi_{F}(\Der_{X,\x}(-\log_{0} f)$. If $\gr_{(0,1,1)}(\psi_{F}(E_{\x})) \in L_{F,\x}$, then $\psi_{F}(E_{\x}) \in \psi_{F}(\Der_{X,\x}(-\log_{0} f)$. This is impossible since $E_{\x} \notin \Der_{X,\x}(- \log_{0} f)$.

Locally, $\gr_{(0,1,1)}(\D_{X,\x}[S]) / \widetilde{L_{F,\x}}$ is obtained from $\gr_{(0,1,1)}(\D_{X,\x}[S]) / L_{F,\x}$ by modding out by a non-zero element, which must be regular. So $\widetilde{L_{F}}$ is locally Cohen--Macaulay of dimension at most $n+r$. That locally the dimension $\widetilde{L_{F}}$ is $n+r$ follows from the dimension inequality in Proposition \ref{prop dim reference}.

The final sentence is true by Theorem \ref{theorem generalized liouville prime}.
\end{proof}

This section's first main result is that $\widetilde{L_{F,\x}}$ is locally prime when $f$ is strongly Euler-homogeneous, Saito-holonomic, and tame. The strategy is the same as in Theorem \ref{theorem generalized liouville prime}. Under much stricter hypotheses, and in his language, Maisonobe shows in Proposition 3 of \cite{MaisonobeArrangement}, that $\widetilde{L_{F}}$ is locally prime. Experts will note that we recycle the part of his argument where he reduced dimension in our proof. 

\iffalse \cite{MaisonobeArrangement} Maisonobe, in our notation, argues, under stricter hypotheses, that $\widetilde{L_{F,\x}}$ is locally prime. We show that the final trick of Proposition 3 also applies in our situation. (Morally, Maisonobe nontrivially cuts some variety with a hyperplane to reduce its dimension--we show that we can do the same cut.)
\fi

\begin{theorem} \label{thm main prime theorem} 
Assume that $f = f_{1} \cdots f_{r}$ is strongly Euler-homogeneous and Saito-holonomic and let $F=(f_{1}, \dots, f_{r}).$ If $\widetilde{L_{F}}$ is locally Cohen--Macaulay of dimension n + r, then $\widetilde{L_{F}}$ is locally prime. In particular, $\widetilde{L_{F}}$ is locally prime, Cohen--Macaulay, and of dimension $n+r$ when $f$ is strongly Euler-homogeneous, Saito-holonomic, and tame.
\end{theorem}

\begin{proof}

%We will prove $\widetilde{L_{F,\x}}$ is prime by inducting on the dimension of the log strata of $x$. If said strata has dimension $n = \text{dim } X$, then by \ref{remark saito-holonomic} (e) and \ref{prop dim reference} the claim is proved. So we may assume the claim holds for all points whose log strata have dimension greater than $k$. Let $D$ be the log strata of $x$.

By Proposition \ref{prop CM dim 2}, it suffices to prove the first claim. The proof follows the inductive argument of Theorem \ref{theorem generalized liouville prime} with a slight modification at the end. 

If $\dim X$ is 1, then $\widetilde{L_{F,\x}}$ is generated by $\psi_{F}(E_{\x})$, $E_{\x}$ a strong Euler-homogeneity. By Lemma \ref{lemma Euler calculation}, $\mathscr{O}_{X,\x}[Y][S]/\widetilde{L_{F,\x}} \simeq \mathscr{O}_{X,\x}[Y][s_{1}, \dots s_{j-1}, s_{j+1}, \dots s_{r}].$

Now assume the claim holds for all $X$ with dimension less than $n$ and $\x$ belongs to a logarithmic stratum $\sigma$ of dimension $k$. If $k = n$, then $\widetilde{L_{F,\x}}$ is prime by Proposition \ref{prop dim reference}. If $0 < k < n$ we can make the same coordinate transformation as in Theorem \ref{theorem generalized liouville prime} and instead prove $\widetilde{L_{G}}$ is locally prime where $g_{i}(x) = f_{i}(x_{1}, \dots, x_{n-k}, 0, \dots, 0).$ Using the notation of Theorem \ref{theorem generalized liouville prime}, $X^{\prime}$ is  strongly Euler-homogeneous and Saito-holonomic and 
\[\Der_{X,\x}(-\log g) = \mathscr{O}_{X,\x} \cdot \Der_{X^{\prime}, \x{\prime}}(- \log g^{\prime}) + \sum\limits_{1 \leq j \leq k} \mathscr{O}_{X,\x} \cdot \partial_{x_{n-k+j}},\]
where $\partial_{x_{n-k+j}} \in \Der_{X,\x}(-\log_{0} g_{i})$ for each $1 \leq j \leq k$ and $1 \leq i \leq r.$ Moreover, the strong Euler-homogeneity $E_{\x{\prime}}$ for $g^{\prime}$ at $\x{\prime} \in X^{\prime}$ can be viewed as a strong Euler-homogeneity for $g$ at $\x \in X.$ Therefore $\mathscr{O}_{X,\x}[y_{1}, \dots, y_{n}][S] / \widetilde{L_{G,\x}} \simeq \mathscr{O}_{X,\x}[y_{1}, \dots, y_{n-k}][S] / \widetilde{L_{G^{\prime}, \x{\prime}}}.$ Since $\dim X^{\prime} < \dim X$, the induction hypothesis shows that $\widetilde{L_{G^{\prime}, \x{\prime}}}$ is prime. 

So we may assume $\sigma$ has dimension 0. Let $\pi: \Spec \mathscr{O}_{X,\x}[Y][S] \twoheadrightarrow \Spec \mathscr{O}_{X}$ be the map induced by $\mathscr{O}_{X} \hookrightarrow \mathscr{O}_{X,\x}[Y][S].$ Reasoning as in Theorem \ref{theorem generalized liouville prime}, we deduce that if $\widetilde{L_{F,\x}}$ is not prime then there exists a irreducible component $\V(\mathfrak{q})$ of $\widetilde{L_{F,\x}}$ contained entirely in $\pi^{-1}(\x).$

\iffalse
Since $\widetilde{L_{F,\x}}$ is generated by $(0,1,1)$-homogeneous elements of degree $1$, we can use Lemma \ref{lemma graded connected spec} and 
\fi

By assumption, $\widetilde{L_{F,\x}}$ is Cohen--Macaulay of dimension $n+r$ and $\V(\mathfrak{q})$ has dimension $n+r$. Let $E_{\x}$ be the strong Euler-homogeneity at $\x$. Then \iffalse $\gr_{(0,1,1)}(\psi_{F}(E_{\x})) \in \widetilde{L_{F,\x}} \subseteq \mathfrak{q}$ and \fi $\V(\mathfrak{q}) \subseteq \pi^{-1}(\x) \cap \V(\gr_{(0,1,1)}(\psi_{F}(E_{\x}))).$ We will show that the intersection of  $\V(\gr_{(0,1,1)}(\psi_{F}(E_{\x})))$ and $\pi^{-1}(\x)$ is proper; since the dimension of $\pi^{-1}(\x)$ is $n+r$ this will show that $\V(\mathfrak{q})$, which we know is of dimension $n+r$, is contained in a closed set of strictly smaller dimension. Therefore no such $\mathfrak{q}$ exists and $\widetilde{L_{F,\x}}$ is prime. 

Recall $\gr_{(0,1,1)}(\psi_{F}(E_{\x})) = \gr_{(0,1,1)}(E_{\x}) - \sum\limits \frac{E_{\x} \bullet f_{k}}{f_{k}} s_{k}.$ Lemma \ref{lemma Euler calculation} proves that there exists an index $j$ such that $\frac{E_{\x} \bullet f_{j}}{f_{j}} \notin \mathfrak{m}_{\x}$. So there is a closed point in $\pi^{-1}(\x)$ that does not lie in $\V (\gr_{(0,1,1)}(\psi_{F}(E_{\x})))$. In particular, the intersection of $\V(\gr_{(0,1,1)}(\psi_{F}(E_{\x})))$ and $\pi^{-1}(\x)$ is proper and the inductive step is complete.
\iffalse  So the maximal ideal generated by $\mathfrak{m}_{\x}$, $y_{1}, \dots y_{n}$, and $s_{1}, \dots , s_{j-1}, s_{j+1}, \dots s_{r}$ does not contain $\gr_{(0,1,1)}(\psi_{F}(E_{\x}))$ and there is a closed point in $\pi^{-1}(\x)$ that does lie in $\V(\gr_{(0,1,1)}(\psi_{F}(E_{\x})))$. Therefore $\V(\gr_{(0,1,1)}(\psi_{F}(E_{\x}))) \not\subseteq \pi^{-1}(\x)$ and the intersection of $\V(\gr_{(0,1,1)}(\psi_{F}(E_{\x})))$ and $\pi^{-1}(\x)$ is proper. Hence $\widetilde{L_{F,\x}}$ is prime. This completes the inductive step and finishes the claim. 
\fi
\end{proof}

\subsection{The $\D_{X}[S]$-annihilator of $F^{S}$.} \text{ }

Let $\Jac(f)$ be the Jacobian ideal of $f$. In a given coordinate system, there is a natural $\mathscr{O}_{X,\x}$-linear map 
\[ \phi_{f}: \gr_{(0,1,1)}(\D_{X,\x}[s]) \twoheadrightarrow \text{Sym}_{\mathscr{O}_{X,\x}}(\Jac(f)) \twoheadrightarrow R(\Jac(f)) \] given by \[s \mapsto f t \text{ and } \gr_{(0,1,1)}(\partial_{x_{k}}) \mapsto (\partial_{x_{k}} \bullet f ) t.\] Its kernel contains $\gr_{(0,1)}(\ann_{\D_{X,\x}[s]} f^{s}).$ (See Section 1.3 in  \cite{MorenoMacarroLogarithmic} for details.) So we have the containments
\[ L_{f,\x} + \gr_{(0,1,1)}(\D_{X,\x}[s]) \cdot \gr_{(0,1,1)}(E_{\x} - s) \subseteq \gr_{(0,1)}(\ann_{\D_{X,\x}[s]} f^{s}) \subseteq \text{ker}(\phi_{f}) \]
and equality will hold throughout if $L_{f,\x} + \gr_{(0,1,1)}(\D_{X,x}[s]) \cdot \gr_{(0,1,1)}(E_{\x}-s)$ agrees with $\text{ker}(\phi_{f}).$

This motivates our analysis of $\ann_{\D_{X}[S]} F^{S}$: we will construct a map $\phi_{F}$ from $\gr_{(0,1,1)}(\D_{X}[S])$ into a Rees-algebra like object and squeeze $\gr_{(0,1,1)}(\ann_{\D_{X}[S]} F^{S})$ between $\widetilde{L_{F}}$ and $\text{ker}(\phi_{F}).$

\begin{definition} \label{def phiF}
Let $\Jac(f_{i})$ be the Jacobian ideal of $f_{i}$ and consider the multi-Rees algebra $R(\Jac(f_{1}), \dots, \Jac(f_{r}))$ associated to these $r$ Jacobian ideals. Consider the $\mathscr{O}_{X,\x}$-linear map \[
\phi_{F}: \gr_{(0,1,1)}(\D_{X}[S]) \to R(\Jac(f_{1}), \dots, \Jac(f_{r})) \subseteq \mathscr{O}_{X}[S]
\]
given, having fixed local coordinates on $U$, by
$$ y_{i} \mapsto \sum\limits_{k}  \frac{f}{f_{k}} (\partial_{x_{i}} \bullet f_{k}) s_{k} \text{ and } s_{k} \mapsto f s_{k}.$$

\end{definition}

\begin{prop} \label{prop phi containment}
Let $f = f_{1} \cdots f_{r}$ and $F = (f_{1}, \dots, f_{r})$. Then\[\gr_{(0,1,1)}(\ann_{\D_{X}[S]} F^{S}) \subseteq \text{\normalfont ker}(\phi_{F}).\]
\end{prop}

\begin{proof}

It is enough to show this locally, so take $P \in \ann_{\D_{X,\x}[S]}F^{S}$ of order $\ell$ under the $(0,1,1)$-filtration. \iffalse 
Write 
\[P = \sum\limits_{|\textbf{u}| + |\textbf{v}| = \ell} g_{\textbf{u},\textbf{v}}\partial^{\textbf{u}}S^{\textbf{v}} + G,\]
with $\textbf{u}$, $\textbf{v}$ nonnegative vectors, $g_{\textbf{u}, \textbf{v}} \in \mathscr{O}_{X,\x}$, and $G$ of lower $(0,1,1)$-order than $P$.
\fi For any $Q$ of order $\ell$ it is always true that $f^{\ell}Q \bullet F^{S} \in \mathscr{O}_{X,\x}[S]F^{S}$. Any time a partial is applied to $g F^{S}$, a $s$-term only comes out of the product rule when the partial is applied to $F^{S}$. A straightforward calculation shows that the $S$-lead term of $f^{\ell} P F^{S}$ is exactly $\phi_{F} (\gr_{(0,1,1)}(P)) F^{S}.$ Since $f^{\ell}P$ annihilates $F^{S}$, we conclude $\gr_{(0,1,1)}(P) \in \text{ker}(\phi_{F}).$
\end{proof}

\begin{prop} \label{prop phi prime}

$\text{\normalfont ker}(\phi_{F})$ is a prime ideal of dimension $n+r$. 

\end{prop}

\begin{proof}

It is prime. Since Rees rings are domains, to count dimension we squeeze $\phi_{F}(\gr_{(0,1,1)}(\D_{X}[S]))$ between two well-behaved multi-Rees algebras: $R((f), \dots, (f))$ and $R(\Jac(f_{1}), \dots, \Jac(f_{r}))$ (the first multi-Rees algebra is built using $r$ copies of $(f)$). As the latter is the target of $\phi_{F}$ and $\phi_{F}(s_{i}) = f s_{i}$ this is easy:
$$ R((f), \dots, (f)) \subseteq \phi_{F}(\gr_{(0,1,1)}\D_{X}[S]) \subseteq R(\Jac(f_{1}), \dots, \Jac(f_{r}))$$
Moreover, the dimension of a multi-Rees algebra is well known: $R(I_{1}, \dots , I_{r}) = r + $ the dimension of the ground ring. \iffalse \todo{add multi-Rees dimension reference?} \fi

So $\phi_{F}(\gr_{(0,1,1)}(\D_{X,\x}[S]))$ is a domain squeezed between subrings of $\mathscr{O}_{X,\x}[S]$ of dimension $n+r$. The result then follows by the following lemma:

\begin{lemma} Let $ R \subseteq A \subseteq B \subseteq C \subseteq R[X]$ be finitely generated, graded $R$-algebras, whose gradings are inherited from the standard grading on $R[X].$ Assume that $R$ is a universally caternary Notherian domain. If $\dim A = \dim C$, then $ \dim A = \dim B = \dim C.$
\end{lemma}

\noindent \emph{Proof.} Claim: if $\mathfrak{m}^{\star}$ is a graded maximal ideal of $A$, then $\mathfrak{m}^{\star} B \neq B$. We prove the contrapositive. So assume $\mathfrak{m}^{\star}B = B$. Then $\mathfrak{m}^{\star} R[X] = R[X].$ Write $\mathfrak{m}^{\star} = (a_{1}, \dots, a_{\ell})$ in terms of homogeneous generators $a_{i} \in A$ and find $r_{1}, \dots, r_{n}$ in $R[X]$ such that $1 = \sum r_{i}a_{i}$. Since the degree of $1$ is zero, we can assume either $r_{i}$ and $a_{i}$ are both degree 0 or $r_{i} = 0$. Thus $1 = \sum r_{i}a_{i}$ occurs in $\mathfrak{m}^{\star} \cap R$ and so $\mathfrak{m}^{\star} = A$, a contradiction.

Now we argue using a version of Nagata's Altitude Formula (see \cite{EisenbudCommutative} Theorem 13.8): $\dim(B_{\mathfrak{q}}) = \dim(A_{\mathfrak{p}}) + \dim (Q(A) \otimes_{A} B)$, for $\mathfrak{q} \in \Spec B$ maximal with respect to the property $\mathfrak{q} \cap A = \mathfrak{p}$. Since $B$ is a finitely generated $A$-algebra, and tensors are right exact, $Q(A) \otimes_{A} B$ is a finitely generated $Q(A)$-algebra. Thus $\dim(Q(A) \otimes_{A} B) = \text{trdeg}_{Q(A)}(Q(A) \otimes_{A} B) = \text{trdeg}_{Q(A)}Q(Q(A) \otimes_{A} B) = \text{trdeg}_{Q(A)}{Q(B)} = \text{trdeg}_{A}B$. Similar statements hold for the other pairs $ A \subseteq C$ and $B \subseteq C.$

Let $\mathfrak{m} \in \Spec A$ such that $\dim(A_{\mathfrak{m}}) = \dim(A)$. By the claim in the first paragraph (so assuming $\mathfrak{m}$ is graded if necessary), we can find $\mathfrak{q} \in \Spec C$ maximal with respect to the property $\mathfrak{q} \cap A = \mathfrak{m}$. So $\dim(C_{\mathfrak{q}}) = \dim(A_{\mathfrak{m}}) + \text{trdeg}_{A}C$. Therefore $\dim(C) \geq \dim(A) + \text{trdeg}_{A}C$ and hence $\text{trdeg}_{A}C = 0$. Since we are looking at algebras finitely generated over the appropriate subring, transcendence degree is additive. So $0 = \text{trdeg}_{A}B$ and $0 = \text{trdeg}_{B}C$.

Let $\mathfrak{m} \in \Spec A$ with $\dim(A_{\mathfrak{m}}) = \dim(A)$, as before. Again, using the claim, select $\mathfrak{p} \in \Spec B$ maximal with respect to the property $\mathfrak{p} \cap A = \mathfrak{m}$. So $\dim(B_{\mathfrak{p}}) = \dim(A_{\mathfrak{m}}) + \text{trdeg}_{A}B$; hence $\dim(B) \geq \dim(A)$. Argue similarly for $B \subseteq C$ to determine $\dim(C) \geq \dim(B)$. This ends the proof. \end{proof}

The following is an analogous statement to Corollary 3.23 in \cite{uli}:

\begin{cor} \label{cor gr-equality}
There is the containment
\[
\widetilde{L_{F}} \subseteq \gr_{(0,1,1)}(\ann_{\D_{X}[S]} F^{S}) \subseteq \text{\normalfont ker}(\phi_{F}).
\]
If $f$ is strongly Euler-homogeneous, Saito-holonomic, and tame then all three ideals are equal.  
\end{cor}

\begin{proof}
The containments follow from the construction of $\widetilde{L_{F}}$ and Proposition \ref{prop phi containment}. They are equalities when $f$ is suitably nice because, by Theorem \ref{thm main prime theorem} and Proposition \ref{prop phi prime}, at each $\x \in X$ the outer ideals are prime of the same dimension.
\end{proof}

Because $\D_{X,\x}[S] \cdot \theta_{F} \subseteq \ann_{\D_{X,\x}[S]} F^{S}$, we can use Corollary \ref{cor gr-equality} and a type of Gr\"obner basis argument to prove:

\iffalse We know $\D_{X,\x}[S] \cdot \theta_{F} \subseteq \ann_{\D_{X,\x}[S]}F^{S}$ and we have just shown they have the same initial terms under the $(0,1,1)$-filtration. By the standard strategy of peeling off initial terms of $\ann_{\D_{X,\x}[S]}F^{S}$ using the appropriate element of $\D_{X,\x}[S] \cdot \theta_{F}$ we can prove: \fi

\begin{theorem} \label{thm gen by derivations}
If $f = f_{1} \cdots f_{r}$ is strongly Euler-homogeneous, Saito-holonomic, and tame and if $F=(f_{1}, \dots, f_{r})$, then the $\D_{X}[S]$-annihilator of $F^{S}$ is generated by derivations, that is
$$ \ann_{\D_{X}[S]} F^{S} = \D_{X}[S] \cdot \theta_{F}.$$
\end{theorem}

\begin{proof}
Take $P \in \ann_{\D_{X}[S]} F^{S}$ of order $k$ under the total order filtration. By Corollary \ref{cor gr-equality}, there exist $L_{1}, \dots, L_{k} \in \theta_{F}$, $n_{1}, \dots, n_{k} \in \mathscr{O}_{X}[Y][S]$ such that 
\[\gr_{(0,1,1)}(P) = \sum n_{i} \cdot \gr_{(0,1,1)}(L_{i}).\]
Since $\gr_{(0,1,1)}(P)$ is homogeneous of degree $k$ and $\gr_{(0,1,1)}(L_{i})$ is homogeneous of degree $1$, we may assume the $n_{i}$ are homogeneous. For each $i$ select $N_{i} \in \D_{X}[S]$ such that $n_{i} = \gr_{(0,1,1)}(N_{i}).$ Consequently, $P - \sum N_{i} \cdot L_{i}$ has order (under the total order filtration) less than $k$ and lies in $\ann_{\D_{X}[S]}F^{S}$. Since $\mathscr{O}_{X}[S] \cap \ann_{\D_{X}[S]}F^{S} = 0$, an induction argument shows $P \in \D_{X}[S] \cdot \theta_{F}.$ \iffalse Note that this means $k$ is the smallest integer such that $P \in F_{(0,1,1)}^{k}$. (Recall $F_{(0,1,1)}$ is the filtration on $\D_{X,\x}[S]$ making every element of $\mathscr{O}_{X,\x}$ weight $0$ and giving each $\partial$ and each $s_{i}$ weight $1$.) Since each $\gr_{(0,1,1)}(L_{i})$ is $(0,1,1)$-homogeneous of degree $1$, we may assume the $n_{i}$ are all $(0,1,1)$-homogeneous of the same degree. In particular, for each $i$, there exists an $N_{i} \in \D_{X}[S]$ such that $n_{i} = \gr_{(0,1,1)}(N_{i}).$ Therefore, $P - \sum N_{i} \cdot L_{i} \in F_{(0,1,1)}^{k-1}.$ Since $P - \sum N_{i} \cdot L_{i} \in \ann_{\D_{X}[S]}F^{S}$ and $\ann_{\D_{X}[S]}F^{S}$ contains no element of $F_{(0,1,1)}^{0}= \mathscr{O}_{X}$, an induction argument shows $P \in \D_{X}[S] \cdot \theta_{F}.$ \fi
\end{proof}

\begin{cor} \label{cor algebraic category}

Let $f= f_{1} \cdots f_{r} \in \mathbb{C}[x_{1}, \dots, x_{n}]$, where each $f_{k} \in \mathbb{C}[x_{1}, \dots, x_{n}],$ and let $F = (f_{1}, \dots, f_{r}).$ If $f$ is strongly Euler-homogeneous, Saito-holonomic, and tame, then the $\D_{X}[S]$-annihilator of $F^{S}$ is generated by derivations, that is
\[ \ann_{\D_{X}[S]} F^{S} = \D_{X}[S] \cdot \theta_{F}. \]
More generally, if $X$ is the analytic space associated to a smooth $\mathbb{C}$-scheme and if $f$ and $F=(f_{1}, \dots, f_{r})$ are algebraic, then the conclusion of Theorem \ref{thm gen by derivations} holds in the algebraic category.

\iffalse 
Let $X$ be the the analytic space associated to a smooth $\mathbb{C}$-scheme and let $f \in \mathscr{O}_{X}$ be strongly Euler-homogeneous, Saito-holonomic, and tame. Suppose $F$ is a decomposition of $f$. Then 
\[ \ann_{\D_{X}[S]} F^{S} = \D_{X}[S] \cdot \theta_{F} \]
and this also holds in the algebraic category.
\fi

\end{cor}

\begin{proof}
This follows from Theorem \ref{thm gen by derivations} and the fact algebraic functions have algebraic derivatives and hence algebraic syzygies. See Theorem 3.26 and Remark 2.11 in \cite{uli} for more details. 
\end{proof}

\subsection{Comparing Different Factorizations of $f$} \text{ }

\begin{definition}
Consider the functional equation 
\[ b_{f,\x}(s) f^{s} = P f^{s+1} \]
where $b_{f,\x}(s) \in \mathbb{C}[s]$ and $P \in \D_{X,\x}[s].$ Let $B_{f,\x}$ be the ideal in $\mathbb{C}[s]$ generated by all such $b_{f,\x}(s)$, that is the ideal generated by the \emph{Bernstein--Sato polynomial}. We may write $B_{f,\x} = (\D_{X,x}[s] \cdot f + \ann_{\D_{X,\x}[s]}f^{s}) \cap \mathbb{C}[s]$. Then the variety $\V(B_{f,\x})$ consists of the roots of the Bernstein--Sato polynomial. 

In the multivariate situation we may consider functional equations of the form
\[b_{F,\x}(S) F^{S} = P F^{S+1} \]
where $b_{F,\x}(S) \in \mathbb{C}[S]$ and $P \in \D_{X,\x}[S].$ Just as above, the set of all such $b_{F,\x}(S)$ form an ideal
$B_{F,\x} = (\D_{X,x}[S] \cdot f + \ann_{\D_{X,\x}[S]}F^{S}) \cap \mathbb{C}[S].$ The variety $\V(B_{F,\x})$ is called the \emph{Bernstein--Sato variety} of $F$.
\end{definition}

It would be interesting to compare $\V(B_{F,\x})$ and $\V(B_{G,\x})$ where $F$ and $G$ correspond to two different factorizations of $f$. The following is a particular case of Lemma 4.20 of \cite{BudurBernstein--Sato}: 

\begin{prop} \label{prop diagonal embedding} \normalfont{(Lemma 4.20 of \cite{BudurBernstein--Sato})}
Suppose that $f = f_{1} \cdots f_{r}$ is strongly Euler-homogeneous, Saito-holonomic, and tame. Let $\mathbb{C}[S] = \mathbb{C}[s_{1}, \dots, s_{r}]$, $F = (f_{1}, \dots, f_{r})$, and $G = (f_{1}, \dots, f_{r-2}, f_{r-1} f_{r})$. Then
\[ 
B_{F,\x} + \mathbb{C}[S] \cdot (s_{r-1} - s_{r}) \subseteq \mathbb{C}[S] \cdot B_{G,\x} + \mathbb{C}[S] \cdot (s_{r-1} - s_{r}). 
\]
In particular, let $\Delta: \mathbb{C} \mapsto \mathbb{C}^{r}$ be the diagonal embedding. Then $\Delta(\V(B_{f,\x})) \subseteq \V(B_{F,\x}).$
\end{prop}

Under the hypotheses of Theorem \ref{thm gen by derivations}, on the level of annihilators we obtain a more precise statement:

\iffalse Under our working hypotheses of strongly Euler-homogeneous, Saito-holonomic, and tame divisors, by Theorem \ref{thm gen by derivations} the annihilator of $F^{S}$ is generated by derivations. Given $f = f_{1} \cdots f_{r}$, this will let us compare $V(B_{F,\x})$ and $V(B_{G,\x})$ where $ F = (f_{1}, \dots, f_{r})$ and $G= (f_{1}, \dots, f_{r-2}, f_{r-1} f_{r}).$ Setting $G = (f)$ lets us compare $\V(B_{F,\x})$ and $\V(B_{f,\x})$. 
\fi

\begin{prop} \label{prop going down}
Suppose $f = f_{1} \cdots f_{r} $ is strongly Euler-homogeneous, Saito-holonomic, and tame. Let $F = (f_{1}, \dots, f_{r})$ and $G = (f_{1},, \dots,f_{r-2}, f_{r-1}f_{r})$. Then there is an isomorphism of rings:
\[
\frac{\D_{X}[s_{1}, \dots, s_{r}]}{\ann_{\D_{X}[s_{1}, \dots, s_{r}]}\D_{X}F^{S} + (s_{r-1}-s_{r})} \simeq \frac{\D_{X}[s_{1}, \dots, s_{r-1}]}{\ann_{\D_{X}[s_{1}, \dots, s_{r-1}]}\D_{X}[s_{1}, \dots, s_{r-1}]G^{S}}.
\]

\end{prop}

\begin{proof}
This follows from Theorem \ref{thm gen by derivations}, the definition of $\psi_{F,\x}(\delta)$ for $\delta$ a logarithmic derivation, and a straightforward computation using the product rule.
\end{proof}

\begin{remark} \label{remark compare complicated decompositions}
Let $F = (f_{1}, \dots, f_{r})$ correspond to a factorization of $f$ where $f$ is strongly Euler-homogeneous, Saito-holonomic, and tame. For $a \in \mathbb{C}$, $\D_{X,\x}[S] \cdot (s_{1} - a, \dots, s_{r} - a) = \D_{X,\x}[S] \cdot (s_{1} - s_{2}, \dots, s_{r-1} - s_{r}, s_{r} - a).$ By Proposition \ref{prop going down}, there is a ring isomorphism $\D_{X,\x}[S] F^{S} / (s_{1} - a, \dots, s_{r} - a) \cdot \D_{X,\x}[S] F^{S} \simeq \D_{X,\x}[s] f^{s} / (s-a) \cdot \D_{X,\x}[s] f^{s}.$ Using this fact we propose in Remark \ref{remark basic nabla facts} a more precise way to analyze the diagonal embedding of Proposition \ref{prop diagonal embedding}.
\iffalse
Let $f = f_{1} \cdots f_{r}$ and $F = (f_{1}, \dots, f_{r})$ the corresponding decomposition of $f$. Let $g_{1} = f_{1} \dots f_{i_{1}}, \dots, g_{t} = f_{i_{t-1}} \dots f_{r}$ and $G = (g_{1}, \dots, g_{t})$ the corresponding decomposition of $f$. It is straightforward to argue as in Proposition \ref{prop diagonal embedding} and compare the varieties of $\V(B_{F,\x})$ and $\V(B_{G,\x}).$
\fi
\end{remark}

\subsection{Hyperplane Arrangements.} \text{ }

Finally let us turn to the algebraic setting and particular to \textit{central hyperplane arrangements} $\mathcal{A} \subseteq \mathbb{C}^{n} = X$ whose defining equations are given by $f_{\mathcal{A}} = \prod L_{i}$, where the $L_{i} \in \mathbb{C}[x_{1}, \dots, x_{n}]$ are homogeneous polynomials of degree 1.  A central hyperplane arrangement is \textit{indecomposable} if there is no choice of coordinates $t_{1} \sqcup t_{2} $, $t_{1}$ and $t_{2}$ disjoint, such that $f_{\mathcal{A}} = g_{1}(t_{1}) g_{2}(t_{2})$. Central hyperplane arrangements are strongly Euler-homogeneous and Saito-holonomic, cf. examples \ref{example hyperplane strongly Euler-homogeneous}, \ref{example hyperplane arrangement saitoholonomic}.

Write $D_{n}$ for the $n^{\text{th}}$ Weyl Algebra $\mathbb{C}[x_{1}, \dots, x_{n}, \partial_{1}, \dots , \partial_{n}].$ Let $F = (f_{1}, \dots, f_{r})$ be some decomposition of $f_{\mathcal{A}}$ into factors. Construct the $D_{n}[s]$-module ($D_{n}[S]$-module) generated by the symbol $f^{s}$ ($F^{S}$) in an entirely similar way as in the analytic setting.  Furthermore, define the roots of the Bernstein--Sato polynomial $B_{f}$ and the Bernstein--Sato variety $B_{F}$ just as before. For an algebraic $f$ equipped with an algebraic decomposition $F$, $B_{f}$ and $B_{F}$ agree with the analytic versions because algebraic functions have algebraic derivatives and syzygies.

In \cite{BudurBernstein--Sato}, Budur makes the following conjecture: 

\begin{conjecture} \label{conjecture Budur indecomposable} \text{\normalfont (Conjecture 3 in \cite{BudurBernstein--Sato})}
Let $\mathcal{A}$ be a central, essential, indecomposable hyperplane arrangement. Factor $f_{\mathcal{A}} = f_{1} \cdots f_{r}$, where each factor $f_{k}$ is of degree $d_{k}$ and the $f_{k}$ are not necessarily reduced, and let $F = (f_{1}, \dots, f_{r})$. Then 
\[ \{ d_{1} s_{1} + \dots + d_{r} s_{r} + n = 0 \} \subseteq \V(B_{F}). 
\]
\end{conjecture}

\noindent This conjecture is related to the Topological Multivariable Strong Monodromy Conjecture, see Conjecture \ref{conjecture topological multi strong}, for hyperplane arrangements, which claims that the polar locus of the topological zeta function of $F = (f_{1}, \dots, f_{r})$ is contained in $\V(B_{F,0})$. In Theorem 8 of loc. cit. Budur proves Conjecture \ref{conjecture Budur indecomposable} implies the Topological Multivariable Strong Monodromy Conjecture for hyperplane arrangements. See \cite{BudurBernstein--Sato}, in particular subsection 1.3 and Theorem 8, for details.

Walther proves in Theorem 5.13 of \cite{uli} the $r=1$ version of this conjecture: if $f$ is a tame and indecomposable central hyperplane arrangement of degree $d$, then $-n/d \in \V(B_{f}).$ Analogously, we prove Conjecture \ref{conjecture Budur indecomposable} in the tame case: \iffalse We prove the corresponding statement in the $F^{S}$ caseWhen $f_{\mathcal{A}}$ is tame, $\ann_{D_{n}[s]} f^{s}$ and $\ann_{D_{n}[S]} F^{S}$ are both generated by $\Der_{X}(-\log f)$ in similar ways; hence we can argue similarly and prove a weaker version of Conjecture \ref{conjecture Budur indecomposable}. \fi

\begin{theorem} \label{thm indecomposable}

Suppose $f_{\mathcal{A}}$ is a central, essential, indecomposable, and tame hyperplane arrangement. Let $F = (f_{1}, \dots, f_{r})$ be a decomposition of $f_{\mathcal{A}}$ where $f_{k}$ has degree $d_{k}$ and the $f_{k}$ are not necessarily reduced. Then 
\[ \{ d_{1} s_{1} + \dots + d_{r}s_{r} + n = 0 \} \subseteq \V(B_{F}). 
\]

\end{theorem}

\begin{proof}
Since $f_{\mathcal{A}}$ is homogeneous, $\Der_{X}(- \log f)$ is a graded $\mathbb{C}[X]$-module after giving each $x_{i}$ degree one and each $\partial_{i}$ degree -1. In the proof of Theorem 5.13 of \cite{uli}, Walther shows that the indecomposablity hypothesis implies there exists a system of coordinates such that $\delta \in \Der_{X}(-\log f)$ is homogeneous of positive total degree or $\delta = w \sum x_{i} \partial_{i}$, $w \in \mathbb{C}.$ Fix this system of coordinates and $E = \sum x_{i} \partial_{i}$ for the rest of the proof.

By Corollary \ref{cor algebraic category}, $\ann_{D_{n}[S]} F^{S}   = D_{n}[S] \cdot \psi_{F}(\Der_{X}(-\log f)).$ Recall $\psi_{F}(\delta) = \delta - \sum \frac{ \delta \bullet f_{k}}{f_{k}} s_{k}.$ If $\delta$ is of positive $(1,-1)$ total degree, then the coefficient of each $s_{k}$ is either 0 or of positive degree as polynomial in $\mathbb{C}[x_{1}, \dots, x_{n}]$. This shows $\psi_{F}(\delta) \in D_{n}[S] \cdot (X)$, where $D_{n}[S] \cdot (X)$ is the left ideal generated by $x_{1}, \dots, x_{n}.$ Because $E + n \in D_{n} \cdot (X)$,  

\begin{align*}
    \ann_{D_{n}[S]}F^{S} + D_{n}[S] \cdot f & \subseteq D_{n}[S] \cdot (X) + D_{n}[S] \cdot \psi_{F}(E) \\
    & = D_{n}[S] \cdot (X) + D_{n}[S] \cdot (E - \sum d_{k} s_{k}) \\
    & = D_{n}[S] \cdot (X) + D_{n}[S] \cdot (-n - \sum d_{k} s_{k}) .
\end{align*}

Suppose $P(S)$ is in the intersection of $D_{n}[S] \cdot (X) + D_{n}[S] \cdot (-n - \sum d_{k} s_{k})$ and $\mathbb{C}[S]$. For each root $\alpha$ of $-n - \sum d_{k} s_{k}$ there is a natural evaluation map $D_{n}[S] \mapsto D_{n}$ sending $P \mapsto P(\alpha) \in D_{n} \cdot (X).$ Since $D_{n} \cdot (X)$ is a proper ideal of $D_{n}$, $P(\alpha) = 0$ for all such $\alpha$. Therefore $\V(P(S)) \supseteq \V(\mathbb{C}[S] \cdot (-n - \sum d_{k} s_{k}))$ and we have shown
\[ 
\V(B_{F}) = \V((\ann_{D_{n}[S]} F^{S} + D_{n}[S]  \cdot f) \cap \mathbb{C}[S]) \supseteq \V(-n - \sum d_{k} s_{k}).
\]
\end{proof}

\iffalse
\begin{remark} \label{remark indecomposable remark}
\begin{enumerate}[(1)]
    \item In Theorem \ref{thm indecomposable}, we only needed tameness to ensure the annihilator of $F^{S}$ is generated by derivations and we only needed indecomposablity to insure $\Der_{X}(-\log_{0} f) \subseteq (X)^{2} \cdot \Der_{X}.$ So Theorem \ref{thm indecomposable} holds for central, essential, and indecomposable hyperplane arrangements such that $\ann_{D_{n}[S]} F^{S} = D_{n}[S] \cdot \theta_{F}.$
    \item In an ongoing project, we plan to generalize this argument to find many more hyperplanes that must lie in $\V(B_{F})$ for central, essential, indecomposable, and tame hyperplane arrangements. When $\V(B_{F})$ corresponds to the roots of the Bernstein--Sato polynomial, i.e. when $F = (f)$, this will let us compute many, if not all, of the roots of the Bernstein--Sato polynomial lying in $[-1, 0).$
\end{enumerate}

\end{remark}
\fi

As outlined in the introduction, Theorem \ref{thm indecomposable} is related to the Topological Multivariable Strong Monodromy Conjecture, that is, to Conjecture \ref{conjecture topological multi strong}. 

\begin{cor} \label{cor Topological Multivariable Strong Monodromy Conjecture}

The Topological Multivariable Strong Monodromy Conjecture is true for (not necessarily reduced) tame hyperplane arrangements.

\end{cor}

\begin{proof}

This follows by Theorem 8 of \cite{BudurBernstein--Sato} since tameness is a local condition.
\end{proof}

\begin{remark}
Not all arrangements are tame. For example, the $\mathbb{C}^{4}$-arrangement
$\prod_{(a_{1}, \dots, a_{4}) \in \{0,1\}^{4}} (a_{1}x_{1} + a_{2}x_{2} + a_{3}x_{3} + a_{4}x_{4})$ is not tame. If an arrangement has rank at most $3$, then it is automatically tame.
\end{remark}

\iffalse It is unclear if the $\D_{X,\x}[s]$-annihilator of $f^{s}$ is generated by derivations implies the $\D_{X,\x}[S]$-annihilator of $F^{S}$ is generated by derivations. \fi

\section{The Map $\nabla_{A}$}

In this section we analyze the injectivity of $\D_{X,\x}$-map
$$\nabla_{A}: \frac{\D_{X,\x}[S]F^{S}}{(S-A)\D_{X,\x}[S]F^{S}} \to \frac{\D_{X,\x}[S]F^{S}}{(S-(A-1))\D_{X,\x}[S]F^{S}}$$
under the nice hypotheses of the previous section. This will, see Section 5, let us better understand the relationship between $\V(B_{F,\x})$ and the cohomology support loci of $f$ near $\x$. The section has two parts: a brief discussion of Koszul complexes associated to central elements over certain non-commutative rings with an application to $\frac{\D_{X,\x}[S]F^{S}}{(S-A)\D_{X,\x}[S]F^{S}}$; a detailed proof that under nice hypotheses, if $\nabla_{A}$ is injective then it is surjective.

Let's first give a precise definition of $\nabla_{A}.$

\begin{definition} (Compare to 5.5 and 5.10, in particular $\rho_{\alpha}$, in \cite{BudurBernstein--Sato}) \label{nabla definition}
Define 
\[
\nabla: \D_{X,\x}[S]F^{S} \to \D_{X,\x}[S]F^{S}
\]
by sending $s_{i} \mapsto s_{i} + 1$ for all $i$. To be precise, in local coordinates declare $\partial^{\textbf{u}} = \prod_{t} \partial_{x_{t}}^{u_{t}}$, $S^{\textbf{v}} = \prod_{k} s_{k}^{v_{k}}$, and let $S+1$ be shorthand for replacing each $s_{i}$ with a $s_{i} + 1.$ Then $\nabla$ is given by the assignment 
\[
\sum\limits_{\textbf{u}, \textbf{v}} Q_{\textbf{u}, \textbf{v}} \partial^{\textbf{u}} S^{\textbf{v}} \bullet F^{S} \mapsto  \sum\limits_{\textbf{u}, \textbf{v}} Q_{\textbf{u}, \textbf{v}} \partial^{\textbf{u}} (S+1)^{\textbf{v}} \bullet F^{S+1}.
\]
This is a homomorphism of $\D_{X,\x}$-modules but is not $\mathbb{C}[S]$-linear. 

Denote the ideal of $\D_{X,\x}[S]$ generated by $s_{1}-a_{1}, \dots , s_{r} - a_{r}$, for $a_{1}, \dots, a_{r} \in \mathbb{C}$ by $(S-A)\D_{X,\x}[S]$. Then $\nabla$ is injective and sends $(S-A)\D_{X,\x}[S]F^{S}$ onto $(S+1-A)\D_{X,\x}[S]F^{S+1} = (S-(A-1)\D_{X,\x}[S]F^{S+1} \subseteq (S-(A-1))\D_{X,\x}[S]F^{S}.$ Let $\nabla_{A}$ be the induced homomorphism of $\D_{X,\x}$-modules:
\[
\nabla_{A}: \ \frac{\D_{X,\x}[S]F^{S}}{(S-A)\D_{X,\x}[S]F^{S}} \longrightarrow \frac{\D_{X,\x}[S]F^{S}}{(S-(A-1))\D_{X,\x}[S]F^{S}}.
\]

\end{definition}

As mentioned in the introduction, a source of our motivation is investigating the three statements that show up in the following proposition.

\begin{prop} \label{prop basic nabla} Consider the following three statements, where $A-1$ denotes the tuple $(a_{1} -1, \dots , a_{r} - 1) \in \mathbb{C}^{r}:$
\begin{enumerate}[(a)]
    \item $A-1 \notin \V(B_{F,\x})$; 
    \item  $\nabla_{A}$ is injective;
    \item $\nabla_{A}$ is surjective.
\end{enumerate}
Then in any case (a) implies (b) and (c).
\end{prop}

\begin{proof} 

Choose a functional equation $B(S) F^{S} = P(S) F^{S+1}$ where we may assume $B(A-1) \neq 0.$ 

We first prove that (a) implies (c). Since $\nabla(P(S-1)F^{S}) = P(S) F^{S+1}$,
\begin{align*} \overline{P(S-1)F^{S}} \stackrel{\nabla_{A}}{\longmapsto} \overline{P(S)F^{S+1}} = \overline{B(S)F^{S}}.
\end{align*}
This shows that $\nabla_{A}(\overline{P(S-1)F^{S}})$ generates $\frac{\D_{X,\x}[S]F^{S}}{(S-(A-1))\D_{X,\x}[S]F^{S}}.$

To show that (a) implies (b) suppose $\nabla_{A}(\overline{Q(S) F^{S}}) = 0.$ This means $Q(S+1)F^{S+1} \in \sum (s_{i} - (a_{i} - 1)) \cdot \D_{X,\x}[S] F^{S}$. Multiplying both sides by $B(S)$ gives $Q(S+1) B(S) F^{S+1} \in \sum (s_{i} - (a_{i} - 1)) \cdot \D_{X,\x}[S] P(S) F^{S+1}.$ So $Q(S)B(S-1)F^{S} \in \sum (s_{i} - a_{i}) \cdot \D_{X,\x}[S] F^{S}$ and $\overline{Q(S)F^{S}}$ is zero in $\frac{\D_{X,\x}[S]F^{S}}{(S-A)\D_{X,\x}[S]F^{S}}.$
\end{proof}

\begin{remark} \label{remark basic nabla facts}
\begin{enumerate}[(1)]
    \item In the classical setting where $F = (f)$ and $F^{S} = f^{s}$, (a), (b), and (c) of Proposition \ref{prop basic nabla} are equivalent (see 6.3.15 in \cite{Bjork} for the equivalence of (a) and (c); the claims involving (b) follow by a similar diagram chase).
    
    \item Suppose $A = (a , \dots, a)$ and $f$ is strongly Euler-homogeneous, Saito-holonomic, and tame. By Remark \ref{remark compare complicated decompositions}, \iffalse Because $(S-A) \cdot \D_{X,\x}[S] = (s_{1} - s_{2}, \dots, s_{r-1} - s_{r}, s_{r} - a) \cdot \D_{X,\x}[S]$, by Proposition \ref{prop going down} \fi there is a commutative square of $\D_{X,\x}$- maps:
    \[
    \begin{tikzcd}
        \frac{\D_{X,\x}[S]F^{S}}{(S-A)\D_{X,\x}[S]F^{S}} \arrow[r, "\simeq"] \arrow[d, "\nabla_{A}"] 
            & \frac{\D_{X,\x}[s] f^{s}}{(s-a)\D_{X,\x}[s]f^{s}} \arrow[d, "\nabla_{a}"] \\
        \frac{\D_{X,\x}[S]F^{S}}{(S-(A-1))\D_{X,\x}[S]F^{S}} \arrow[r, "\simeq"]
            & \frac{\D_{X,\x}[s] f^{s}}{(s-(a-1))\D_{X,\x}[s]f^{s}}.
    \end{tikzcd}
    \]
\iffalse So $\nabla_{A}$ is a injective (surjective) if and only if $\nabla_{a}$ is injective (surjective). \fi If the conditions in Proposition \ref{prop basic nabla} were equivalent, then the inclusion induced by the diagonal embedding $\V(B_{f,\x}) \hookrightarrow \V(B_{F,\x}) \bigcap \V(s_{1} - s_{2}, \dots, s_{r-1} - s_{r})$, given in Proposition \ref{prop diagonal embedding} would be surjective.
\end{enumerate}
\end{remark}

\begin{example} \label{running non-free example BS computation}
Let $f = x(2x^2 + yz)$ and $F = (x, 2x^2 + yz)$. This is strongly Euler-homogeneous, Saito-holonomic (cf. Examples \ref{example running non-free example strong Euler}, \ref{example running non-free example saito-holonomic}), and tame ($n \leq 3$). Using Singular and Macaulay2 we compute $\V(B_{F,0}) = (s_{1} + 1)(s_{2} + 1) \prod_{k=3}^{r}(s_{1} + 2s_{2} + k)$ and $\V(B_{f,0}) = (s+1)^{3}(s+\frac{4}{3})(s+\frac{5}{3}).$ In this case, the diagonal embedding $\V(B_{f,0}) \hookrightarrow \V(B_{F,0}) \bigcap \V(s_{1} - s_{2})$ of Proposition \ref{prop diagonal embedding} is surjective and, see Remark \ref{remark basic nabla facts}, $\nabla_{-k+1, -k+1}$ is neither surjective nor injective for $k = 3, 4, 5.$
\end{example}

The rest of this section is devoted to proving that under the nice hypotheses of the previous section and in the language of Proposition \ref{prop basic nabla}, that (b) implies (c). Our proof makes use of a Koszul resolution over the central elements $S-A.$

\begin{Convention}
A resolution is a (co)-complex with a unique (co)homology module at its end. An acyclic (co)-complex has no (co)homology. Given a (co)-complex ($C^{\bullet}$) $C_{\bullet}$ resolving $A$, the augmented (co)-complex ($C^{\bullet} \to A$) $C_{\bullet} \to A$ is acyclic. 
\end{Convention}

\begin{definition}

For a (not necessarily commutative) ring $R$ and a sequence of central $R$-elements $a = a_{1}, \dots, a_{k}$ let $K^{\bullet}(a)$ be the Koszul co-complex induced by the elements $a$, cf. Section 6 in \cite{24Hours}. For a left $R$-module $M$, let $K^{\bullet}(a; M) = K^{\bullet}(a) \otimes M$ be the Koszul co-complex on $M$ induced by $a$. We index $K^{\bullet}(a)$ so that the right most object is $K^{0}(a).$

\iffalse denote $aM = \sum\limits_{i} a_{i}M$ and $K(a; M)$ be the Koszul complex defined by 
\begin{align*}
& K^{\bullet}(a_{1}) := 0 \rightarrow R \xrightarrow{a_{1}} R \rightarrow 0 \\
& K^{\bullet}(a_{1}, \dots, a_{s+1}) := K^{\bullet}(a_{1}, \dots, a_{s}) \underset{R}{\otimes} K^{\bullet}(a_{s+1}) \\
& K^{\bullet}(a; M):= K^{\bullet}(a) \underset{R}{\otimes} M
\end{align*} 

\fi

\iffalse
As left $R$-modules, $K^{s}(a; M) = R^{\left( \frac{k}{s} \right)} \underset{R}{\otimes} M \simeq M^{\left( \frac{k}{s} \right)} $ (to make the indices nice we remove extraneous 0 terms). $K^{\bullet}(a; M)$ is a natural candidate for a resolution of $M / aM$ since the terminal cohomology group is $$
\frac{R \underset{R}{\otimes} M}{\text{im}(R^{k} \underset{R}{\otimes} M \to R \underset{R}{\otimes} M)} \simeq M / a{M},$$
with the isomorphism induced by $r \underset{R}{\otimes} m \mapsto rm$.
\fi

\end{definition}

The following lemma is immediate after considering $ H^{-1} (K^{\bullet}(c_{1}, \dots, c_{r}; M)):$

\begin{lemma} \label{lemma standard koszul argument}

Let $R$ be a, possibly noncommutative, ring, $M$ a left $R$-module, $m_{i} \in M$, and $c_{1}, \dots, c_{r}$ central elements of $R$. Assume $H^{-1}(K(c_{1}, \dots, c_{r}); M) = 0.$ If $c_{i}m_{i} \in (c_{1}, \dots, c_{i-1}, c_{i+1}, \dots, c_{r})M$, then $m_{i} \in (c_{1}, \dots, c_{i-1}, c_{i+1}, \dots, c_{r})M.$

\end{lemma} 

Let $v_{1}, \dots, v_{k}$ be positive integers. If $R$ is commutative and if $K^{\bullet}(a; M)$ is a resolution, we know $K^{\bullet}(a_{1}^{v_{1}}, \dots, a_{k}^{v_{k}}; M)$ is a resolution, cf. Exercise 6.16 in \cite{24Hours}. A routine induction argument (that we omit) using the the tensor product of Koszul co-complexes verifies that this is also true for general $R$ and central $a$:

\iffalse
Proceeding by induction, we must show that  $K^{\bullet}(x^{j}, y_{2}, \dots , y_{k}; M) =  K^{\bullet}(x^{j}; M) \otimes K^{\bullet}(y_{2}, \dots, y_{k}; M)$ is a resolution. Note that $K^{\bullet}(y_{2}, \dots, y_{k}; M)$ is a resolution. By the spectral sequence of a double co-complex, $K^{\bullet}(x^{j}; M) \otimes K^{\bullet}(y_{2}, \dots, y_{k}; M) = K^{\bullet}(x^{j}; M / (y_{2}, \dots, y_{k})M).$ Now because $K^{\bullet}(x, y_{2}, \dots , y_{k}; M)$ is a resolution, so is $K^{\bullet}(x; M / (y_{2}, \dots, y_{k})M)$. In other words, multiplication by $x$ on $M / (y_{2}, \dots, y_{k})M$ is injective. Thus multiplication by $x^{j}$ is injective and $K^{\bullet}(x^{j}; M / (y_{2}, \dots, y_{k})M)$ is a resolution. This justifies the following:
\fi

\begin{prop} \label{prop-koszul central 2}

Let $R$ be a, possibly non-commutative, ring, $M$ a $R$-module, $c_{1}, \dots, c_{r}$ central elements of $R$, and $v_{1}, \dots, v_{r} \in \mathbb{Z}_{+}$. If $K^{\bullet}(c_{1}, \dots, c_{r}; M)$ is a resolution, then $K^{\bullet}(c_{1}^{v_{1}}, \dots, c_{r}^{v_{r}}; M)$ is a resolution.

\end{prop}

Now return to $\gr_{(0,1,1)}(\D_{X,\x}[S] F^{S})$. Under the nice hypothesis of the previous section, $\gr_{(0,1,1)}(s_{1}), \dots, \gr_{(0,1,1)}(s_{r})$ act like a regular sequence:

\begin{prop} \label{prop s regular}

Let $f = f_{1} \cdots f_{r}$ and let $F=(f_{1}, \dots, f_{r}).$ Suppose that for $\x \in X$ the following hold:
\begin{itemize}
    \item $f$ has the strong Euler-homogeneity $E_{\x}$ at $\x$;
    \item $\widetilde{L_{F,\x}} \subseteq \gr_{(0,1,1)}(\D_{X,\x}[S])$ is Cohen--Macaulay of dimension $n + r$;
    \item $L_{f,\x} + \gr_{(0,1)}(\D_{X,\x}) \cdot \gr_{(0,1)}(E_{\x}) \subseteq \gr_{(0,1)}(\D_{X,\x})$ is Cohen--Macaulay of dimension $n$.
\end{itemize}
\iffalse Label $S = s_{1}, \dots, s_{r}$ the sequence of elements in  $\gr_{(0,1,1)}(\D_{X,\x}[S])$. \fi Then $K^{\bullet}(S; \gr_{(0,1,1)}(\D_{X,\x}[S]) / \widetilde{L_{F,\x}})$ is co-complex of $\gr_{(0,1,1)}(\D_{X,\x}[S])$-modules resolving $\gr_{(0,1,1)}(\D_{X,\x}[S]) / (\widetilde{L_{F,\x}}, S) \simeq \gr_{(0,1)}(\D_{X,\x}) / (L_{f,\x} + \gr_{(0,1)}(\D_{X,\x}) \cdot \gr_{(0,1)}(E_{\x})).$

\end{prop}

\begin{proof}

The last isomorphism is immediate from the definition of $\psi_{F}$ and the construction of $\widetilde{L_{F,\x}}$ and $L_{f,\x}$, see Definition \ref{definition generalized Liouville} and the preceding comments.

Multiplying $\gr_{(0,1,1)}(\D_{X,\x}[S])$ by $s_{k}$ increases the degree of an element by one. So after doing the appropriate degree shifts, we may view $K^{\bullet}(S; \gr_{(0,1,1)}(\D_{X,\x}[S]))$ as a sequence of graded modules with degree preserving maps. By Proposition 1.5.15 (c) of \cite{BrunsHerzogCMRings}, exactness of such a sequence is a graded local property. The only $(0,1,1)$-graded maximal ideal $\mathfrak{m}^{\star}$ is generated by $\mathscr{O}_{X,\x}$ and the irrelevant ideal. So localize $K^{\bullet}(S; \gr_{(0,1,1)}(\D_{X,\x}[S]) / \widetilde{L_{F,\x}})$ at $\mathfrak{m}^{\star}$. 

By Theorem 2.1.2 of \cite{BrunsHerzogCMRings}, if both $(\gr_{(0,1,1)}(\D_{X,\x}[S])/ \widetilde{L_{F,\x}})_{\mathfrak{m}^{\star}}$ and 
\[
(\gr_{(0,1,1)}(\D_{X,\x}[S]) / \widetilde{L_{F,\x}} + \gr_{(0,1,1)}(\D_{X,\x}[S])_{\mathfrak{m}^{\star}} \simeq (\gr_{(0,1)}(\D_{X,\x}) / L_{f,\x} + \gr_{(0,1)}(E_{x}))_{\mathfrak{m}^{\star}}
\]
are Cohen--Macaulay and the difference in their dimensions is the length of the sequence $S$, then our localized Koszul co-complex is a resolution. Since the dimension of a graded-local ring equals the dimension after localization at the graded maximal ideal, cf. Corollary 13.7 of \cite{EisenbudCommutative}, we are done.
\end{proof}

For $a_{1}, \dots, a_{r} \in \mathbb{C}$, label $S-A = s_{1} - a_{1}, \dots, s_{r} - a_{r} \in \D_{X,\x}[S].$ Being central elements, $S-A$ yields the Koszul co-complex $K^{\bullet}(S-A; \D_{X,\x}[S]F^{S})$ of $\D_{X,\x}[S]$-modules. Its terminal cohomology module is $\D_{X,\x}[S]F^{S} / (S-A) \D_{X,\x}[S]F^{S}$. We show that under our standard hypotheses on $f$, i.e. strongly Euler-homogeneous, Saito-holonomic, and tame, that $s_{1} - a_{1}$, \dots, $s_{r} - a_{r}$ behaves like a regular sequence.

\begin{prop} \label{prop Koszul resolution}

Suppose $f = f_{1} \cdots f_{r}$ is strongly Euler-homogeneous, Saito-holonomic, and tame and let $F=(f_{1}, \dots, f_{r}).$ Then $K^{\bullet}(S-A; \D_{X,\x}[S]F^{S})$ resolves $\D_{X,\x}[S]F^{S} / (S-A) \D_{X,\x}[S]F^{S}$.

\end{prop}

\begin{proof}

Under the total order filtration, $s_{k} - a_{k}$ has weight one. It is routine to define a filtration $G$, compatible with the total order filtration, on the augmented co-complex $K^{\bullet}(S-A; \D_{X,\x}[S]F^{S}) \to \D_{X,\x}[S]F^{S} / (S-A)\D_{X,\x}[S] F^{S}$ such
$\gr_{G}(K^{\bullet}(S-A; \D_{X,\x}[S]F^{S}) \to \D_{X,\x}[S]F^{S} / (S-A)\D_{X,\x}[S] F^{S})$ is isomorphic to $K^{\bullet}(S; \gr_{(0,1,1)}({\D_{X,\x}[S]}) / \widetilde{L_{F,\x}}) \to \gr_{(0,1)}(\D_{X,\x})/(L_{f,\x} + \gr_{(0,1)}(\D_{X,\x}) \cdot \gr_{(0,1)}(E_{x})).$
If this co-complex is acyclic, then a standard argument using the spectral sequence attached to a filtered co-complex proves that $K^{\bullet}(S-A; \D_{X,\x}[S]F^{S}) \to \D_{X,\x}[S]F^{S} / (S-A)\D_{X,\x}[S] F^{S}$ is acyclic is well. The claim then follows by Theorem \ref{thm main prime theorem}, Corollary 3.19 of \cite{uli}, and Proposition \ref{prop s regular}. \iffalse the co-complex $K^{\bullet}(\gr_{(0,1,1)}(\D_{X,\x}[S]) / \widetilde{L_{F,\x}})$ is a resolution. A standard argument using the spectral sequence attached to a filtered co-complex now proves $K^{\bullet}(S-A; \D_{X,\x}[S]F^{S}) \to \D_{X,\x}[S]F^{S} / (S-A)\D_{X,\x}[S] F^{S}$ is acyclic. \fi
\end{proof}

Finally we can prove the section's main theorem:
\begin{theorem} \label{thm injective implies surjective}

Let $f= f_{1} \cdots f_{r}$ be strongly Euler-homogeneous, Saito-holonomic, and tame and let $F=(f_{1}, \dots, f_{r}).$ If $\nabla_{A}$ is injective, then it is surjective. 
\end{theorem}
\begin{proof}

For this proof, and this proof alone, write $\widetilde{s_{i}} = s_{i} - (a_{i} - 1).$ Also, $\overline{(-)}$ denotes the image of $(-)$ in the appropriate quotient object.

\textit{The Plan:} If there is some multivariate Bernstein--Sato polynomial $B(S)$ that does not vanish at $(a_{1}-1, \dots, a_{r} - 1)$, then the claim follows by Proposition \ref{prop basic nabla}. So pick a multivariate Bernstein--Sato polynomial $B(S) = \sum A_{k} \widetilde{s_{k}}$, $A_{k} \in \mathbb{C}[S].$ The idea is to successively ``remove" each $s_{k}$ factor from each $A_{k}$. In doing so, we will produce a finite sequence of polynomials $B_{0}, B_{i}, \dots $ satisfying the technical condition \eqref{inductive condition} introduced in Step 1, starting with our multivariate Bernstein--Sato polynomial, such that each polynomial uses fewer variables than its predecessor. The terminal polynomial will demonstrate that the cokernel of $\nabla_{A}$ vanishes. 

The inductive construction of these polynomials is not hard but technical. Before doing it we prove three claims. The first is that a particular cohomology module of the Koszul co-complex of $\widetilde{s_{1}}, \dots, \widetilde{s_{r}}$ on $\frac{\D_{X,\x}[S]F^{S}}{\D_{X,\x}[S]F^{S+1}}$ vanishes. We use this to ``remove" the $\widetilde{s_{k}}$ factors. The second and third claims are the technical details comprising the inductive algorithm used to produce these polynomials.

\textit{Claim 1:} For all positive integers $v_{1}, \dots, v_{r}$, 
\[ H^{-1} \left( K^{\bullet} \left( \widetilde{s_{1}}^{v_{1}}, \dots , \widetilde{s_{r}}^{v_{r}} ; \frac{\D_{X,\x}[S] F^{S}}{ \D_{X,\x}[S]F^{S+1}} \right) \right)= 0. \]

\textit{Proof of Claim $1$:} The $\D_{X,\x}$-map $\nabla_{A}$ is always injective and sends $F^{S} \mapsto F^{S+1}$. If $\nabla_{A}$ is also injective there is a short exact sequence of augmented co-complexes: 
\begin{align*} 
0 &\to (K^{\bullet}(S-A ; \D_{X,\x}[S]F^{S}) \to \frac{\D_{X,\x}[S]F^{S}}{(S-A) \D_{X,\x}[S]F^{S}}) \\
&\to (K^{\bullet}(S-(A-1); \D_{X,\x}[S]F^{S}) \to \frac{\D_{X,\x}[S]F^{S}}{(S-(A-1)) \D_{X,\x}[S]F^{S}}) \\
&\to (K^{\bullet}(S-(A-1); \frac{\D_{X,\x}[S]F^{S}}{\D_{X,\x}[S]F^{S+1}}) \to \frac{\D_{X,\x}[S]F^{S}}{(S-(A-1))\D_{X,\x}[S]F^{S+1}}) \to 0. 
\end{align*}
The first map is induced by $\nabla$ and by $\nabla_{A}$ on the augmented part; the second by quotient maps. By Proposition \ref{prop Koszul resolution} and the canonical long exact sequence, the last (nonzero) augmented co-complex is acyclic. Claim $1$ follows by Proposition \ref{prop-koszul central 2}.

\textit{Claim $2$:} Write $\overline{F^{S}}$ for the image of $F^{S}$ in $\frac{\D_{X,\x}[S] F^{S}}{\D_{X,\x}[S] F^{S+1}}.$ Suppose there exists $P(S) \in \mathbb{C}[S]$, $1 \leq j < r$, positive integers $n_{j+1}, \dots, n_{r}$, and an integer $m \geq \text{max} \{ n_{j+1}, \dots, n_{r} \}$ such that
\[ 
\bigg( \prod_{j+1 \leq k \leq r} \widetilde{s_{k}}^{n_{k}} \bigg) P(S) \bullet \overline{F^{S}} \in (\widetilde{s_{1}}, \dots, \widetilde{s_{j}}, \widetilde{s_{j+1}}^{m}, \dots, \widetilde{s_{r}}^{m}) \frac{\D_{X,\x}[S] F^{S}}{ \D_{X,\x}[S] F^{S+1}}. 
\]
Then for $m^{\prime} = \text{min}\{m-n_{j+1}, \dots , m-n_{r}\}$ we have
\[
P(S) \bullet \overline{F^{S}} \in (\widetilde{s_{1}}, \dots, \widetilde{s_{j}}, \widetilde{s_{j+1}}^{m^{\prime}}, \dots, \widetilde{s_{r}}^{m^{\prime}}) \frac{\D_{X,\x}[S] F^{S}}{ \D_{X,\x}[S] F^{S+1}}.
\]
\textit{Proof of Claim $2$:} The idea is to use Claim 1 and Lemma \ref{lemma standard koszul argument} to ``remove" each $\widetilde{s_{k}}^{n_{k}}$ factor one at a time. We first ``remove" the $\widetilde{s_{j+1}}^{n_{j+1}}$ factor. 

By hypothesis, there exists $Q_{j+1} \in \D_{X,\x}[S]$ such that 
\begin{align*}
    \bigg( \prod_{j+1 \leq k \leq r} \widetilde{s_{k}}^{n_{k}} &\bigg)  P(S) \bullet \overline{F^{S}} - \widetilde{s_{j+1}}^{m} Q_{j+1} \bullet \overline{F^{S}} \\
    &= \widetilde{s_{j+1}}^{n_{j+1}} \Bigg( \bigg( \prod_{j+2 \leq k \leq r} \widetilde{s_{k}}^{n_{k}} \bigg) P(S) \bullet \overline{F^{S}} - \widetilde{s_{j+1}}^{m-n_{j+1}} Q_{j+1} \bullet \overline{F^{S}} \Bigg) \\
    &\in (\widetilde{s_{1}}, \dots, \widetilde{s_{j}}, \widetilde{s_{j+2}}^{m}, \dots, \widetilde{s_{r}}^{m}) \frac{\D_{X,\x}[S] F^{S}}{ \D_{X,\x}[S] F^{S+1}}.
\end{align*}
By Claim $1$, $H^{-1}(K(\widetilde{s_{1}}, \dots, \widetilde{s_{j}}, \widetilde{s_{j+1}}^{n_{j+1}}, \widetilde{s_{j+2}}^{m}, \dots, \widetilde{s_{r}}^{m}; \frac{\D_{X,\x}[S] F^{S}}{\D_{X,\x}[S]F^{S+1}}))$ vanishes. So Lemma \ref{lemma standard koszul argument} implies 
\begin{align*}
\bigg( \prod_{j+2 \leq k \leq r} \widetilde{s_{k}}^{n_{k}} \bigg) P(S) \bullet \overline{F^{S}} &- \widetilde{s_{j+1}}^{m-n_{j+1}} Q_{j+1} \bullet \overline{F^{S}} \\
&\in (\widetilde{s_{1}}, \dots, \widetilde{s_{j}}, \widetilde{s_{j+2}}^{m}, \dots, \widetilde{s_{r}}^{m}) \frac{\D_{X,\x}[S] F^{S}}{ \D_{X,\x}[S] F^{S+1}}.
\end{align*}
Rearrange to see
\[
\bigg( \prod_{j+2 \leq k \leq r} \widetilde{s_{k}}^{n_{k}} \bigg) P(S) \bullet \overline{F^{S}} \in (\widetilde{s_{1}}, \dots, \widetilde{s_{j}}, \widetilde{s_{j+1}}^{m - n_{j+1}}, \widetilde{s_{j+2}}^{m}, \dots, \widetilde{s_{r}}^{m}) \frac{\D_{X,\x}[S] F^{S}}{ \D_{X,\x}[S] F^{S+1}}.
\]
Repeat this process on each remaining factor $\widetilde{s_{k}}^{n_{k}}$, $j+2 \leq k \leq r$ one at a time to conclude
\[
P(S) \bullet \overline{F^{S}} \in (\widetilde{s_{1}}, \dots, \widetilde{s_{j}}, \widetilde{s_{j+1}}^{m - n_{j+1}}, \widetilde{s_{j+2}}^{m - n_{j+2}}, \dots, \widetilde{s_{r}}^{m - n_{r}}) \frac{\D_{X,\x}[S] F^{S}}{ \D_{X,\x}[S] F^{S+1}}.
\]

\textit{Claim $3$:} Suppose $B_{j} \in \mathbb{C}[s_{j+1}, \dots, s_{r}]$, where $j < r$, with $B_{j} \in \mathbb{C}[s_{j+1}, \dots , s_{r}] \cdot (\widetilde{s_{j+1}}, \dots, \widetilde{s_{r}})$ but $B_{j} \notin \mathbb{C}[\widetilde{s_{j+1}}, \dots, \widetilde{s_{r}}] \cdot (\widetilde{s_{k}})$ for all $j+1 \leq k \leq r$. Furthermore, assume that for $m \geq \text{max} \{ n_{j+1}, \dots, n_{r} \}$ we have
\[
B_{j} \bullet F^{S} \in (\widetilde{s_{1}}, \dots, \widetilde{s_{j}}, \widetilde{s_{j+1}}^{m}, \dots, \widetilde{s_{r}}^{m}) \frac{\D_{X,\x}[S] F^{S}}{ \D_{X,\x}[S] F^{S+1}}.
\]
Then, relabeling the $s_{k}$ if necessary, there exists $B_{i} \in \mathbb{C}[s_{i+1}, \dots s_{r}]$, where $j < i < r$, $B_{t} \notin \mathbb{C}[s_{i+1}, \dots, s_{r}] \cdot (\widetilde{s_{k}})$ for $i+1 \leq k \leq r$, so that for $m^{\prime} = \text{min}\{m-n_{j+1}, \dots , m-n_{r}\}$ we have
\[
B_{i} \bullet F^{S} \in (\widetilde{s_{1}}, \dots, \widetilde{s_{i}}, \widetilde{s_{t+1}}^{m^{\prime}}, \dots, \widetilde{s_{r}}^{m^{\prime}}) \frac{\D_{X,\x}[S] F^{S}}{ \D_{X,\x}[S] F^{S+1}}.
\]
\textit{Proof of Claim $3$:} Note that the hypotheses imply $j < r-1$ so the promised choice of $i$ is possible. Since $B_{j} \notin \mathbb{C}[s_{j+1}, \dots, s_{r}] \cdot (\widetilde{s_{k}})$ for all $j+1 \leq k \leq r$, there exists a largest $\emptyset \neq I = \{\widetilde{s_{i_{1}}}, \dots, \widetilde{s_{i_{|i|}}} \} \subsetneq \{j+1, \dots, r\}$ such that $B_{j} \notin \mathbb{C}[s_{j+1}, \dots, s_{r}] \cdot (\widetilde{s_{i_{1}}}, \dots, \widetilde{s_{i_{|I|}}}).$ Relabel so that $I = \{j+1, \dots, i\}.$ This means there exist positive integers $n_{k}$, polynomials $A_{l} \in \mathbb{C}[S]$, and a polynomial $B_{i} \in \mathbb{C}[s_{i+1}, \dots, s_{r}]$ such that
\[ 
B_{j} = \bigg( \prod_{i+1 \leq k \leq r} \widetilde{s_{k}}^{n_{k}} \bigg) B_{i} + \sum_{1 \leq \ell \leq i} \widetilde{s_{\ell}} A_{\ell}.
\]
We may make each $n_{k}$ large enough so as to assume $B_{i} \notin \mathbb{C}[s_{i+1}, \dots , s_{r}] \cdot (\widetilde{s_{k}})$ for any $i+1 \leq k \leq r.$

Therefore
\[
\biggr( \prod_{i+1 \leq k \leq r} \widetilde{s_{k}}^{n_{k}} \biggr) B_{i} \bullet \overline{F^{S}} \in (\widetilde{s_{1}}, \dots, \widetilde{s_{i}}, \widetilde{s_{i+1}}^{m}, \dots, \widetilde{s_{r}}^{m}) \frac{\D_{X,\x}[S]F^{S}}{\D_{X,\x}[S]F^{S+1}}.
\]
Then Claim $3$ follows from Claim $2$. 

\textit{Proof of Theorem.} 

\textit{Step $1$:} We will inductively construct a sequence of polynomials $B_{i_{1}}$, $B_{i_{2}}, \dots$, such that (after potentially relabelling the $s_{k}$) the following hold: $0 \leq i_{t} < r$ for each $i_{t}$; $i_{t} < i_{t+1}$; $B_{i_{t}} \in \mathbb{C}[s_{i_{t}+1}, \dots, s_{r}]$; for $m_{i_{t}}$ arbitrarily large
\begin{gather} \label{inductive condition}
B_{i_{t}} \bullet \overline{F^{S}} \in (\widetilde{s_{1}}, \dots, \widetilde{s_{i_{t}}}, \widetilde{s_{i_{t} + 1}}^{m_{i_{t}}}, \dots, \widetilde{s_{r}}^{m_{i_{t}}}) \frac{\D_{X,\x}[S] F^{S}}{\D_{X,\x}[S]F^{S+1}}.
\end{gather}
We terminate the induction once we produce a $B_{i}$ such that, in addition to the above properties, $B_{i} \notin \mathbb{C}[s_{i+1}, \dots , s_{r}] \cdot (\widetilde{s_{i+1}}, \dots, \widetilde{s_{r}})$.

\qquad \textit{Base Case:} Take a multivariate Bernstein--Sato polynomial $B(S) \in B_{F,\x}$. If $B(S) \notin \mathbb{C}[s_{1}, \dots, s_{r}] \cdot (\widetilde{s_{1}}, \dots, \widetilde{s_{r}})$ then we are done: $B(S) = B_{0}$ works. (Recall $B(S) \bullet F^{S} \in \D_{X,\x}[S]F^{S+1}.$) Otherwise find the largest $J \subsetneq [r]$ such that
$B(S) \notin C[S] \cdot (\widetilde{s_{j_{1}}}, \dots, \widetilde{s_{j_{|J|}}}).$
Re-label to assume $J = \{1, \dots, j\}, j < r$. (We allow $J = \emptyset$, in which case $j=0$.) This means we can write $B(S)$ as 
$$ B(S) =  \bigg( \prod\limits_{j+1 \leq k \leq r} \widetilde{s_{k}}^{n_{k}} \bigg) B_{j} + \sum\limits_{1 \leq t \leq j} \widetilde{s_{t}}A_{t}$$
where  $B_{j} \in \mathbb{C}[s_{j+1}, \dots, s_{r}]$ and each $n_{k}$ a positive integer chosen large enough so that $B_{j} \notin \mathbb{C}[s_{j+1}, \dots, s_{r}] \cdot (\widetilde{s_{k}})$, for $j+1 \leq k \leq r$,  Because $B(S)$ is a multivariate Bernstein--Sato polynomial, $B(S) \bullet F^{S} \in \D_{X,\x}[S]F^{S+1}.$ Therefore,
\begin{equation} \label{eqn obvious containment}
\bigg(\prod\limits_{j+1 \leq k \leq r} \widetilde{s_{k}}^{n_{k}} \bigg) B_{j}\bullet \overline{F^{S}} \in (\widetilde{s_{1}}, \dots, \widetilde{s_{j}}) \frac{\D_{X,\x}[S]F^{S}}{\D_{X,\x}[S]F^{S+1}}.
\end{equation}

Now \eqref{eqn obvious containment} trivially implies that for all $m \geq 0$
\[
\bigg(\prod\limits_{j+1 \leq k \leq r} \widetilde{s_{k}}^{n_{k}} \bigg) B_{j}\bullet \overline{F^{S}} \in (\widetilde{s_{1}}, \dots, \widetilde{s_{j}}, \widetilde{s_{j+1}}^{m}, \dots, \widetilde{s_{r}}^{m}) \frac{\D_{X,\x}[S]F^{S}}{\D_{X,\x}[S]F^{S+1}}.
\]
In particular, the above holds for $m$ arbitrarily large. By Claim 2, there exists $m_{j}$ arbitrarily large such that 
\[
 B_{j}\bullet \overline{F^{S}} \in (\widetilde{s_{1}}, \dots, \widetilde{s_{j}}, \widetilde{s_{j+1}}^{m_{j}}, \dots, \widetilde{s_{r}}^{m_{j}}) \frac{\D_{X,\x}[S]F^{S}}{\D_{X,\x}[S]F^{S+1}}.
\]
Then $B_{j}$ is the first element in our sequence of polynomials.

\qquad \textit{Inductive Step:} Suppose $B_{j} \in \mathbb{C}[s_{j+1}, \dots, s_{r}]$ has already been defined. If the algorithm has not terminated, $j < r$ and $B_{j} \notin \mathbb{C}[s_{j+1}, \dots, s_{r}] \cdot (\widetilde{s_{k}})$ for all $j+1 \leq k \leq r.$ Then use Claim $3$ to define $B_{i}$, where $j < i < r.$ Note that if $j = r-1$ then $B_{r-1} \notin \mathbb{C}[s_{r}] \cdot (\widetilde{s_{r}})$ and so the algorithm terminates at $B_{r-1}.$ 

\textit{Step $2$:} Use the terminal polynomial $B_{i} \in \mathbb{C}[s_{i+1}, \dots, s_{r}]$, $i < r$, produced by Step 1. This means $B_{i} \notin \mathbb{C}[s_{i+1}, \dots, s_{r}] \cdot (\widetilde{s_{i+1}}, \dots, \widetilde{s_{r}})$ and easily implies
\[
B_{i} \bullet \overline{F^{S}} \in (\widetilde{s_{1}}, \dots, \widetilde{s_{r}}) \frac{\D_{X,\x}[S]F^{S}}{\D_{X,\x}[S]F^{S+1}}.
\]
On one hand, since $B_{i}$ does not vanish at $(a_{i+1} -1, \dots, a_{r} -1)$, $B_{i} \overline{F^{S}}$ and $\overline{F^{S}}$ generate the same submodule of $\frac{\D_{X,\x}[S]F^{S}}{(\widetilde{s_{i}}, \dots, \widetilde{s_{r}})\D_{X,\x}[S] F^{S} + \D_{X,\x}[S]F^{S+1}};$ on the other hand, $0 = B_{i} \bullet \overline{F^{S}} \in \frac{\D_{X,\x}[S]F^{S}}{(\widetilde{s_{i}}, \dots, \widetilde{s_{r}})\D_{X,\x}[S] F^{S} + \D_{X,\x}[S]F^{S+1}}.$ \iffalse Since $\frac{\D_{X,\x}[S]F^{S}}{(\widetilde{s_{i}}, \dots, \widetilde{s_{r}})\D_{X,\x}[S] F^{S} + \D_{X,\x}[S]F^{S+1}}$ is generated by $\overline{F^{S}}$, \fi Thus, $\frac{\D_{X,\x}[S]F^{S}}{(\widetilde{s_{i}}, \dots, \widetilde{s_{r}})\D_{X,\x}[S] F^{S} + \D_{X,\x}[S]F^{S+1}} = 0$ (because it is generated by $\overline{F^{S}}).$ That is, the cokernel of $\nabla_{A}$ vanishes.
\end{proof}

Using Theorem \ref{thm injective implies surjective} we can show that the three conditions of Proposition \ref{prop basic nabla} are equivalent in a very special and restricted case: \iffalse where $f$ cuts out a central, tame hyperplane arrangement and $A-1$ lies in the hyperplane of $\V(B_{F,0})$ promised by Theorem \ref{thm buduIn light of Proposition \ref{prop basic nabla} and Theorem \ref{thm injective implies surjective}, to prove, under our working hypotheses, the three conditions of Proposition \ref{prop basic nabla} are equivalent, it suffices to show that if $A-1 \in \V(B_{F,\x})$ then $\nabla_{A}$ is not surjective. We show that this holds for central, tame hyperplane arrangements if we assume $A-1$ lies in a certain hyperplane. \fi

\begin{prop} \label{prop nabla converse ex}

Suppose $f = f_{1} \cdots f_{r}$ is a central, essential, indecomposable, and tame hyperplane arrangement, where each $f_{k}$ is of degree $d_{k}$ and the $f_{k}$ are not necessarily reduced. Let $F = (f_{1}, \dots, f_{r}).$ If $A-1 \in \{d_{1}s_{1} + \dots + d_{r} s_{r} + n = 0\},$ then $A-1 \in \V(B_{F,0})$ and $\nabla_{A}$ is neither surjective nor injective. 

\end{prop}

\begin{proof}
An easy extension of the argument in Theorem \ref{thm indecomposable} shows that both
\begin{equation} \label{eqn analytic indecomposable}
\ann_{\D_{X,0}[S]}F^{S} + \D_{X,0}[S] \cdot f \subseteq \D_{X,0}[S] \cdot \mathfrak{m}_{0} + \D_{X,0}[S] \cdot (-n - \sum d_{k} s_{k}).
\end{equation}
and  $A-1 \in \V(B_{F,0}).$ Now $\nabla_{A}$ is surjective precisely when
\[
\D_{X,0}[S] = \ann_{\D_{X,0}[S]}F^{S} + \D_{X,0}[S] \cdot f + \sum \D_{X,0}[S] \cdot (s_{k} - (a_{k} -1)).
\]
By \eqref{eqn analytic indecomposable}, if $\nabla_{A}$ is surjective, 
\[
\D_{X,0}[S] \subseteq \D_{X,0}[S] \cdot \mathfrak{m}_{0} + \D_{X,0}[S] \cdot (-n - \sum d_{k} s_{k}) + \sum \D_{X,0}[S] \cdot (s_{k} - (a_{k} -1)).
\]
After evaluating each $s_{k}$ at $a_{k}-1$, we deduce 
$\D_{X,0} \subseteq \D_{X,0} \cdot \mathfrak{m}_{0}$. Therefore $\nabla_{A}$ is not surjective. By Theorem \ref{thm injective implies surjective}, $\nabla_{A}$ is not injective.
\iffalse
That $\nabla_{A}$ is not injective is a consequence of Theorem \ref{thm injective implies surjective}.
\fi
\end{proof}

\section{Free Divisors, Lie--Rinehart Algebras, and $\nabla_{A}$}

In Definition \ref{def tame} we defined tame divisors. A stronger condition is freeness:

\begin{definition}

A divisor $Y$ is \emph{free} if it locally everywhere admits a defining equation $f$ such that $\Der_{X,\x}(- \log f)$ is a free $\mathscr{O}_{X,\x}$-module.

\end{definition}

Freeness implies tameness because $\Omega_{X,\x}( \log f)$ and $\Der_{X,\x}(- \log f)$ are dual and if $\Omega_{X,\x}( \log f)$ is free, then $\Omega_{X,\x}^{p}( \log f) = \bigwedge^{p} \Omega_{X,\x}( \log f)$ (see 1.7, 1.8 of \cite{SaitoTheoryLogarithmic}).

Throughout this section we upgrade our working hypotheses of strongly Euler-homogeneous, Saito-holonomic, and tame to reduced, strongly Euler-homogeneous, Saito-holonomic, and free. The goal is to investigate the surjectivity of the map $\nabla_{A}.$ Let's give a road map. First we compute $\text{Ext}$ modules of $\D_{X,\x}[S]F^{S} / (S-A) \D_{X,\x}[S]F^{S}$ using \cite{MacarroDuality} and the rich theory of Lie--Rinehart algebras. Lifting a surjective $\nabla_{A}$ to these $\text{Ext}$-modules will produce an injective map. This injective map acts like $\nabla_{-A}$. By Theorem \ref{thm injective implies surjective}, $\nabla_{-A}$ is surjective. Using duality again will show that $\nabla_{A}$ is injective.

\subsection{Lie--Rinehart Algebras and the Spencer Co-Complex $\Sp$.} \text{ }

\begin{definition} (Compare with \cite{MorenoMacarroLogarithmic}), \cite{Rinehart} and the appendix of \cite{Malgrange}) Fix a homomorphism of commutative rings $k \to A$. A \emph{Lie--Rinehart algebra} $L$ over $(k,A)$ is a $A$-module L with \emph{anchor map} $\rho : L \to \Der_{k}(A)$ that is $A$-linear, a $k$-Lie algebra map, and satisfies, for all $\lambda$, $\lambda^{\prime} \in L$, $a \in A,$
\begin{align*}
    [ \lambda, a \lambda^{\prime}] = a[\lambda,\lambda^{\prime}] + \rho(\lambda)(a) \lambda^{\prime}.
\end{align*}
We will usually drop $\rho$ and replace $\rho(\lambda)(a)$ with $\lambda(a)$. A morphism $F: L \to L^{\prime}$ of Lie--Rinehart algebras over $(k,A)$ is a $A$-linear map that is a morphism of Lie-algebras satisfying $\lambda(a) = F(\lambda)(a).$
\end{definition}

\begin{example}
\begin{enumerate}[(a)]

    \item $\Der_{k}(A)$ is a Lie--Rinehart algebra over $(k,A)$ with the identity as the anchor map.
    
    \item Any $A$-submodule of $\Der_{k}(A)$ that is also a $k$-Lie algebra is a Lie--Rinehart algebra over $(k,A)$, with anchor map induced by the inclusion into $\Der_{k}(A)$. In particular $\Der_{X,\x}(- \Log(f))$ is a Lie--Rinehart algebra over $(\mathbb{C}, \mathscr{O}_{X,\x}).$
    
    \item If $L$ is a Lie--Rinehart algebra over $(k,A)$, then $L \oplus A$ is a Lie--Rinehart algebra over $(k,A)$ with anchor map induced by the projection $L \oplus A \to L$, $(\lambda, a) \mapsto \lambda.$ So $\Der_{X,\x} \oplus \mathscr{O}_{X,\x}^{r}$ and $\Der_{X,\x}(-\log(f)) \oplus \mathscr{O}_{X,\x}^{r}$ are Lie--Rinehart algebras over $(\mathbb{C}, \mathscr{O}_{X,\x})$. 
    
\end{enumerate}
\end{example}

\begin{definition}

Let $L$ be a Lie--Rinehart algebra over $(k, A)$ with $k \to A.$ Suppose $R$ is a ring (not necessarily a Lie--Rinehart algebra) and $A \to R$ a ring homomorphism that makes $R$ central over $k$, i.e. images of elements of $k$ are central elements in $R$. Then a $k$-linear map $g: L \to R$ is \emph{admissible} if:
\begin{enumerate}[(a)]
    \item $g(a \lambda) = a g(\lambda)$, for $a \in A$, $\lambda \in L$ ($g$ is a morphism of A-modules);
    \item $g([\lambda, \lambda^{\prime}]) = [g(\lambda), g(\lambda^{\prime})]$, for $\lambda$, $\lambda^{\prime} \in L$ ($g$ is a morphism of Lie-algebras);
    \item $g(\lambda) a - a g(\lambda) = \lambda(a) 1_{R}$ for $\lambda \in L$, $a \in A$.
\end{enumerate}

\end{definition}

The following theorem will be our definition of the \emph{universal algebra} $U(L):$

\begin{theorem} \text{ \normalfont (cf. \cite{Rinehart})} \label{thm- universal algebra} For any Lie--Rinehart algebra $L$ over $(k,A)$ there exists a ring $U(L)$, a ring homomorphism $A \to U(L)$ making $U(L)$ central over $k$, and an admissible map $\theta : L \to U(L)$ that is universal in the following sense: for any ring $R$ with a ring homomorphism $A \to R$ making $R$ central over $k$, and any admissible map $g: L \to R$, there is a unique ring homomorphism $h: U(L) \to R$ such that $h \circ \theta = g$. The natural map $\theta: L \to U(L)$ induces a filtration on $U(L)$ given by the powers of images of $\theta$. 
\end{theorem}

We omit the proof of the following proposition. It uses the (not provided) explicit construction of $U(L)$ and standard universal object arguments.

\begin{prop} \label{prop LR polynomial} Given a Lie--Rinehart algebra $L$ over $(k,A)$, consider the direct sum $L \oplus A$. This is a Lie--Rinehart algebra over $(k,A)$ with anchor map induced by projection: $L \oplus k \twoheadrightarrow L \to \Der_{k}(A).$ Then $U(L \oplus A) \simeq U(L)[s].$ Moroever, the natural filtration on $U(L \oplus A)$ corresponds to a ``total order filtration" on $U(L)[S]$, i.e. a filtration where $s$ has weight one.
\end{prop}

\iffalse
From the explicit construction of $U(L)$ and $U(L \oplus A)$ there is a natural inclusion $U(L) \hookrightarrow U(L \oplus A)$ given by sending $(a, \lambda) \mapsto (a, (\lambda, 0)).$ This induces a map $p: U(L)[S] \to U(L \oplus A)$ by $s \mapsto (0,(0,1)).$ Now define $L \oplus A \to U(L)[s]$ by extending the natural map $L \to U(L)$ by sending $(0,1) \mapsto s.$ This map is admissible. By the universal property of $U(L \oplus A)$, there is an induced map $h: U(L \oplus A) \to U(L)[S]$. 

$h \circ p$ and $p \circ h$ are the identity on the respective images of $L \oplus A.$ Since these respective images generate $U(L)[S]$ and $U(L \oplus A)$, $h$ and $p$ are left and right inverses; that is $h$ and $p$ are isomorphisms. This proves the claim. The $0^{\text{th}}$ filtered piece is generated over $A$ by $L \hookrightarrow U(L)$ and $s$.
\fi

\begin{example} \label{example universal LR algebras}
\begin{enumerate}[(a)]
    \item The universal Lie--Rinehart algebra of $\Der_{X,\x}$ over $(\mathbb{C},\mathscr{O}_{X,\x})$ is $\D_{X,\x}.$ The natural filtration is the order filtration.
    \item By repeated application of Proposition \ref{prop LR polynomial}, the universal Lie--Rinehart algebra of $\Der_{X,\x} \oplus \mathscr{O}_{X,\x}^{r}$ over $(\mathbb{C},\mathscr{O}_{X,\x})$ is $\D_{X,\x}[s_{1}, \dots, s_{r}].$ The natural filtration is the total order filtration $F_{(0,1,1)}$.
    \iffalse
    \item We already know that $\Der_{X,\x}(-\Log f)$ is a $(\mathbb{C}, \mathscr{O}_{X,\x})$ Lie Rinehart algebra. In \cite{MorenoLogarithmicDifferential}, (cf. section 1.2 in \cite{MorenoMacarroLogarithmic}) Calderon--Moreno shows that if $f$ is a free divisor, then $U(\Der_{X,\x}(-\log f)) \simeq \D_{X,\x}(- \log f)$, where $\D_{X,\x}(- \log f)$ is the $0^{\text{th}}$ term of the Malgrange-Kashiwara filtration relative to $f$ on $\D_{X,\x}$. Concretely, if $\delta_{1}, \dots, \delta_{n}$ is a $\mathscr{O}_{X,\x}$-basis for $\Der_{X,\x}(-\log f)$, then $\D_{X,\x}(- \log f)$ is generated by finite sums of the elements $\prod_{\textbf{u}} \delta_{1}^{u_{1}} \dots \delta_{n}^{u_{n}}$ over $\mathscr{O}_{X,\x}$. Every element in $\Der_{X,\x}(- \log f)$ has a unique expression in terms of these sums and the relations are the expected ones from $\D_{X,\x}$. The natural filtration on $U(\Der_{X,\x}(- \log f))$ is inherited from the order filtration on $\D_{X,\x}$. 
    \item By Proposition \ref{prop LR polynomial}, the universal algebra of the  Lie--Rinehart algebra $\Der_{X,\x}(-\log f) \oplus \mathscr{O}_{X,\x}^{r}$ over $(\mathbb{C}, \mathscr{O}_{X,\x})$ is $\D_{X,\x}(- \log f)[S]$. The filtration is induced by $F_{(0,1,1)}.$
    \fi
    \item For $F = (f_{1}, \dots, f_{r})$ a decomposition of $f = f_{1} \cdots f_{r}$, the annihilating derivations $\theta_{F,\x}$ constitute a Lie--Rinehart algebra over $(\mathbb{C}, \mathscr{O}_{X,\x}).$ The $\mathscr{O}_{X,\x}$-map $\psi_{F}: \Der_{X,\x}(-\log f) \to \theta_{F,\x}$ is an isomorphism of  Lie--Rinehart algebras over $(\mathbb{C}, \mathscr{O}_{X,\x})$. So there is a containment of Lie--Rinehart algebras over $(\mathbb{C}, \mathscr{O}_{X,\x})$: $\theta_{F,\x}  \subseteq \Der_{X,\x} \oplus \mathscr{O}_{X,\x}^{r}$.
    \item The universal algebra of the  Lie--Rinehart algebra $\Der_{X,\x}[S]$ over $(\mathbb{C}[S], \mathscr{O}_{X,\x}[S])$ is $\D_{X,\x}[S]$. \iffalse Similarly, the universal algebra of the Lie--Rinehart algebra $(\Der_{X,\x}(-\log f)[S]$ over $(\mathbb{C}[S], \mathscr{O}_{X,\x}[S])$ is $\D_{X,\x}(-\log f)[S]$. \fi Note that $s_{k}$ is contained in the $0^{\text{th}}$ filtered part and the filtration is induced by the order filtration. \iffalse So we can get isomorphic (as rings) universal algebras out of two Lie--Rinehart algebras that are not even defined over the same ``ground" objects. \fi
\end{enumerate}
\end{example}

\iffalse

\begin{remark} 
In view of Example \ref{example universal LR algebras} (d), note that the action of $\D_{X,\x}[S]$ on $F^{S}$ induces an action of $\D_{X,\x}(-\log f)[s]$ on $F^{S}$. When $f$ is free, the explicit description of $\D_{X,\x}(-\log f)$ and Proposition \ref{prop psi} imply that $\D_{X,\x}(-\log f)$ sends $F^{S}$ into $\mathscr{O}_{X,\x}[S]F^{S}$. In this setting, $\mathscr{O}_{X,\x}[S]F^{S}$ is a $\D_{X,\x}(-\log f)$-module.

\iffalse
\item Note Example \ref{example universal LR algebras} (c) and (f). We will be exploiting the fact that $\D_{X,\x}[S]$ arises as a universal algebra over two different Lie--Rinehart algebras. We will use the additional structure of $U(\Der_{X,\x} \oplus \mathscr{O}_{X,\x}^{r})$ to define resolutions of $\D_{X,\x}[S]$-modules. On the other hand, given the complexes of $\D_{X,\x}[S]$-modules we will define isomorphisms and functors using the added structure of $U(\Der_{X,\x}[S])$.
\fi
\end{remark}
\fi

We care about the formalism of Lie--Rinehart algebras because we want to construct complexes of the universal algebras. Given two Lie--Rinehart algebras $L \subseteq L^{\prime}$, the following gives a complex of $U(L^{\prime})$-modules. 
\iffalse and $U(L^{\prime \prime})$-modules, and identify a straightforward way to compare the two complexes. 
\fi
\begin{definition} \label{def LR complex}
(Compare with 1.1.8 of \cite{MorenoMacarroLogarithmic}) Let $L$ and $L^{\prime}$ be Lie--Rinehart algebras over $(k, A).$ The \emph{Cartan--Eilenberg--Chevalley--Rinehart--Spencer co-complex} associated to $L \subseteq L^{\prime}$ and the left $U(L)$-module $E$ is the co-complex $\Sp_{L,L^{\prime}}(E)$. Here \[\text{Sp}^{-r}(L,L^{\prime}) := U(L^{\prime}) \otimes_{A} \bigwedge^{r} L \otimes_{A} E\]and the $U(L^{\prime})$-linear differential 
\[d^{-r}: \text{Sp}_{L,L^{\prime}}^{-r}(E) \to \text{Sp}_{L,L^{\prime}}^{-(r-1)}(E)\] is given by
\begin{align} \label{eqn- cartan complex}
    d^{-r}(P \otimes \lambda_{1} \wedge \dots \wedge \lambda_{r} \otimes {e}) = &\sum\limits_{i=1}^{r} (-1)^{i-1} P \lambda_{i} \otimes \widehat{\lambda_{i}} \otimes e - \sum\limits_{i=1}^{r} (-1)^{i-1} P \otimes \widehat{\lambda_{i}} \otimes \lambda_{i}e \\
    &+ \sum\limits_{1 \leq i < j \leq r} (-1)^{i+j} P \otimes [\lambda_{i}, \lambda_{j}] \wedge \widehat{\lambda_{i,j}} \otimes e \nonumber.
\end{align}
(Here $\widehat{\lambda_{i,j}}$ is the wedge of the of all the $\lambda$'s except $\lambda_{i}$ and $\lambda_{j}$.)
\iffalse 
There is a natural augmentation map $U(L^{\prime}) \otimes_{A} E \to U(L^{\prime}) \otimes_{U(L)} E.$ 
Moreover, if $L \subseteq L^{\prime} \subseteq L^{\prime \prime}$ then 
$$ \Sp_{L, L^{\prime \prime}}(E) = U(L^{\prime \prime}) \otimes_{U(L^{\prime})} \Sp_{L, L^{\prime}}(E).$$
\fi
There is a natural augmentation map
\[
U(L^{\prime}) \otimes_{A} E \to U(L^{\prime}) \otimes_{U(L)} E.
\]When $E = A$, write $\Sp_{L, L^{\prime}}(A)$ as $\Sp_{L, L^{\prime}}.$

\end{definition}

In general, the cohomology of $\Sp_{L, L^{\prime}}(E)$ is mysterious. In principal, it can be computed using the spectral sequence associated to the filtration of $U(L^{\prime})$ promised by Theorem \ref{thm- universal algebra}. In the classical case of a Lie algebra, the Poincar\'e-Birkhoff-Witt theorem says that the natural associated graded ring of universal algebra of $g$ is canonically isomorphic (as algebras) to the symmetric algebra of $g$. Rinehart proved, cf. \cite{Rinehart}, the analogous result for $L$: the natural associated graded ring of $U(L)$ is isomorphic to $\text{Sym}_{A}(L)$. A spectral sequence argument gives the following: \iffalse uIn \cite{Rinehart} a Poincare-Birkhoff-Whitt theorem for these universal algebras, the associated graded of $U(L^{\prime})$ is $\text{Sym}_{A}(L^{\prime})$. So the $0^{\text{th}}$ page of this spectral sequence can be described in terms of a co-complex involving the symmetric algebra. When $L$ has a basis whose symbols in $\text{Sym}_{A}(L^{\prime})$ constitute a regular sequence and when $E$ is a free $A$-module, the $0^{\text{th}}$ page of this spectral sequence looks like a Koszul co-complex of a regular sequence. Arguing in this fashion gives the following: \fi

\begin{prop} \text{\normalfont{(Proposition 1.5.3 in \cite{MorenoMacarroLogarithmic})}} \label{prop koszul LR resolution}
Suppose $L \subseteq L^{\prime}$ are Lie--Rinehart algebras over $(k,A)$ and $E$ a left $U(L)$-module free over $A$. Moreover, suppose $L$, $L^{\prime}$ are free $A$-modules of finite rank such that a basis of $L$ forms a regular sequence in the symmetric algebra $\text{Sym}_{A}(L^{\prime})$. Then $\Sp_{L,L^{\prime}}(E)$ is a finite free $U(L^{\prime})$-resolution of $U(L^{\prime}) \otimes_{U(L)} E$.
\end{prop}

We may use Proposition \ref{prop koszul LR resolution} to resolve $\D_{X,\x}[S]F^{S}$, provided $f$ is nice enough:

\begin{prop} \label{prop many sp complexes} \label{LR resolution facts}

Suppose $f = f_{1} \cdots f_{r}$ is reduced, strongly Euler-homogeneous, Saito-holonomic, and free and let $F=(f_{1}, \dots, f_{r}).$ Then $\Sp_{\theta_{F,\x}, \Der_{X,\x} \oplus \mathscr{O}_{X,\x}^{r}}$ is a free $\D_{X,\x}[S]$-resolution of $\D_{X,\x}[S]F^{S}$. 

\iffalse Then the following are true:
\begin{enumerate}[(a)] 
    \item $\Sp_{\theta_{F,\x}, \Der_{X,\x}(-\log f) \oplus \mathscr{O}_{X,\x}^{r}}$ is a free $\D_{X,\x}(-\log f)[S]$-resolution of $\mathscr{O}_{X,\x}[S]F^{S}$.
    \item $\Sp_{\theta_{F,\x}, \Der_{X,\x} \oplus \mathscr{O}_{X,\x}^{r}}$ is a free $\D_{X,\x}[S]$-resolution of $\D_{X,\x}[S]F^{S}$.
    \item $\D_{X,\x}[S] \otimes_{\D_{X,\x}(-\log f)[S]}^{L} \mathscr{O}_{X,\x}[S]F^{S} = \Sp_{\theta_{F,\x}, \Der_{X,\x} \oplus \mathscr{O}_{X,\x}^{r}}$.
\end{enumerate}
\fi
\end{prop}

\begin{proof}

We argue as in Section 1.6 of \cite{MorenoLogarithmicDifferential}. First, note that by Proposition 6.3 of \cite{BruceRoberts} and Corollary 1.9 of \cite{MorenoMacarroQuasi-Homogeneous}, that for reduced free divisors being Saito-holonomic is equivalent to being Koszul free, where Koszul free means there is a basis $\delta_{1}, \dots, \delta_{n}$ of $\Der_{X,\x}(-\log f)$ that $\gr_{(0,1)}(\delta_{1}) \dots \gr_{(0,1)}(\delta_{n})$ is a regular sequence in $\gr_{(0,1)}(\D_{X,\x}).$ Let $\delta_{1}, \dots, \delta_{n}$ be such a basis. Then $s_{1}, \dots, s_{n}, \psi_{F,\x}(\delta_{1}), \dots, \psi_{F,\x}(\delta_{n})$ is a regular sequence in $\gr_{(0,1,1)}(\D_{X,\x}[S])$. As these elements are all $(0,1,1)$-homogeneous, we may rearrange them and conclude $\psi_{F,\x}(\delta_{1}), \dots, \psi_{F,\x}(\delta_{n})$ is a regular sequence in $\gr_{(0,1,1)}(\D_{X,\x}[S]) \simeq \text{Sym}_{\mathscr{O}_{X,\x}}(\Der_{X,\x} \oplus \mathscr{O}_{X,\x}^{r}).$ Now Proposition \ref{prop koszul LR resolution} implies that $\Sp_{\theta_{F,\x}, \Der_{X,\x} \oplus \mathscr{O}_{X,\x}^{r}}$ is a free $\D_{X,\x}[S]$-resolution and inspecting the terminal map of this co-complex shows it resolves $\D_{X,\x}[S] / \D_{X,\x}[S] \cdot \theta_{F,\x}$, which, by Theorem \ref{thm gen by derivations}, is isomorphic to $\D_{X,\x}[S] F^{S}$. 
\end{proof}

When $f$ is strongly Euler-homogeneous, Saito-holonomic, and tame we showed in Proposition \ref{prop Koszul resolution} that there is a Koszul co-complex resolution of $\D_{X,\x}[S] F^{S} / (S-A) \D_{X,\x}[S] F^{S}.$ Using $\Sp_{\theta_{F,\x}, \Der_{X,\x} \oplus \mathscr{O}_{X,\x}^{r}}$ we construct a free $\D_{X,\x}[S]$-resolution of $\D_{X,\x}[S] F^{S} / (S-A) \D_{X,\x}[S] F^{S}$.

\begin{prop} \label{prop the LR resolution}
Suppose $f = f_{1} \cdots f_{r}$ is reduced, strongly Euler-homogeneous, Saito-holonomic, and free and let $F=(f_{1}, \dots, f_{r}).$ Then there is a finite, free resolution of $\D_{X,\x}[S]$-modules
\[\frac{\D_{X,\x}[S]}{(S-A)\D_{X,\x}[S]} \otimes_{\D_{X,\x}[S]} \Sp_{\theta_{F,\x}, \Der_{X,\x} \oplus \mathscr{O}_{X,\x}^{r}} \to \frac{\D_{X,\x}[S]F^{S}}{(S-A)\D_{X,\x}[S]}.\]
\end{prop}

\begin{proof}

%We first must show $(S-A)\D_{X,\x}[S] \ \otimes_{\D_{X,\x}[S]} \Sp_{\theta_{F,\x}, \Der_{X,\x} \oplus \mathscr{O}_{X,\x}^{r}}$ is acylic. Since $\Sp_{\theta_{F,\x}, \Der_{X,\x} \oplus \mathscr{O}_{X,\x}^{r}}$ resolves $\D_{X,\x}[S]F^{S}$,
By Proposition \ref{prop many sp complexes}, it is enough to prove that, for $k \geq 1$,  
\[
\text{Tor}_{\D_{X,\x}[S]}^{k}(\frac{\D_{X,\x}[S]}{(S-A)\D_{X,\x}[S]}, \D_{X,\x}[S]F^{S}) = 0.
\]
As $K^{\bullet}(S-A; \D_{X,\x}[S])$ resolves $\D_{X,\x}[S] / (S-A) \D_{X,\x}[S]$, use Proposition \ref{prop Koszul resolution}.
\iffalse
Since $\Sp_{\theta_{F,\x}, \Der_{X,\x} \oplus \mathscr{O}_{X,\x}^{r}}$ is a finite free resolution of $\D_{X,\x}[S]F^{S}$ by $\D_{X,\x}[S]$-modules, this claim amounts to showing  \iffalse higher (i.e. $k \geq 1)$ \fi $\text{Tor}_{\D_{X,\x}[S]}^{k}(\frac{\D_{X,\x}[S]}{(S-A)\D_{X,\x}[S]}, \D_{X,\x}[S]F^{S})$ vanishes for $k \geq 1$. We can compute this by taking a free $\D_{X,\x}[S]$-resolution of the first argument.

Mimicking the first part of the proof of Proposition \ref{prop Koszul resolution}, use the $F_{(0,1,1)}$ filtration on $\D_{X,\x}[S]$ to define a filtration on the Koszul co-complex $K^{\bullet}(S-A)$ of $\D_{X,\x}[S]$-modules. The only interesting part of spectral sequence associated to this filtered co-complex is the Koszul co-complex $K^{\bullet}(S)$ on $\mathscr{O}_{X,\x}[Y][S],$ which is acyclic. So $K^{\bullet}(S-A)$ resolves $\frac{\D_{X,\x}[S]}{(S-A)\D_{X,\x}[S]}.$

That \iffalse By Proposition \ref{prop Koszul resolution}, \fi $K^{\bullet}(S-A; \D_{X,\x}[S]F^{S})$ resolves $\frac{\D_{X,\x}[S]}{(S-A)\D_{X,\x}[S]} \otimes_{\D_{X,\x}[S]} \D_{X,\x}[S]F^{S}$ is the content of Proposition \ref{prop Koszul resolution}. \iffalse $\simeq \frac{\D_{X,\x}[S]F^{S}}{(S-A)\D_{X,\x}[S])}$ \fi So $\text{Tor}_{\D_{X,\x}[S]}^{k}(\frac{\D_{X,\x}[S]}{(S-A)\D_{X,\x}[S])}, \D_{X,\x}[S]F^{S})$ vanishes for $k \geq 1$ and the claim is proved.
\fi
\end{proof}
\iffalse
we can take a $\D_{X,\x}[S]$ resolution of the first or second argument. 

Proceed as in Proposition \ref{prop Koszul resolution}. Take the $F_{(0,1,1)}$ filtration on $\D_{X,\x}[S]$ and use this to define a filtration on $K^{\bullet}(S-A).$ Then the $0^{\text{th}}$ page of the associated spectral sequence is $K^{\bullet}(S, \mathscr{O}_{X,\x}[Y][S])$. This is acylic. Therefore $K^{\bullet}(S-A)$ is a free $\D_{X,\x}[S]$-resolution of $\frac{\D_{X,\x}[S]}{(S-A), \D_{X,\x}[S])}.$ So we can compute the above Tor using this resolution. Propositon \ref{prop Koszul resolution} says
$$K^{\bullet}(S-A) \otimes_{\D_{X,\x}[S]} \D_{X,\x}[S]F^{S} = K^{\bullet}(S-A, \D_{X,\x}[S]F^{S})$$
is acylic. Therefore 
$$\text{Tor}_{\D_{X,\x}[S]}^{\bullet}(\frac{\D_{X,\x}[S]}{(S-A\D_{X,\x}[S])}, \D_{X,\x}[S]F^{S}) = 0.$$

Now use the free $\D_{X,\x}[S]$-resolution $\Sp_{\theta_{F,\x}, \Der_{X,\x} \oplus \mathscr{O}_{X,\x}^{r}}$ of $\D_{X,\x}[S]F^{S}$ to compute this Tor. \iffalse 
We 
$$(S-A)\D_{X,\x}[S] \otimes_{\D_{X,\x}[S]} \Sp_{\theta_{F,\x}, \Der_{X,\x} \oplus \mathscr{O}_{X,\x}^{r}}.$$
\fi
Since our Tor groups vanish, $\frac{\D_{X,\x}[S]} {(S-A)\D_{X,\x}[S]} \otimes_{\D_{X,\x}[S]} \Sp_{\theta_{F,\x}, \Der_{X,\x} \oplus \mathscr{O}_{X,\x}^{r}}$ must be acyclic. Therefore it resolves
$\frac{\D_{X,\x}[S]}{(S-A)\D_{X,\x}[S]} \otimes_{\D_{X,\x}[S]} \D_{X,\x}[S]F^{S} \simeq \frac{\D_{X,\x}[S]F^{S}}{(S-A)\D_{X,\x}[S])}$.
\fi

\subsection{Dual of $\D_{X,\x}[S] F^{S} / (S-A) \D_{X,\x}[S] F^{S}$.} \text{ }

Now that we have resolutions, we can proceed to our first goal: to compute the $\D_{X,\x}$-dual of $\frac{\D_{X,\x}[S]F^{S}}{(S-A)\D_{X,\x}[S]F^{S}}.$ 
\begin{definition}
(Compare with Appendix A of \cite{MacarroDuality}) Consider a Lie--Rinehart algebra $L$ over $(k,A)$ that is $A$-projective of constant rank $n$. There is an equivalence of categories from right $U(L)$-modules $Q$ to the left $U(L)$-modules given by $Q^{\text{left}} = \Hom_{A}(w_{L}, Q)$ where $w_{L}$ is the dualizing module of $L$, namely, $w_{L} = \Hom_{A}(\bigwedge^{n} L, A)$. Regard $\D_{X,\x}$ as the universal algebra of the Lie--Rinehart algebra $\Der_{X,\x}$ over $(\mathbb{C}, \mathscr{O}_{X,\x})$ and $\D_{X,\x}[S]$ as the universal algebra of the  Lie--Rinehart algebra $\Der_{X,\x}[S]$ over $(\mathbb{C}[S], \mathscr{O}_{X,\x}[S])$. \iffalse 
and $\D_{X,\x}(-\log f)[S]$ as the universal algebra of the $(\mathbb{C}[S], \mathscr{O}_{X,\x}[S]$ Lie--Rinehart algebra $\Der_{X,\x}(-\log f)[S]$, 
\fi 
In the appropriate derived category of left modules, where $N$ is a left $U(L)$-module, let:
\begin{align*}
    & \mathbb{D}(N) := (\text{RHom}_{\D_{X,\x}}(N, \D_{X,\x})^{\text{left}}; \\
    & \mathbb{D}_{S}(N) := (\text{RHom}_{\D_{X,\x}[S]}(N, \D_{X,\x}[S])^{\text{left}}. 
\end{align*}
    \iffalse
    & \mathbb{V}_{S}(M) = (\text{RHom}_{\D_{X,\x}(-\log f)[S]}(M, \D_{X,\x}(-\log f)[S])^{\text{left}}.
    \fi
\iffalse To be clear: to compute $\mathbb{D}(N)$ take an appropriate resolution of left $\D_{X,\x}$-modules of $N$, apply $\Hom_{\D_{X,\x}}(-,\D_{X,\x})$, and then apply the functor $(-)^{\text{left}}$.
\fi
\end{definition}

The following demystifies how $(-)^{\text{left}}$ works for the above universal algebras. Its proof is entirely similar to the classical case of $(-)^{\text{left}}$ for $\D_{X,\x}$-modules.

\begin{lemma} \label{lemma equivalence}

Take a $\ell \times m$ matrix $M$ with entries in $\D_{X,\x}[S]$ so that multiplication on the left gives a map of right $\D_{X,\x}[S]$-modules $\D_{X,\x}[S]^{m} \to \D_{X,\x}[S]^{\ell}$. Here an element $\D_{X,\x}[S]^{m}$ is a column vector. For some fixed coordinate system, define the map $\tau: \D_{X,\x}[S] \to \D_{X,\x}[S]$, $\tau (x^{\textbf{u}} \partial^{\textbf{v}} s^{\textbf{w}}) =  (- \partial)^{\textbf{v}} x^{\textbf{u}} s^{\textbf{w}}.$ Extend $\tau$ to $\D_{X,\x}[S]^{m}$ in an obvious way and to $M$ by applying $\tau$ to each entry. Then there is a commutative diagram of left $\D_{X,\x}[S]$-modules, where elements in the bottom row are row vectors and $(-)^{T}$ denotes the transpose:
\[
\begin{tikzcd}
    (\D_{X,\x}[S]^{m})^{\text{\normalfont left}} \arrow[r, "M^{\text{\normalfont left}}"] \arrow[d, "\simeq"]
        & (\D_{X,\x}[S]^{\ell})^{\text{\normalfont left}} \arrow[d, "\simeq"]\\
    \D_{X,\x}[S]^{m} \arrow[r, "\boldsymbol{\cdot} \tau(M)^{T}"]
        & \D_{X,\x}[S]^{\ell}.
\end{tikzcd}
\]

Given a right $\D_{X,\x}$-linear map $M: \D_{X,\x}^{m} \to \D_{X,\x}^{\ell}$, there is an entirely similar commutative diagram of left-$\D_{X,\x}$ modules (where $\tau$ has the obvious definition).
\iffalse
\[
\begin{tikzcd}
    (\D_{X,\x}^{m})^{\text{\normalfont left}} \arrow[r, "M^{\text{\normalfont left}}"] \arrow[d, "\simeq"]
        & (\D_{X,\x}^{\ell})^{\text{\normalfont left}} \arrow[d, "\simeq"]\\
    \D_{X,\x}^{m} \arrow[r, "\boldsymbol{\cdot} \tau(N)^{T}"]
        & \D_{X,\x}^{\ell}.
\end{tikzcd}
\]
\fi
\end{lemma}

The first step in computing $\mathbb{D}(\frac{\D_{X,\x}[S]F^{S}}{(S-A\D_{X,\x}[S])})$ is finding a resolution--this is Proposition \ref{prop the LR resolution}. The second is the following technical lemma:

\begin{lemma} \label{LR lemma}
Suppose $f = f_{1} \cdots f_{r}$ is reduced, strongly Euler-homogeneous, Saito-holonomic, and free and let $F=(f_{1}, \dots, f_{r}).$ As complexes of free $\D_{X,\x}$-modules,
\begin{align*}
    \mathbb{D} \Big( \frac{\D_{X,\x}[S]}{(S-A)\D_{X,\x}[S]} & \otimes_{\D_{X,\x}[S]} \Sp_{\theta_{F,\x}, \Der_{X,\x} \oplus \mathscr{O}_{X,\x}^{r}}  \Big) \\
        & \simeq \frac{\D_{X,\x}[S]}{(S-A)\D_{X,\x}[S]} \otimes_{\D_{X,\x}[S]} \mathbb{D}_{S}\left(\Sp_{\theta_{F,\x}, \Der_{X,\x} \oplus \mathscr{O}_{X,\x}^{r}}\right)
\end{align*}
\end{lemma}

\begin{proof}
For brevity, abbreviate $\Sp_{\theta_{F,\x}, \Der_{X,\x} \oplus \mathscr{O}_{X,\x}^{r}}$ to $\SP^{\bullet}$. Write the differential as $d^{-k}: \SP^{-k} \to \SP^{-(k-1)}.$

We will first compute the objects and maps of $\mathbb{D}\left( \frac{\D_{X,\x}[S]}{(S-A)\D_{X,\x}[S]} \otimes_{\D_{X,\x}[S]} \Sp \right).$ Since $\Sp$ is a co-complex of finite, free $\D_{X,\x}[S]$-modules,  $\SP^{-k} \simeq \D_{X,\x}[S]^{\binom{n}{k}}.$ Therefore, as $\D_{X,\x}[S]$-modules,
\begin{equation} \label{eqn pre-A complex}
\frac{\D_{X,\x}[S]}{(S-A)\D_{X,\x}[S]} \otimes_{\D_{X,\x}[S]} \SP^{-k} \simeq \frac{\D_{X,\x}[S]}{(S-A)\D_{X,\x}[S]}^{\binom{n}{k}}.
\end{equation}
On the LHS of \eqref{eqn pre-A complex}, we have the differential $\frac{\D_{X,\x}[S]}{(S-A)\D_{X,\x}[S]} \otimes_{\D_{X,\x}[S]} d^{-k}$. Think of $d^{-k}$ as a matrix. On the RHS of \eqref{eqn pre-A complex} the differential is $\text{eval}_{A}({d^{-k}})$: the matrix $d^{-k}$ except each $s_{i}$ is replaced with $a_{i}$. As right $\D_{X,\x}$-modules, 
\begin{align*} 
    \Hom_{\D_{X,\x}} \left( \frac{\D_{X,\x}[S]}{(S-A)\D_{X,\x}[S]} \otimes_{\D_{X,\x}[S]} \text{Sp}^{-k}, \D_{X,\x} \right) & \simeq \Hom_{\D_{X,\x}} \left( \D_{X,\x}^{\binom{n}{k}}, \D_{X,\x} \right) \nonumber \\
        &\simeq \D_{X,\x}^{\binom{n}{k}}.
\end{align*}
Making the above identification, $\Hom_{\D_{X,\x}} \left( \frac{\D_{X,\x}[S]}{(S-A)\D_{X,\x}[S]} \otimes_{\D_{X,\x}[S]} \Sp, \D_{X,\x} \right)$ has a differential given by multiplication on the left by $(\text{eval}_{A}(d^{-k}))^{T}$--the transpose of $\text{eval}_{A}(d^{-k}).$ To make the Hom complex a complex of left modules we apply the equivalence of categories $(-)^{\text{left}}$. By Lemma \ref{lemma equivalence} we get a complex of left $\D_{X,\x}$ modules isomorphic to the following, with differential given by matrix multiplication on the right
$$ A_{\bullet} := \dots \xrightarrow{} \D_{X,\x}^{\binom{n}{k-1}} \xrightarrow{\boldsymbol{\cdot} \tau((\text{eval}_{A}(d^{-k}))^{T})^{T}} \D_{X,\x}^{\binom{n}{k}} \to \dots .$$
%Identifying $\Hom_{\mathscr{O}_{X,\x}}(w_{\Der_{X,\x}}, \D_{X,\x}^{n_{-k}})$ with $\D_{X,\x}^{n_{-k}}.$ After making this identification, the differentials on the left Hom complex are still the matrices $\overline{d_{-k}}^{t}.$ We have shown that $\mathbb{D}\left( \frac{\D_{X,\x}[S]}{(S-A)\D_{X,\x}[S]} \otimes_{\D_{X,\x}[S]} \Sp \right) \simeq_{\D_{X,\x}} A^{\bullet}$, where $A^{\bullet}$ is the complex given by $A^{-k} = \D_{X,\x}^{n_{-k}}$ and the differentials $\overline{d_{-k}}^{t}.$

Now we compute the objects and maps of $\frac{\D_{X,\x}[S]}{(S-A)\D_{X,\x}[S]} \otimes_{\D_{X,\x}[S]} \mathbb{D}_{S}(\Sp).$ As right $\D_{X,\x}[S]$-modules, $\Hom_{\D_{X,\x}[S]} \left( \SP^{-k}, \D_{X,\x}[S] \right) \simeq \D_{X,\x}[S]^{\binom{n}{k}}.$ The induced differential is multiplication on the left by $(d^{-k})^{T}.$ By Lemma \ref{lemma equivalence}, we can identify the complex obtained by applying $(-)^{\text{left}}$ with a complex whose terms are $\D_{X,\x}[S]^{\binom{n}{k}}$ and whose differentials are $\tau((d^{-k})^{T})^{T}.$ As left $\D_{X,\x}[S]$-modules (and so as left $\D_{X,\x}$-modules),
\begin{equation} \label{eqn pre-B complex}
\frac{\D_{X,\x}[S]}{(S-A)\D_{X,\x}[S]} \otimes_{\D_{X,\x}[S]} \D_{X,\x}[S]^{\binom{n}{k}} \simeq \frac{\D_{X,\x}[S]}{(S-A)\D_{X,\x}[S]}^{\binom{n}{k}}.
\end{equation}
The RHS of \eqref{eqn pre-B complex} is isomorphic as a left $\D_{X,\x}$-module to $\D_{X,\x}^{\binom{n}{k}}.$ With this identification, the differentials of the complex $\frac{\D_{X,\x}[S]}{(S-A)\D_{X,\x}[S]} \otimes_{\D_{X,\x}[S]} \mathbb{D}_{S}(\Sp)$ are given by $\text{eval}_{A}(\tau((d^{-k})^{T})^{T}).$ Thus the complex of left $\D_{X,\x}[S]$-modules $\frac{\D_{X,\x}[S]}{(S-A)\D_{X,\x}[S]} \otimes_{\D_{X,\x}[S]} \mathbb{D}_{S}(\Sp)$ is isomorphic as a complex of left $\D_{X,\x}$-modules to
\[ B_{\bullet} := \dots \xrightarrow{} \D_{X,\x}^{\binom{n}{k-1}} \xrightarrow{\boldsymbol{\cdot} \text{eval}_{A}(\tau((d^{-k})^{T})^{T})} \D_{X,\x}^{\binom{n}{k}} \to \dots .
\]

%\simeq_{\D_{X,x}[S]} B^{\bullet}$ where $B^{\bullet}$ is the complex given by $B^{-k} =\frac{\D_{X,x}[S]}{(S-A)\D_{X,x}[S]}^{n_{-k}}$ and the differentials  $\overline{d_{-k}^{t}}.$

We will be done once we show that $A_{\bullet}$ and $B_{\bullet}$ are isomorphic complexes of $\D_{X,\x}$-modules. Because $\tau((\text{eval}_{A}(d^{-k}))^{T})^{T} = \tau(\text{eval}_{A}(d^{-k})) = \text{eval}_{A}(\tau(d^{-k})) = \text{eval}_{A}(\tau((d^{-k})^{T})^{T})$, $A_{\bullet}$ and $B_{\bullet}$ have the same differentials. 
\end{proof}

%$B^{\bullet}$ is also a complex of $\D_{X,\x}$-modules. Since
%$\frac{\D_{X,\x}[S]}{(S-A)\D_{X,\x}[S]}^{n_{-k}} \simeq_{\D_{X,\x}} \D_{X,\x}^{n_{-k}}$, as a complex of $\D_{X,\x}$-modules each $B^{-k} \simeq \D_{X,\x}^{n_{-k}}$. The maps do not change. Therefore as complexes of $\D_{X,\x}$-modules, $A^{\bullet} \simeq B^{\bullet}$. Indeed $A^{-k} = B^{-k}$ and $\overline{d_{-k}}^{t} = \overline{d_{-k}^{t}}$ (you get the same matrix if you first replace $s_{i}$'s with $a_{i}$'s and then take a transpose or if you first take the transpose and then replace $s_{i}$'s with $a_{i}$'s). 

%This precisely gives the $\D_{X,\x}$-isomorphisms proving (as $\D_{X,\x}$-modules)
%\begin{align*}
%    \mathbb{D} \Big( \frac{\D_{X,\x}[S]}{(S-A)\D_{X,\x}[S]} & \otimes_{\D_{X,\x}[S]} \Sp_{\theta_{F,\x}, \Der_{X,\x} \oplus \mathscr{O}_{X,\x}^{r}}  \Big) \\
%        & \simeq \frac{\D_{X,\x}[S]}{(S-A)\D_{X,\x}[S]} \otimes_{\D_{X,\x}[S]} \mathbb{D}_{S}\left(\Sp_{\theta_{F,\x}, \Der_{X,\x} \oplus \mathscr{O}_{X,\x}^{r}}\right).
%\end{align*}

So we have reduced our problem to, in light of Proposition \ref{LR resolution facts}, computing $\mathbb{D}_{S}(\D_{X,\x}[S]F^{S}).$ In Corollary 3.6 of \cite{MacarroDuality} Narv\'aez--Macarro does this for $\D_{X,\x}[s]f^{s}$ with similar working hypotheses as ours and Maisonobe shows in \cite{MaisonobeArrangement} that this result generalizes to $\D_{X,\x}[S]F^{S}$ as well. \iffalse Because of Lie--Rinehart formalism, this amounts to arguing as in Proposition \ref{LR resolution facts} to prove a version of Corollary 3.6 in \cite{MacarroDuality}. (Note that locally quasi-homogeneous immediately implies locally strongly Euler-homogeneous and that for reduced free divisors, being Saito-holonomic is equivalent to being Koszul free, cf. Proposition 6.3 in \cite{BruceRoberts} and Corollary 1.9 in \cite{MorenoMacarroQuasi-Homogeneous}.) \fi In our language, cf. the proof of Proposition \ref{prop many sp complexes}, this result is as follows:

\iffalse We are almost ready for the formalism of Lie--Rinehart algebras to bear fruit. We will need the one more brief definition and Proposition which come from Section 2 in \cite{MacarroDuality}.

\begin{definition}
Let $\varphi$ be a $\mathbb{C}$-linear automorphism of $\mathbb{C}[S]$ that sends each $s_{k}$ to a linear polynomial in $s_{k}$ alone. Then $\D_{X,\x}[S]$ acts on $F^{\varphi(S)} = \prod\limits_{k} f_{k}^{\varphi(s_{k})}$ in the same way $D_{X,\x}[S]$ acts on $F^{S}$: just replace every instance of $s_{k}$ with $\varphi(s_{k})$. This makes $\D_{X,\x}[S]F^{\varphi(S)}$ a $\D_{X,\x}[S]$-module and $\mathscr{O}_{X,\x}[S]F^{\varphi(S)}$ a $\D_{X,\x}(- \log f)[S]$-module. 
\end{definition}

\begin{prop} \label{prop phi tensor iso}
Lef $f$ be strongly Euler-homogeneous, Saito-holonomic, and free. Then for $\varphi$ as described above there is the isomoprhism in the category of derived $D_{X,\x}[S]$-modules
$$ \D_{X,\x}[S] \otimes_{\D_{X,\x}(-\log f)[S]}^{L} \D_{X,\x}(- \log f)[S] F^{\varphi(S)} \simeq \D_{X,\x}[S]F^{\varphi(S)}.$$
\end{prop}
\begin{proof}
Our hypothesis guarantee $\ann_{\D_{X,\x}[S]}F^{S} = \theta_{F,\x}$. So using Proposition \ref{LR resolution facts} (b) and (c), the proof works in exactly the same way as Proposition 2.7 in \cite{MacarroDuality}. 
\end{proof}

We are finally in position to use the entirety of the appendix of \cite{MacarroDuality} to explicitly compute one of the dual complexes appearing in Lemma \ref{LR lemma}. The statement and proof mimic Corollary 3.6 in \cite{MacarroDuality}.
\fi

\begin{prop} \text{\normalfont{(Proposition 6 in \cite{MaisonobeArrangement})}} \label{prop Ds dual}
Let $f=f_{1} \cdots f_{r}$ be reduced, strongly Euler-homogeneous, Saito-holonomic, and free and let $F=(f_{1}, \dots, f_{r}).$ Then, in the category of left derived $\D_{X,\x}[S]$-modules, there is a canonical isomorphism
$$\mathbb{D}_{S}(\D_{X,\x}[S]F^{S}) \simeq \D_{X,\x}[S]F^{-S-1}[n].$$
\end{prop}

\begin{theorem} \label{thm nabla dual} Suppose $f = f_{1} \cdots f_{r}$ is reduced, strongly Euler-homogeneous, Saito-holonomic, and free and let $F=(f_{1}, \dots, f_{r}).$ Then in the derived category of $\D_{X,\x}$-modules there is a $\D_{X,\x}$-isomorphism $\chi_{A}$ given by
\begin{align} \label{eqn iso thm nabla dual}
    \chi_{A}: \mathbb{D} \left( \frac{\D_{X,\x}[S] F^{S}}{(S-A)\D_{X,\x}[S] F^{S}} \right) & \xrightarrow{\simeq} \frac{\D_{X,\x}[S] F^{-S-1}}{(S-A)\D_{X,\x}[S]F^{-S-1}}[n] \\
    & \xrightarrow{\simeq} \frac{\D_{X,\x}[S] F^{S}}{(S-(-A-1)\D_{X,\x}[S]F^{S}}[n]. \nonumber
\end{align}

\end{theorem}

\begin{proof}
The $\D_{X,\x}$-linear involution on $\D_{X,\x}[S]$ defined by sending each $s_{k} \mapsto -s_{k} -1$ induces a $\D_{X,\x}$-linear map $\D_{X,\x}[S]F^{-S-1} \simeq \D_{X,\x}[S]F^{S}$. This gives the second isomorphism of \eqref{eqn iso thm nabla dual}. Considerations using this map and Proposition \ref{prop many sp complexes} show that $\text{Tor}_{\D_{X,\x}[S]}^{k} \left( \frac{\D_{X,\x}[S] F^{S}}{(S-A)\D_{X,\x}[S] F^{S}}, \D_{X,\x}[S] F^{-S-1} \right)$ vanishes for $k \geq 1.$ Proposition \ref{prop Ds dual} then implies the acylicity of the augmented co-complex
\[
\frac{\D_{X,\x}[S]}{(S-A)\D_{X,\x}[S]} \otimes_{\D_{X,\x}[S]} \mathbb{D}_{S} \left( \Sp_{\theta_{F,\x}, \Der_{X,\x} \oplus \mathscr{O}_{X,\x}^{r}} \right) \to \frac{\D_{X,\x}[S]F^{-S-1}}{(S-A)\D_{X,\x}[S]F^{-S-1}}.
\]
This, Proposition \ref{prop the LR resolution}, and Lemma \ref{LR lemma} give the first isomorphism of \eqref{eqn iso thm nabla dual}.
\end{proof}

\begin{remark} \label{rmk holonomic dual}
When $f$ is reduced, strongly Euler-homogeneous, Saito-holonomic, and free, this immediately implies $\D_{X,\x}[S]F^{S}/ (S-A)\D_{X,\x}[S]F^{S}$ is a holonomic $\D_{X,\x}$-module. \iffalse --it has only one nonzero $\text{Ext}$-module sitting in the $-n^{\text{th}}$ position. When $f$ is tame but not free computing $\text{Ext}$ is currently intractable. \fi Without freeness, computing $\text{Ext}$ is currently intractable.
\end{remark}

\subsection{Free Divisors and $\nabla_{A}$.} \text{ }

Recall from Definition \ref{nabla definition} the $\D_{X,\x}$-linear map
$$\nabla_{A}: \frac{\D_{X,\x}[S] F^{S}}{(S-A)\D_{X,\x}[S] F^{S}} \to \frac{\D_{X,\x}[S] F^{S}}{(S-(A-1))\D_{X,\x}[S] F^{S}}$$
induced by $s_{k} \mapsto s_{k} +1$, for each $k$. If $f$ is reduced, strongly Euler-homogeneous, Saito-holonomic, and free, by Proposition \ref{thm nabla dual} the complexes $\mathbb{D}(\frac{\D_{X,\x}[S] F^{S}}{(S-A)\D_{X,\x}[S] F^{S}})$ and $\mathbb{D}(\frac{\D_{X,\x}[S] F^{S}}{(S-(A-1))\D_{X,\x}[S] F^{S}})$ can be identified with modules (i.e. Ext vanishes in all but one place). $\nabla_{A}$ lifts to a map between the resolutions of $\frac{\D_{X,\x}[S] F^{S}}{(S-A)\D_{X,\x}[S] F^{S}}$ and $\frac{\D_{X,\x}[S] F^{S}}{(S-(A-1))\D_{X,\x}[S] F^{S}}$ and to the Hom of those resolutions. Therefore $\nabla_{A}$ induces a map (thinking of these as modules)
$$\mathbb{D}(\frac{\D_{X,\x}[S] F^{S}}{(S-A)\D_{X,\x}[S] F^{S}}) \longrightarrow \mathbb{D}(\frac{\D_{X,\x}[S] F^{S}}{(S-(A-1))\D_{X,\x}[S] F^{S}}).$$
Name this map $\mathbb{D}(\nabla_{A})$.

\begin{theorem} \label{thm D(nabla)}
Suppose $f = f_{1} \cdots f_{r}$ is reduced, strongly Euler-homogeneous, Saito-holonomic, and free and let $F=(f_{1}, \dots, f_{r}).$ Let $\chi_{A}$ be the $\D_{X,\x}$-isomorphism of Theorem \ref{thm nabla dual}. Then there is a commutative diagram
\[
\begin{tikzcd}
    \mathbb{D} \left(\frac{\D_{X,\x}[S]F^{S}}{(S-A)\D_{X,\x}[S]} \right)  \arrow{r}{\simeq}[swap]{\chi_{A}}
            & \frac{\D_{X,\x}[S]F^{S}}{(S-(-A-1))\D_{X,\x}[S]F^{S}} \\
    \mathbb{D} \left(\frac{\D_{X,\x}[S]F^{S}}{(S-(A-1))\D_{X,\x}[S]} \right) \arrow{r}{\simeq}[swap]{\chi_{A-1}} \arrow[u, "\mathbb{D}(\nabla_{A})"]
        &\frac{\D_{X,\x}[S]F^{S}}{(S-(-A))\D_{X,\x}[S]F^{S}} \arrow[u, "\nabla_{-A}"].
\end{tikzcd}
\]
\end{theorem}

\begin{proof}

First consider the $\D_{X,\x}$-linear map $\nabla: \D_{X,\x}[S] F^{S} \to \D_{X,\x}[S] F^{S}$ given by sending $s_{i} \to s_{i+1}$ for all $i$. By Propositon \ref{prop many sp complexes}, the co-complex of free $\D_{X,\x}[S]$-modules $\Sp_{\theta_{F,\x}, \Der_{X,\x} \oplus \mathscr{O}_{X,\x}^{r}}$ resolves $\D_{X,\x}[S] F^{S}$. For readability, in this proof we will write this co-complex as $\Sp.$ Regarding this as a co-complex of $\D_{X,\x}$-modules, we may lift $\nabla$ to a chain map. A straightforward computation using \eqref{eqn- cartan complex} and the definition of $\psi_{F,\x}$ shows that one such lift is given by 
\[
\begin{tikzcd}
    \SP^{-k} \arrow{r}{\simeq} \arrow[dashed]{d}
        & \D_{X,\x}[S]^{\binom{n}{k}} \arrow{d}{\sigma_{-k}} \\
    \SP^{-k} \arrow{r}{\simeq} &
        \D_{X,\x}[S]^{\binom{n}{k}},
\end{tikzcd}
\]
where the dashed line is the lift of $\nabla$ at the $-k$ slot and $\sigma_{-k}$ multiplies each component of the direct sum by $f$ on the right and sends each $s_{i}$ to $s_{i+1}$ in every component. 

\iffalse By Proposition \ref{prop the LR resolution}, \fi We may use the finite, free $\D_{X,\x}$-resolution of $\D_{X,\x}[S]F^{S} / (S-A) \D_{X,\x}[S] F^{S}$ by $\frac{\D_{X,\x}[S]}{(S-A)\D_{X,\x}[S]} \otimes_{\D_{X,\x}[S]} \Sp$ to lift $\nabla_{A}$ to a chain map, cf. Proposition \ref{prop the LR resolution}. One such lift is given by 

\begin{equation} \label{diagram lift}
\begin{tikzcd}
    \frac{\D_{X,\x}[S]}{(S-A)\D_{X,\x}[S]} \otimes_{\D_{X,\x}[S]} \SP^{-k}  \arrow[dashed]{d}{\ell_{-k}(\nabla_{A})} \arrow{r}{\simeq}
        & \frac{\D_{X,\x}[S]}{(S-A)\D_{X,\x}[S]}^{\binom{n}{k}} \arrow{d}{\sigma_{-k}^{A}} \\
    \frac{\D_{X,\x}[S]}{(S-(A-1))\D_{X,\x}[S]} \otimes_{\D_{X,\x}[S]} \SP^{-k} \arrow{r}{\simeq}
        & \frac{\D_{X,\x}[S](S-(A-1))}{(S-(A-1))\D_{X,\x}[S]}^{\binom{n}{k}},
\end{tikzcd}
\end{equation}
where the $\ell_{-k}(\nabla_{A})$ is the lift of $\nabla_{A}$ the $-k$ slot and $\sigma_{-k}^{A}$ is induced by $\sigma_{-k}$. That is, $\sigma_{-k}^{A}$ is given by multiplying each component of the direct sum by $f$ on the right.

Apply $\Hom_{\D_{X,\x}}(-, \D_{X,\x})^{\text{left}}$ to the chain map given by the $l_{-k}(\nabla_{A})$. Then \eqref{diagram lift} implies that at the $-n$ slot we have

\begin{equation} \label{diagram hom lift}
\begin{tikzcd}
    \Hom_{\D_{X,\x}}(\frac{\D_{X,\x}[S]}{(S-A)\D_{X,\x}[S]} \otimes_{\D_{X,\x}[S]} \SP^{-n}, \D_{X,\x})^{\text{left}}  
        & \frac{\D_{X,\x}[S]}{(S-A)\D_{X,\x}[S]}  \arrow{l}{\simeq} \\
    \Hom_{\D_{X,\x}}(\frac{\D_{X,\x}[S]}{(S-(A-1))\D_{X,\x}[S]} \otimes_{\D_{X,\x}[S]} \SP^{-n}, \arrow[dashed]{u}{\Hom_{\D_{X,\x}}(\ell_{-n}(\nabla_{A}), \D_{X,\x})^{\text{left}}} \D_{X,\x})^{\text{left}} 
        & \frac{\D_{X,\x}[S]}{(S-(A-1)\D_{X,\x}[S]} \arrow{u}{\Hom_{\D_{X,\x}}(\sigma_{-n}^{A}, \D_{X,\x})^{\text{left}}} \arrow{l}{\simeq},
\end{tikzcd}
\end{equation}
where $\Hom_{\D_{X,\x}}(\sigma_{-n}^{A}, \D_{X,\x})^{\text{left}}$ is simply multiplication by $f$ on the right. Since $\mathbb{D}(\D_{X,\x}[S]F^{S}/ (S-A)\D_{X,\x}[S]F^{S})$ has nonzero homology only at the $-n$ slot, we may identify this complex with that homology module and the map $\mathbb{D}(\nabla_{A})$ is induced by $\Hom_{\D_{X,\x}}(\sigma_{-n}^{A}, \D_{X,\x})^{\text{left}}$, i.e. by multiplication by $f$ on the right. So \eqref{diagram hom lift} and Theorem \ref{thm nabla dual} give the following commutative diagram

\[
\begin{tikzcd}
    \frac{\D_{X,\x}[S]}{(S-A)\D_{X,\x}[S]} \arrow[two heads]{r} 
        & \mathbb{D}(\frac{\D_{X,\x}[S]F^{S}}{(S-A)\D_{X,\x}[S]F^{S}}) \arrow{r}{\simeq}[swap]{\chi_{A}}
            & \frac{\D_{X,\x}[S]F^{S}}{(S-(-A-1))\D_{X,\x}[S]F^{S}} \\
    \frac{\D_{X,\x}[S]}{(S-(A-1))\D_{X,\x}[S]} \arrow{u}{\cdot f} \arrow[two heads]{r} 
        & \mathbb{D}(\frac{\D_{X,\x}[S]F^{S}}{(S-(A-1))\D_{X,\x}[S]F^{S}}) \arrow{r}{\simeq}[swap]{\chi_{A-1}} \arrow{u}{\mathbb{D}(\nabla_{A})} 
            & \frac{\D_{X,\x}[S]F^{S}}{(S-(-A))\D_{X,\x}[S]F^{S}} \arrow[dashed]{u}.
\end{tikzcd}
\]
A straightforward diagram chase shows that the dashed map is $\nabla_{-A}$.
\end{proof}

\begin{theorem} \label{thm injective iff surjective}
Suppose $f = f_{1} \cdots f_{r}$ is reduced, strongly Euler-homogeneous, Saito-holonomic, and free and let $F=(f_{1}, \dots, f_{r}).$ Then $\nabla_{A}$ is injective if and only if it is surjective.
\end{theorem}
\begin{proof}
By Theorem \ref{thm injective implies surjective}, we may assume $\nabla_{A}$ is surjective. So we have a short exact sequence of holonomic left $\D_{X,\x}$-modules:
\[
0 \to N \to \frac{\D_{X,\x}[S]F^{S}}{(S-A)\D_{X,\x}[S]F^{S}} \xrightarrow{\nabla_{A}} \frac{\D_{X,\x}[S]F^{S}}{(S-(A-1))\D_{X,\x}[S]F^{S}} \to 0.
\]
Using the long exact sequence of $\text{Ext}$ and basic properties of holonomic modules, one checks $\nabla_{A}$ is surjective if and only if $\mathbb{D}(\nabla_{A})$ is injective. Similarly, $\nabla_{A}$ is injective if and only if $\mathbb{D}(\nabla_{A})$ is surjective. We are done by the following: 
\begin{align*}
    \nabla_{A} \text{ is surjective} &\iff \mathbb{D}(\nabla_{A}) \text{ is injective} & [\text{Duality}] \\
    &\iff \nabla_{-A} \text{ is injective} &[\text{Theorem } \ref{thm D(nabla)}]\\
    & \implies \nabla_{-A} \text{ is surjective} & [\text{Theorem } \ref{thm injective implies surjective}]\\
    & \iff \mathbb{D}(\nabla_{A}) \text{ is surjective} &[\text{Theorem } \ref{thm D(nabla)}] \\
    & \iff \nabla_{A} \text{ is injective} & [\text{Duality}].
\end{align*}
\end{proof}
\iffalse
Since $\mathbb{D}(\nabla_{A})$ is injective, Theorem \ref{thm D(nabla)} implies  $\nabla_{-A}: \frac{\D_{X,\x}[S]F^{S}}{(S-(-A))\D_{X,\x}[S]F^{S}} \to \frac{\D_{X,\x}[S]F^{S}}{(S-(-A-1))\D_{X,\x}[S]F^{S}}$ is injective. So here Theorem \ref{thm injective implies surjective} shows that $\mathbb{D}(\nabla_{A})$ is an isomorphism. So $\text{Ext}_{\D_{X,\x}}^{\bullet}(N, \D_{X,\x}) = 0$.

Let $\mathbb{D}(N)$ be the holonomic dual of $N$ (i.e. $\text{Ext}_{\D_{X,\x}}^{n}(N, \D_{X,\x})$). It is well known that $N \simeq \mathbb{D}(\mathbb{D}(N))$. The above shows that $\mathbb{D}(N) = 0$. So $Q=0$. Therefore 
$$ \frac{\D_{X,\x}[S]F^{S}}{(S-A)\D_{X,\x}[S]F^{S}} \xrightarrow{\nabla_{A}} \frac{\D_{X,\x}[S]F^{S}}{(S-(A-1))\D_{X,\x}[S]F^{S}} $$
is injective as desired.
\fi

\iffalse 
In Section 5 we will use the above theorem to relate cohomology support loci of $f$ near $\x$ to the Bernstein--Sato variety $\V (B_{F,\x})$ when $f$ is reduced, strongly Euler-homogeneous, Saito-holonomic, and free.
\fi

\section{Free Divisors and the Cohomology Support Loci} \label{final section}

In this short section, we assume $f_{1}, \dots, f_{r}$ are mutually distinct and irreducible germs at $\x \in X$ that vanish at $\x$. Let $f = f_{1} \cdots f_{r}$. Take a small open ball $B_{\x}$ about $\x$ and let $U_{\x} = B_{\x} \setminus \text{Var}(f).$ Define $U_{\y}$ for $\y \in \text{Var}(f)$ and $\y$ near $\x$ similarly. 

\begin{definition} (\normalfont{Compare with Section 1, \cite{BudurWangEtAl}})
\iffalse
Let $M(U)$ denote the rank one local systems on $U$. The \emph{cohomology support locus} $V(U_{\y})$ of $U_{\y}$ is the set of rank one local systems $L$ on $U_{\y}$ with nonvanishing cohomology. That is
\[
V(U_{\y}) := \{ L \in M(U_{\y}) \mid \text{dim } H^{\bullet}(U_{\y}, L) \neq 0 \}.
\]
\fi
Let $M(U)$ denote the rank one local systems on $U$. Define the \textit{cohomology support loci of $f$ near $\x$}, denoted as $V(U_{\x}, B_{\x})$, by: 
\[
V(U_{\x}, B_{\x}) := \bigcup\limits_{\y \in D \text{ near } \x} \text{res}_{\y}^{-1} (\{ L \in M(U_{\y}) \mid H^{\bullet}(U_{\y}, L) \neq 0 \}),
\]
where $\text{res}_{\y}: M(U_{\x}) \to M(U_{\y})$ is given by restriction. This agrees with the notion of ``uniform cohomology support locus" given in \cite{BudurBernstein--Sato}, cf. Remark 2.8 \cite{BudurWangEtAl} and \cite{CharacteristivcVaritiesHypersurface}.
\end{definition}
\iffalse
There is a natural map of local systems on $U_{\x}$ to $U_{\y}$, $\y$ near $\x$, given by restriction:
\[
\text{res}_{\y} : \{\text{ local systems on }U_{\x} \}\to \{\text{ local systems on } U_{\y}\}.
\]
Define the \textit{cohomology support loci of $f$ near $\x$}, denoted as $V(U_{\x}, B_{\x})$, by:
\[
V(U_{\x}, B_{\x}) := \bigcup\limits_{\y \in D \text{ near } \x} \text{res}_{\y}^{-1} (\text{rank one local systems on } U_{\y} \text{ with nonvanishing cohomology}).
\]
It is often easier to consider the data about all the cohomology support loci of $f$ at $\y$ for $\y$ near $\x$ at once. With this in mind, let \emph{the cohomology support loci of $f$ near $\x$}, denoted by $V(U_{\x}, B_{\x})$, be defined as
\[
V(U_{\x}, B_{\x}) := \bigcup\limits_{\y \in D \text{ near } \x} \text{res}_{\y}^{-1} (V(U_{\y})).
\]
Note that this definition agrees with the notion of ``uniform cohomology support locus" as defined in \cite{BudurBernstein--Sato}, cf. Remark 2.8 \cite{BudurWangEtAl} and \cite{CharacteristivcVaritiesHypersurface}.
\end{definition}
\fi

\begin{Convention}
For $A \in \mathbb{C}^{r}$ and $k \in \mathbb{Z}$, let $A-k$ denote $(a_{1} -k, \dots, a_{r} - k).$
\end{Convention}

Let $j$ be the embedding of $U_{\x} \hookrightarrow B_{\x}$. For $L \in M(U_{\x})$, $Rj_{\star}(L[n])$ is a perverse sheaf (hence of finite length). \iffalse Given a rank one local system $L$ on $U_{\x}$, the derived direct image $Rj_{\star}(L[n])$ is a perverse sheaf on $B_{\x}$ and hence of finite length (in the category of perverse sheaves). \fi In Theorem 1.5 of \cite{BudurWangEtAl}, the authors prove that
\begin{equation} \label{eqn simple local system}
    V(U_{\x}, B_{\x}) = \{ L \in M(U_{\x}) \mid R j_{\star}(L[n]) \text{ is not a simple perverse sheaf on } B_{\x} \}.
\end{equation}

Using this Budur proves in Theorem 1.5 of \cite{BudurBernstein--Sato}, cf. Remark 4.2 of \cite{BudurWangEtAl}, that
\begin{equation} \label{eqn main budur conjecture}
    \text{Exp}(\V(B_{F,\x})) \supseteq V(U_{\x}, B_{\x}).
\end{equation}
Here $M(U_{\x})$ are identified with representations $\{ \pi_{1}(U_{\x}) \to \mathbb{C}^{\star} \} \subseteq \mathbb{C}^{\star^{r}}$.

While we cannot prove the converse containment to \eqref{eqn main budur conjecture}, we can prove a weaker statement about simplicity of modules:

\begin{theorem} \label{thm my budur converse}

Suppose $f = f_{1} \cdots f_{r}$ and $F = (f_{1}, \dots, f_{r})$, where the $f_{k}$ are mutually distinct and irreducible hypersurface germs at $\x$ vanishing at $\x$. Suppose $f$ is reduced, strongly Euler-homogeneous, Saito-holonomic, and free. If $A \in \mathbb{C}^{r}$ such that the rank one local system $L_{\text{\normalfont Exp}(A)} \notin V(U_{\x}, B_{\x})$, then, for all $k \in \mathbb{Z}$, the map $\nabla_{A+k}$ is an isomorphism and $\frac{\D_{X,\x}[S]F^{S}}{(S-(A+k))\D_{X,\x}[S]F^{S}}$ is a simple $\D_{X,\x}$-module.

\end{theorem}

\begin{proof}
For all $A \in \mathbb{C}^{r}$ there is a cyclic $\D_{X,\x}$-module $\D_{X,\x}F^{A}$ defined similarly to $\D_{X,\x}[S]F^{S}.$ Moreover, there is a commutative diagram of $\D_{X,\x}$-modules and maps
\begin{equation} \label{diagram commutative diagram}
\begin{tikzcd}
    \frac{\D_{X}[S]F^{S}}{(S-(A+k))\D_{X}[S]F^{S}} \arrow[r, twoheadrightarrow, "p_{A+k}"] \arrow[d, "\nabla_{A+k}"] 
        & \D_{X,\x}F^{A+k} \arrow[d, hookrightarrow] \\
    \frac{\D_{X}[S]F^{S}}{(S-(A+k-1))\D_{X}[S]F^{S}} \arrow[r, twoheadrightarrow, "p_{A+k-1}"]
        & \D_{X,\x} F^{A+k-1}.
\end{tikzcd}
\end{equation}
By Theorem 5.2 in \cite{BudurBernstein--Sato}, $\D_{X,\x}F^{A}$ is regular, holonomic and
\begin{align*} \label{eqn deRham}
DR(\D_{X,\x}F^{A-k}) = Rj_{\star}L_{\text{Exp}(A)}[n] \text{, for } k \in \mathbb{N} \text{, } k \gg 0; \\
DR(\D_{X,\x}F^{A+k}) = IC(L_{\text{Exp}(A)}[n]) \text{, for } k \in \mathbb{N} \text{, } k \gg 0. \nonumber
\end{align*}
Here $DR$ is the de Rham functor and $L_{\text{Exp}(A)}$ is the local system given by a representation $\pi_{1}(U_{\x}) \to \mathbb{C}^{\star^{r}}.$ Because of \eqref{eqn simple local system}, our hypotheses on $L_{\text{Exp}(A)}$ imply $\D_{X,\x}F^{A+k}$ is simple for all $k \in \mathbb{Z}$. So to prove the theorem, it suffices to show that the $\D_{X,\x}$-maps $p_{A+k}$ and $\nabla_{A+k}$ of \eqref{diagram commutative diagram} are isomorphisms for all $k \in \mathbb{Z}$.

By Proposition 3.2 and 3.3 of \cite{OakuTakayama} there exists an integer $t \in \mathbb{Z}$ such that $p_{A+t-1 - j}$ is an isomorphism for all $j \in \mathbb{Z}_{\geq 0}$. By the commutativity of \eqref{diagram commutative diagram}, $\nabla_{A+t}$ is surjective. By Theorem \ref{thm injective iff surjective}, $\nabla_{A+t}$ is an isomorphism. Thus $p_{A+t}$ is as well. Repeat this procedure to finish the proof.
\end{proof}

\appendix
\section{Initial Ideals}

Suppose the commutative Noetherian ring $R$ is a domain containing a field $\mathbb{K}$. Consider the polynomial ring over many variables $R[X]$, graded by the total degree of a non-negative integral vector $\textbf{u}$. Let $I$ be an ideal contained in $(X) \cdot R[X]$. We closely follow the treatment of Bruns and Conca in \cite{BrunsConca} to obtain our main result, Proposition \ref{main grobner theorem}, which establishes a relationship between the initial ideal $\In_{\textbf{u}}$ of $I$ with respect to the $\textbf{u}$-grading and $I$ itself. This is a weaker analogue to Proposition 3.1 of loc. cit. and is integral to the strategy of Section 2.

\iffalse

In the standard setting $R = \mathbb{K}$ and one uses Gr\"obner basis techniques with impunity (cf. \cite{EisenbudCommutative} Chapter 15, in particular Section 15). We use the same general strategy, but because $R$ is not even assumed to be a finitely generated $\mathbb{K}$-algebra we have to proceed with care.

The arguments and structure closely follow those of Bruns and Conca \cite{BrunsConca} but because $R \neq \mathbb{K}$ we have had to alter many things slightly. First, if $R=\mathbb{K}$, then for  $I \subseteq R[X]$ an ideal, the monomials not in $\In_{>}(I)$ form a basis for $R / I$ (see \cite{EisenbudCommutative} Theorem 15.3); this does not hold for $R \neq \mathbb{K}$. Second, because $R$ is not a finitely generated $\mathbb{K}$-algebra, the maximal ideals of any finitely generated $R$-algebra may have different heights. This weakens Proposition \ref{main grobner theorem} (b) considerably, (cf. Proposition 3.1 \cite{BrunsConca}). Finally, because $\textbf{u}$ may declare non-units to be of degree 0, $R[X]$ may not be graded local.
\fi

\begin{remark} \label{remark grading 1}
\begin{enumerate}[(a)]

\item The \emph{monomials} of $R[X]$ are the elements $x^{\textbf{v}} = \prod x_{i}^{v_{i}}$ for $\textbf{v}$ a non-negative integral vector. \iffalse Because the coefficient ring $R$ is larger than $\mathbb{K}$ it may be that, for $r_{1}$, $r_{2} \in R$, $r_{1}x^{\textbf{v}}$ and $r_{2}x^{\textbf{v}}$ generate different ideals. \fi

\item Just as in the case $R = \mathbb{K}$ we can declare a \emph{monomial ordering} $>$ on $R[X]$. \iffalse Because the ordering is inherited from when $R=\mathbb{K}$ this \fi This ordering is Artinian, with least element $1 \in R$.

\item Every $f \in R[X] $ has a unique expression in monomials: $f = \sum r_{i}m_{i}$, $r_{i} \in R$, $m_{i}$ a monomial, $m_{i} > m_{i +1}$, for some total ordering $>$ of the monomials.

\item Let the \emph{initial term} of $f$ be $\In_{>}(f) := r_{1}m_{1}$, where we appeal to the unique expression of $f$ above. For $V$ a     $R$-submodule of $R[X]$ let $\In_{>}(V)$ be the R-submodule generated by all the $\In_{>}(f)$ elements for $f \in V$. 

\item Given a nonnegative integral vector $\textbf{u} = (u_{1}, \dots u_{n})$ there is a canonical grading on $R[X]$ given by $\textbf{u}(x_{i}) = u_{i}.$ Every monomial $\prod x_{i}^{v_{i}}$ is $\textbf{u}$-homogeneous of degree $\sum v_{i} u_{i}$ and every element $f \in R[X]$ has a unique decomposition into $\textbf{u}$-homogeneous pieces. The \emph{degree} $\textbf{u}(f)$ is the largest degree of a monomial of f; the \emph{initial term} $\In_{\textbf{u}}(f)$ is the sum of the monomials of $f$ of largest degree.
\end{enumerate}
\end{remark}

\begin{definition} 

Let $f \in R[X]$, $f = \sum r_{i}m_{i}$ its monomial expression, $\textbf{u}$ a non-negative integral vector defining a grading on $R[X]$. We introduce a new variable t by letting T = R[X][t]. Define the \emph{homogenization} of f with respect to $\textbf{u}$ to be 
\[\homu(f) := \sum r_{i}m_{i} t^{\textbf{u}(f) - \textbf{u}(m_{i})} \in T.\]
For a $R$-submodule $V$ of $R[X]$ let 
$$\homu(V) := \text{ the $R[t]$-submodule generated by } \{\homu(f) \mid f \in V\}.$$
\end{definition}

\begin{remark}

\begin{enumerate}[(a)]

\item If $I$ is an ideal of $R[X]$, $\homu(I)$ is an ideal of T.

\item Let $\textbf{u}^{\prime}$ be the non-negative integral vector $(\textbf{u}, 1)$ and extend the grading on $R[X]$ to $T$ by declaring $t$ to have degree 1. Then $\homu(f)$ is a $\textbf{u}^{\prime}$-homogeneous of degree $\textbf{u}(f)$. 

\end{enumerate}
\end{remark}

\begin{prop} Suppose $R$ is a Noetherian domain containing the field $\mathbb{K}$ and let $I$ be an ideal of $R[X]$. Then

\[
    \frac{T}{(\homu(I), t)} \simeq \frac{R[X]}{\In_{\textbf{u}}(I)} \quad \text{and} \quad \frac{T}{(\homu(I), t-1)} \simeq \frac{R[X]}{I}.
\]

\end{prop}

\begin{proof}
Argue as in Proposition 2.4 of \cite{BrunsConca}.
\end{proof}

As in the classical case, $t$ is regular on $T / \homu(I)$ (cf. Proposition 2.3 (d) in \cite{BrunsConca}). The argument is similar so we only outline the basic steps.

\begin{definition}

Let $\tau$ be a monomial ordering on $R[X]$, $\textbf{u}$ a non-negative integral vector. Let $n$, $m$ be monomials in $R[X]$. Define a monomoial ordering $\tau \textbf{u}$ on R[X]:
\[
    [m >_{\tau \textbf{u}} n]   \iff   [\textbf{u}(m) > \textbf{u}(n)] \quad \text{ or } \quad [\textbf{u}(m) = \textbf{u}(n)\text{, } m >_{\tau} n].
\]
Define a monomial ordering $\tau \textbf{u}^{\prime}$ on T:
\begin{align*}
    [mt^{i} >_{\tau \textbf{u}^{\prime}} nt^{j}]   \iff 
        & [\textbf{u}^{\prime}(mt^{i}) > \textbf{u}^{\prime}(nt^{j})] \text{ or,}\\
        & [\textbf{u}^{\prime}(mt^{i}) = \textbf{u}^{\prime}(nt^{j}) \text{ and } i < j]     \text{ or,} \\
        & [\textbf{u}^{\prime}(mt^{i}) =   \textbf{u}^{\prime}(nt^{j}) \text{ and } i < j   \text{ and } m >_{\tau} n].
\end{align*}

\end{definition}

\begin{lemma} \label{lemma initial} \text{\normalfont (Compare with 2.3(c) in \cite{BrunsConca})} Suppose $R$ is a Noetherian domain containing the field $\mathbb{K}$. For $V$ a $R$-submodule of $R[X]$, 
$$\In_{\tau \textbf{u}}(V) R[t] = \In_{\tau \textbf{u}^{\prime}}(\homu(V)).$$

\begin{proof} Argue similarly to Proposition 2.3 (c) in \cite{BrunsConca}.
\end{proof}

\end{lemma}

\begin{prop} \label{prop: flat1} \text{ \normalfont (Compare with 2.3(d) in \cite{BrunsConca})} Suppose $R$ is a Noetherian domain containing the field $\mathbb{K}$. Let $I \subseteq X \cdot R[X]$ be an ideal of $R[X]$ and $\textbf{u}$ a nonnegative integral vector. \iffalse Suppose there exists a non-negative integral grading $b$ on $R[X]$ such that $R \subseteq \{f \in R[X] \mid b(f) = 0\}$ and $I \subseteq B_{+}$, the irrelevant ideal of the $b$-grading. \fi Then $T / \homu(I)$ is a torsion-free $\mathbb{K}[t]$ module. \iffalse In particular, $T / \homu(I)$ is a flat $K[t]$-module. \fi 
\end{prop}

\begin{proof} We give a sketch. Suppose $h\in T$, $s(t) \in \mathbb{K}[t]$ such that $s(t)h \in \homu(I)$. We must show that $h \in \homu(I)$. 

Because $\tau \textbf{u}^{\prime}$ is a monomial order, $\In_{\tau \textbf{u}^{\prime}}(s(t)h) = s_{k}t^{k} \In_{\tau \textbf{u}^{\prime}}(h)$, for $s_{k} \in K.$ By hypothesis and Lemma \ref{lemma initial}, $s_{k}t^{k} \In_{\tau \textbf{u}^{\prime}}(h) \in \In_{\tau \textbf{u}}(I)R[t].$ By comparing monomials and using the fact we can ``divide" an equation by $t$ if both sides are multiples of $t$, careful bookkeeping yields the following: there exists $g \in \homu(I)$ such that $h - g < h$ and $s(t)(h - g) \in \homu(I)$. Repeat the process to continually peel off initial terms and conclude either $h \in \homu(I)$ or there exists $0 \neq r \in R \cap \In_{\tau \textbf{u}}(I).$ Because $I \subseteq X \cdot R[X]$, we have $\In_{\tau \textbf{u}}(I) \subseteq X \cdot R[X].$ Therefore no such $r$ exists and the claim is proved.
\end{proof}

The following is the section's main proposition: \iffalse \todo{irrelevant ideal, degree 0 subring, fix references to Armendiaz paper} \todo{make sure nonnegative u grading is emphasized} \fi

\begin{prop} \label{main grobner theorem}\text{ \normalfont (Compare with 3.1 in \cite{BrunsConca})} Suppose $R$ is a Noetherian domain containing the field $\mathbb{K}$. Let $I \subseteq X \cdot R[X]$ be an ideal of $R[X]$ and $\textbf{u}$ a non-negative integral vector. \iffalse Assume the $\textbf{u}$-degree $0$ subring of $R/I$. \fi \iffalse For $B^{+}$ the irrelevant ideal with respect to the $\textbf{u}$-grading, assume $R[X] / B^{+}$ is connected. Suppose there exists a non-negative integral grading $b$ on $R[X]$ such that $R = \{f \in R[X] \mid b(f) = 0\}$ and $I \subseteq B_{+}$, the irrelevant ideal of the $b$-grading. Furthermore, assume $Spec(R)$ is connected. \fi Then the following hold:

\begin{enumerate}[(a)]

\item If $R[X] / \In_{\textbf{u}}(I)$ is Cohen--Macaulay, then $R[X]/ I$ is Cohen--Macaulay;
\item $\dim(R[X]/\In_{\textbf{u}}(I)) \geq \dim(R[X]/I)$.
\end{enumerate}
\end{prop}

\begin{proof}
(a). We follow the argument of Proposition 3.1 in \cite{BrunsConca}: first, we show that Cohen--Macaulayness percolates from $T/(\homu(I), t)$ to $T/\homu(I)$; second, that it descends from $T/\homu(I)$ to $T/(\homu(I), t-1)$. 

First, the percolation. Since $\textbf{u}^{\prime}(t) = 1$, any maximal $\textbf{u}^{\prime}$-graded ideal $\mathfrak{m}^{\star}$ of $ T /\homu(I)$ contains t. \iffalse (Note $t \neq 0$ in $T / \homu(I)$.) \fi Consider the commutative diagram

\[
\begin{tikzcd}
    & T / \homu(I) \arrow[r] \arrow[d]
        & T_{m} / (\homu(I))_{\mathfrak{m}^{\star}} \arrow[d] \\
    & T/ (\homu(I), t) \arrow[r]
        & T_{\mathfrak{m}^{\star}}/(\homu(I), t)_{\mathfrak{m}^{\star}},
\end{tikzcd}
\]
with horizontal maps localization at $\mathfrak{m}^{\star}$, vertical maps quotients by $t$.

It suffices to show that $T/\homu(I)$ is Cohen--Macaulay after localizing at a maximal $\textbf{u}^{\prime}$-graded ideal $\mathfrak{m}^{\star}$ (cf. Exercise 2.1.27 \cite{BrunsHerzogCMRings}). Since $t \in \mathfrak{m}^{\star}$, by assumption $T_{\mathfrak{m}^{\star}}/(\homu(I), t)_{\mathfrak{m}^{\star}}$ is Cohen--Macaulay. And since $t$ is a non-zero divisor on $T_{\mathfrak{m}^{\star}}/\homu(I)_{\mathfrak{m}^{\star}}$ by Proposition \ref{prop: flat1}, we see $T_{\mathfrak{m}^{\star}}/\homu(I)_{\mathfrak{m}^{\star}}$ is Cohen--Macaulay (cf. Theorem 2.1.3 in \cite{BrunsHerzogCMRings}).

It remains to show that Cohen--Macaulayness descends from $T/ \homu(I)$ to $T/ (\homu(I), t-1)$. By the universal property of localization we have:

\begin{equation} \label{commutative diagram localization}
\begin{tikzcd}
    & T / \homu(I) \arrow[r] \arrow[dd] 
        & T/(\homu(I), t-1) \\
                                \\
    & (T / \homu(I))[t^{-1}] \arrow[uur, "\gamma"].
\end{tikzcd}
\end{equation}

It is well known (cf. Proposition 1.5.18 in \cite{BrunsHerzogCMRings}) that 
\[
(T / \homu(I))[t^{-1}] \simeq ((T / \homu(I))[t^{-1}])_{0}[y, y^{-1}].
\]
So $\gamma$ of \eqref{commutative diagram localization} induces, where $-_{0}$ denotes the degree $0$ elements, the ring maps:
\[
T / (\homu, t-1) \simeq \frac{(T / \homu(I))[t^{-1}]}{(t-1)(T/\homu(I))[t^{-1}])} \simeq ((T / \homu(I))[t^{-1}])_{0}.
\]
We have 
\begin{equation} \label{eqn section 2 CM iso}
(T/\homu(I))[t^{-1}] \simeq (T/(\homu(I), t-1))[y, y^{-1}].
\end{equation}

Therefore, since Cohen--Macualayness is preserved under localization at a non-zero divisor, all we need to show is that if $B[y, y^{-1}]$ is a Laurent polynomial ring that is Cohen--Macaulay then $B$ is an Cohen--Macaulay. To see this take a $\mathfrak{m} \in \text{mSpec}(B)$ and look at $(\mathfrak{m}, y-1) \in \Spec(B[y])$ and the corresponding prime ideal in $B[y, y^{-1}]$.

Now we move onto (b). The descent part of part (a) gives us the plan:
\begin{align*}
    \dim(T / (\homu(I),t-1)) &= \dim((T/(\homu(I), t-1))[y, y^{-1}]) -1 \\ 
        &= \dim((T/\homu(I))[t^{-1}]) - 1 \\
        &\leq \dim(T/\homu(I))-1 \\
        &= \dim(T/(\homu(I), t)).
\end{align*}

The second equality follows by \eqref{eqn section 2 CM iso}. The inequality is not an equality because localization may lower dimension.
For the last equality use the fact dimension of a graded ring can be computed by looking only at the height of the graded maximal ideals (Corollary 13.7 \cite{EisenbudCommutative}). In $T/\homu(I)$, $t$ is contained in all graded maximal ideals; since it is a non-zero divisor, its associated primes are not minimal. 
\end{proof}

\begin{remark}\label{remark grobner justification}
\begin{enumerate}[(a)]

\iffalse
\item The big difference between our version of \ref{main grobner theorem} and 3.1 in \cite{BrunsConca} is that here there is only an inequality relating the fiber dimensions, not an equality. Bruns and Conca get the equality in the argument as follows: since $t$ is a non-zero divisor, it is contained in no minimal prime. Find such a $p$ such that $\dim(T/ \homu(I) = \dim((T/\homu(I))/p$ to assume we are working in a domain $B$. The question then becomes, does there exist a maximal ideal of maximal height not containing t? If so, localization at t preserves dimension. When $R=k$ it is a general fact that in all finitely generated $\mathbb{K}$-algebras, (so in B) maximal ideals have the same height, so we need only show that $t \notin$ the Jacobson radical of $B$. This is immediate since $1+t$ cannot be a unit in a graded domain. In general characterizing these maximal ideals of maximal height is trickier. 
\fi

\item This proposition generalizes the common geometric picture for $R=\mathbb{K}$. In this setting the map $\mathbb{K}[t] \to T/\homu(I)$ gives a flat family whose generic fiber is $R/I$ and whose special fiber is $R/\In_{\textbf{u}}(I)$. In our generality, it is easy to extend Proposition \ref{prop: flat1} and show that $\mathbb{K}[t] \hookrightarrow T / \homu(I)$ is a flat ring map whose special fiber is $R[X] / \In_{\textbf{u}}(I)$ and whose generic fiber is $R[X] / I.$

\item In Section 2 we study ideals $I \subseteq (Y,S) \cdot \mathscr{O}_{X,\x}[Y,S]$ where $\mathscr{O}_{X}$ is an analytic structure sheaf and the $\textbf{u}$-grading assigns $1$ to the $y$-terms and $0$ to the $s$-terms. Proposition \ref{main grobner theorem} applies with $R = \mathscr{O}_{X,\x}$. 

\end{enumerate}
\end{remark}

\section{List of Symbols}

\begin{itemize}
    \item $\overline{(-)}$ is the coset representative of $(-)$ in the appropriate quotient object.
    \item $\In_{\textbf{u}}$ is the initial ideal with respect to the $\textbf{u}$-grading;
    \item $F_{(0,1)}$ is the order filtration on $\D_{X,\x}$ and $F_{(0,1,1)}$ is the total order filtration on $\D_{X,\x}$;
    \item $\gr_{(0,1)} (\D_{X,\x})$ is the associated graded object of $\D_{X,\x}$ corresponding to the order filtration and $\gr_{(0,1,1)}(\D_{X,\x}[S])$ is the associated graded object corresponding to the total order filtration;
    \item $\theta_{F} = \ann_{\D_{X,\x}[S]} \bigcap F_{(0,1,1)}^{1}$ are the annihilating derivations of $F^{S}$;
    \item $\psi_{F} : \Der_{X}(-\log f) \to \theta_{F}$ is the map $\delta \mapsto \delta - \sum \frac{\delta \bullet f_{k}}{f_{k}} s_{k};$
    \item $L_{F} = \gr_{(0,1,1)}(\D_{X,\x}[S]) \cdot \gr_{(0,1,1)}(\psi_{F}(\Der_{X}(-\log_{0} f)$ is the generalized Liouville ideal and $\widetilde{L_{F}} = \gr_{(0,1,1)}(\D_{X,\x}[S]) \cdot \theta_{F}$;
    \iffalse
    \item $\phi_{F}: \gr_{(0,1,1)}(\D_{X}[S]) \to R(\Jac(f_{1}), \dots \Jac (f_{r}))$ is given by $y_{i} \mapsto \sum \frac{f}{f_{k}} (\partial_{x_{i}} \bullet f_{k}) s_{k}$ and $ s_{k} \mapsto f s_{k}$
    \fi
    \item $B_{f}$ and $\V (B_{f})$ are the classical Bernstein--Sato ideal and variety of $f$ whereas $B_{F}$ and $ \V (B_{F})$ are the multivariate Bernstein--Sato ideal and variety of $F$;
    \item $\nabla_{A} : \frac{\D_{X,\x}[S] F^{S}}{(S-A) \D_{X,\x}[S] F^{S}} \to \frac{\D_{X,\x}[S] F^{S}}{(S-(A-1)) \D_{X,\x}[S] F^{S}}$ is the $\D_{X,\x}$-map induced by sending each $s_{k}$ to $s_{k} + 1.$
    \item $\Sp_{\theta_{F,\x}, \Der_{X,\x} \oplus \mathscr{O}_{X,\x}^{r}}$ is a complex associated to the Lie--Rinehart algebras $\theta_{F,\x} \subseteq \Der_{X,\x} \oplus \mathscr{O}_{X,\x}^{r}$, cf. Definition \ref{def LR complex};
    \item $\mathbb{D}(-) = (\text{RHom}_{\D_{X,\x}}(-, \D_{X,\x})^{\text{left}}$, $\mathbb{D}_{S}(-) = (\text{RHom}_{\D_{X,\x}[S]}(-, \D_{X,\x}[S])^{\text{left}};$
    \item $\chi_{A}: \mathbb{D} \left( \frac{\D_{X,\x}[S] F^{S}}{(S-A)\D_{X,\x}[S] F^{S}} \right) \xrightarrow{\simeq} \frac{\D_{X,\x}[S] F^{S}}{(S-(-A-1)\D_{X,\x}[S]F^{S}}[n]$, for strongly Euler-homogeneous, Saito-holonomic, and free $f$, cf. Theorem \ref{thm nabla dual}.
\end{itemize}

\bibliographystyle{amsplain}

%    Insert the bibliography data here.

\end{document}